\documentclass[a4paper,12pt]{article}
\usepackage[utf8]{inputenc}
\usepackage[english]{babel}
\usepackage[ruled,section]{algorithm}
\usepackage{makeidx}
\usepackage[font=small]{caption}
\usepackage{amsmath}
\usepackage{amsthm}
\usepackage{amsfonts}
\usepackage{amssymb}
\usepackage{mathrsfs}
\usepackage{graphicx}
\usepackage{hyperref}
\usepackage[margin=25mm]{geometry}
\usepackage{enumerate}
\usepackage{indentfirst}
\usepackage{color}
\usepackage{authblk}
\makeindex
\newcommand{\real}{\mathbb{R}}
\newcommand{\n}{\mathbb{N}}

\newcommand{\rN}{ {\mathbb{R}^N} }

\newcommand{\C}{\mathcal{C}_1}

\providecommand{\norm}[1]{\left \| #1\right \|}

\numberwithin{equation}{section}
\newtheorem{theorem}{Theorem}[section]
\newtheorem{lem}[theorem]{Lemma}
\newtheorem{prop}[theorem]{Proposition}
\newtheorem{corol}[theorem]{Corollary}

\theoremstyle{definition}

\newtheorem{rmk}[theorem]{Remark}
\newtheorem{example}[theorem]{Example}

\newtheorem{defin}[theorem]{Definition}

\newtheorem{claim}[theorem]{Claim}

\usepackage{xcolor}

\begin{document}

\title{\bf\Large Principal spectral curves for Lane-Emden fully nonlinear type systems and applications}

\author[1]{Ederson Moreira dos Santos\footnote{ederson@icmc.usp.br; partially supported by CNPq grant 309006/2019-8.}}
\author[2]{Gabrielle Nornberg\footnote{gnornberg@dim.uchile.cl; supported by FAPESP grants 2018/04000-9 and 2019/031019-9, São Paulo Research Foundation.}}
\author[3]{Delia Schiera\footnote{delia.schiera@uniroma1.it; 
 supported by project Vain-Hopes within the program VALERE: VAnviteLli pEr la RicErca.}}
\author[4]{Hugo Tavares\footnote{hugo.n.tavares@tecnico.ulisboa.pt; partially supported by the Portuguese government through FCT-Funda\c c\~ao para a Ci\^encia e a Tecnologia, I.P., under the projects PTDC/MAT-PUR/28686/2017 and UID/MAT/04459/2020.}}
\affil[1]{\small Instituto de Ci\^encias Matem\'aticas e de Computa\c c\~ao, Universidade de S\~ao Paulo,
	
 Av.\ Trabalhador São-carlense 400, 13566-590 Centro, São Carlos SP, Brazil}
\affil[2]{\small Departamento de Ingeniería Matemática, Universidad de Chile,	
	
	Beauchef 851, Torre Norte, Santiago, Chile}
\affil[3]{\small Dipartimento di Matematica e Fisica, Universit\`a degli Studi della Campania ``L. Vanvitelli'', \, Viale A. Lincoln 5, 81100 Caserta, Italy}
\affil[4]{\small CAMGSD and Mathematics Department, Instituto Superior T\'ecnico, Universidade de Lisboa, Av.\  Rovisco Pais, 1049-001 Lisboa, Portugal}

\date{\today}

\maketitle

{\small\noindent{\bf{Abstract.}} 
In this paper we exploit the phenomenon of two principal half eigenvalues in the context of fully nonlinear Lane-Emden type systems  with possibly unbounded coefficients and weights. We show that this gives rise to the existence of two principal spectral curves on the plane. 
We also construct a possible third spectral curve related to a second eigenvalue and an anti-maximum principle, which are novelties even for Lane-Emden systems involving linear operators.
As applications, we derive a maximum principle in small domains for these systems, as well as existence and uniqueness of positive solutions in the sublinear regime. 
Most of our results are new even in the scalar case, in particular for a class of Isaac's operators with unbounded coefficients, whose $W^{2,\varrho}$ regularity estimates we also prove.
\medskip

{\small\noindent{\bf{Keywords.}} {First eigenvalue; Lane-Emden elliptic systems; Fully nonlinear operators; Solvability.}

\medskip


\tableofcontents
\section{Introduction and main results}\label{Introduction}

In this paper we study existence and uniqueness properties of the Dirichlet problem for partial differential fully nonlinear systems of Lane-Emden nature with weights, such as 
\begin{align}\label{Dir lambda mu}
\left\{
\begin{array}{rclcc}
F_1 (x,u,Du, D^2 u)+ \lambda\tau_1 (x) |v|^{q-1}v &=&f_1(x) &\mbox{in} & \;\Omega \\
F_2 (x,v,Dv,D^2 v)+ \mu \tau_2(x) |u|^{p-1}u &=& f_2(x) &\mbox{in} & \;\Omega \\
u\,\;= \,\;v\;&=& 0 &\mbox{on} & \;\partial\Omega ,
\end{array}
\right.
\end{align}
in the viscosity sense. Here $\Omega$ is a $C^{1,1}$ bounded domain in $\rN$ with $N\ge 1$,  $\lambda, \mu \in \real$ and $p,q>0$ are constants, $F_i$ is a uniformly elliptic fully nonlinear operator in nondivergence form,
 $f_i\in L^\varrho (\Omega)$ for some $\varrho > N$, $i=1,2$, and the respective weights satisfy
\begin{align}\label{H weights}
\textrm{$\tau_i \in L^\varrho(\Omega)$\; with \; $\tau_i\gneqq 0$ in $\Omega$, \;$i=1,2$, \quad } |\textrm{supp}\tau_1\cap\textrm{supp}\tau_2|>0.
\end{align}
Here $\tau_i \gneqq 0$ means that $\tau_i \geq 0$ a.e.\ in $\Omega$ and $\tau_i \not\equiv 0$. 
When $p=q=1$, $\tau_1=\tau_2=:\tau$, $F_1=F_2=:F$ and $f_1=f_2=:f$, we recover the scalar case
\begin{align}\label{F=f scalar intro}
F(x,u,Du,D^2u)+\lambda \tau(x)u =f(x)\;\; \textrm{ in }\, \Omega,\quad u=0\quad \textrm{ on }\, \partial \Omega,
\end{align} 
for which we also present new results. 

\smallbreak

Spectral properties of uniformly elliptic PDEs in nonvariational form have long been recognized since the seminal work \cite{BNV}. Its fully nonlinear scalar theory in terms of viscosity solutions was developed in \cite{BQeq}, for convex operators with bounded coefficients, and unveils the phenomenon of two half eigenvalues corresponding to both positive and negative eigenfunctions. The case of nonconvex operators (again with bounded coefficients) was analyzed in \cite{Arms2009} under additional continuity  restrictions on the data and on the operators. 

In general, problems involving systems may be much more involved than their scalar counterpart, specially in the strongly coupled case -- for instance we mention the so called Lane-Emden conjecture, see \cite{LE3d, LE4d}, a long standing open problem for which only partial results are known.  As far as spectral properties are concerned, in \cite{BQsystems} the authors extended their article \cite{BQeq} to gradient-like systems. 
Our systems, instead, have a strongly coupled nature, whose prototype is also called Hamiltonian. 
Spectral properties for related cooperative systems with linear operators in nondivergence form $F_i=L_i$ have been extensively investigated, see \cite{BMS99} for $p=q=1$ and references therein. 
When more general power-like nonlinearities are taken into account, still for lineal operators,
in \cite{Montenegro}  a spectral curve was constructed for \eqref{LE} when $pq=1$; more recently, related comparison principles appear in \cite{LeiteMontenegro}. Both \cite{Montenegro} and \cite{LeiteMontenegro} deal with linear operators with bounded drift and unbounded weights.

Our main goal here is to understand the phenomenon of two principal half eigenvalues induced by the fully nonlinear operators $F_1, F_2$ in light of \cite{BQeq}, under the framework of Lane-Emden systems in the regime $pq=1$, including nonconvex operators with possibly unbounded coefficients and weights. In this sense, we show that the homogeneous version of \eqref{Dir lambda mu}, i.e.\ 
\begin{align}\label{LE} \tag{LE}
\left\{
\begin{array}{rclcc}
F_1 (x,u,Du, D^2 u)+ \lambda\tau_1 (x) |v|^{q-1}v &=&0 &\mbox{in} & \;\Omega \\
F_2 (x,v,Dv,D^2 v)+ \mu \tau_2(x) |u|^{p-1}u &=& 0 &\mbox{in} & \;\Omega \\
u\,\;= \,\;v\;&=& 0 &\mbox{on} & \;\partial\Omega ,
\end{array}
\right.
\end{align}
gives rise to the existence of two principal spectral curves to \eqref{LE} in the plane $(\lambda, \mu)$.
We stress that principal eigenvalues are related to the solvability of \eqref{Dir lambda mu}, and to the validity of maximum principles, which we also study. Moreover, we construct a possible third spectral curve and an anti-maximum principle, which are novelties even for Lane-Emden systems involving linear operators. 
All our results are valid also for a class of Isaac's operators with unbounded coefficients \eqref{Isaacs} (see Example \ref{example} ahead), and therefore are new and improve results in the literature even in the scalar case. In this context, we mention that in \cite{regularidade} it was started a spectral analysis involving a class of proper operators with unbounded drift and weight in the scalar case, but only in what concerns existence of eigenvalues. Here we complement that study, by giving a full characterization of the first scalar eigenvalues in terms of validity of maximum principles, solvability of the Dirichlet problem, and more generally the validity of Alexandrov-Bakelman-Pucci (ABP) inequality for nonproper and possible nonconvex operators, under improved assumptions. Observe that, once ABP is proved, for any solution $(u,v)$ we have $uv>0$ in $\Omega$ whenever $u$ is signed in $\Omega$.

This problem brings about several applications. For instance, one may view the pair $(\lambda,\mu)$ as risk-sensitive averages of the weights $\tau_1$ and $\tau_2$, respectively, over the diffusions $F_1, F_2$, see \cite{controlJMPA2019, BFQsiam2010}. Besides, it characterizes the range of solvability for equations with superlinear gradient growth, as well as existence and uniqueness of positive solutions for \eqref{LE} in the sublinear regime $pq<1$, which we also prove.

For the Laplacian operator, the study of the problem with $pq=1$ involves basic and important questions in the theory of Harmonic Analysis. As a matter of fact, it is known that the standard Fourier series of an $L^r(0,1)$-function $f$ converges to $f$ in $L^r(0,1)$, for any $1 < r < \infty$.  This information was essential to treat the problem 
\[
-u''= \lambda |v|^{q-1}v \ \ \text{and} - v''= \mu |u|^{\frac{1}{q}-1}u \ \ \text{in $(0,1)$} \ \ \text{with} \ \ u(0)=v(0)=u(1)=v(1)=0,
\]
see \cite{DenisEderson2010}, where the asymptotic growth of the eigenvalues for the general case $pq=1$ was controlled through the eigenfunctions for $p=q=1$. However, the same question is much more challenging in higher dimensions, see \cite{CordobaAJM11977, CordobaAM1977}, and indeed it is false in general since the ``ball summation'' for the double Fourier series does not work; see \cite[Section 3.3 and Theorem 3.5.6]{Krantz}. In this case, when $\Omega \subset \mathbb{R}^2$ is a square, the Fourier functions (product of sines) do not form a Schauder basis in $L^r(\Omega)$, for $r \neq 2$. For systems with nondivergence operators we cannot expect such explicit formulas for eigenfunctions, and the problem is by far more delicate.

\subsection{Assumptions on the operators}

Next we list our hypotheses on the operators $F_1$ and $F_2$. We denote by $\mathbb{S}^N$ the space of $N\times N$--symmetric matrices. Let us bear in mind the following general structural hypothesis on a fully nonlinear operator $F:\Omega\times \mathbb{R}\times \mathbb{R}^N\times \mathbb{S}^N\to \mathbb{R}$ given by
\begin{align}\label{SC}\tag{H1}
\mathcal{L}^-(x,r-s,\xi-\eta, X-Y)  \le  F(x,r, \xi, X) - F(x,s, \eta, Y)\le \mathcal{L}^+(x,r-s,\xi-\eta, X-Y),
\end{align}
for all $X, Y  \in \mathbb{S}^N$, $\eta,\xi \in \mathbb{R}^N$, $r,s \in \mathbb{R}$, and $x  \in \Omega$, where $F(\cdot,0,0,0)\equiv 0$, and
 \begin{align}\label{def Lpm rho}
\mathcal{L}^\pm (x,r,\xi,X):= \mathcal{L}_0^\pm (x,\xi,X) \pm \vartheta (x) |r|, \;\; \textrm{ for }\; \mathcal{L}^\pm_0 (x,\xi,X):=\mathcal{M}^\pm (X)\pm \gamma (x)|\xi|,
\end{align}
for $\gamma,\vartheta\in L^\varrho (\Omega)$, $\varrho>N$, with $\gamma\ge 0$ and $\vartheta \gneqq 0$ a.e.\ in $\Omega$. Also, $\mathcal{M}^\pm=\mathcal{M}^\pm_{\alpha,\beta}$ are the Pucci's extremal operators with ellipticity constants $0<\alpha\leq \beta$, see \eqref{def Pucci} ahead.

Note that \eqref{SC} corresponds to a uniform bound for all operators satisfying a prescribed ellipticity. In order to measure how far a particular fully nonlinear operator $F$ is from a linear one, in the spirit of \cite{Arms2009, IY, BQeq}, we may construct from \eqref{SC} a more accurate structure: 
\[
F_*(x,r-s,\xi-\eta,X-Y) \le  F(x,r, \xi, X) - F(x,s, \eta, Y) \le F^*(x,r-s,\xi-\eta,X-Y)  ,
\]
where 
\vspace{-0.3cm}
\begin{align*}
F^*(x,r,\xi,X):=\sup_{r^\prime, \xi^\prime, X^\prime}\, \{ F(x,r+r^\prime, \xi+\xi^\prime,X+X^\prime) - F(x,r^\prime,\xi^\prime,X^\prime)\},
\end{align*}
\vspace{-0.7cm}
\begin{align*}
F_*(x,r,\xi,X):=\inf_{r^\prime, \xi^\prime, X^\prime} \,\{ F(x,r+r^\prime, \xi+\xi^\prime,X+X^\prime) - F(x,r^\prime,\xi^\prime,X^\prime)\},
\end{align*}
for all $x\in \Omega$, $r \in \mathbb{R}$, $\xi \in \mathbb{R}^N$, $X \in \mathbb{S}^N$.  Assume that $F$ satisfies  \eqref{SC}. Then both $F^*,F_*$ satisfy \eqref{SC} as well; $F^*$ is convex and $F_*$ is concave in $(r,\xi,X)$; $F=F^*$ if and only if $F$ is convex, $F=F_*$ if and only if $F$ is concave. Also, the following identity holds
\begin{center}
	$F^*(x,r,\xi,X)=-F_*(x,-r,-\xi,-X)$;
\end{center} see \cite[Proposition 4.2]{Arms2009}.
Finally, we have the ordering $ \mathcal{L}^- \le F_*\le F\le F^*\leq \mathcal{L}^+$.

 \begin{defin}\label{notation}
For a function $w$ and an operator $F$,  we consider the following notations:
\begin{enumerate}[(a)]
	\item \label{notation F, Liota}
		$F[w]:=F(x,w,Dw,D^2w)$;
	
	\item \label{notation lambda1 scalar} $\lambda_1^+(F(\vartheta))$ is the principal weighted eigenvalue associated to a positive eigenfunction of the scalar Dirichlet problem
	$ F[u] + \lambda \vartheta (x) u=0$ in $ \Omega$, $u=0$ on $\partial \Omega$, see Section \ref{Preliminaries_1};
	
	\item \label{notation viscosity} viscosity solutions are meant in the $L^N$-viscosity sense, see Section \ref{Preliminaries_2};

		\item \label{notation W2rho regularity} we say that $F$ enjoys $W^{2,\varrho}$ regularity if any viscosity solution $u$ of $F[u]=f(x)$ in $\Omega$ with $f\in L^\varrho(\Omega)$ belongs to $W^{2,\varrho}_{\text loc}(\Omega)$, and in addition $u\in W^{2,\varrho}(\Omega)$ if $u=\psi$ on $\partial\Omega$, $\psi\in W^{2,\varrho}(\Omega)$.
\end{enumerate}
\end{defin}

Having in mind the existence of eigenvalues, another condition we ask on an operator $F$ is that it satisfies a positive homogeneity of order one, namely
\begin{align}\label{H homogeneity}\tag{H2}
 F(x, t r, t\xi , tX)=t F(x, r, \xi, X) \; \text{ for all } t \ge 0,
 \textrm{ for any $X \in \mathbb{S}^N$, $\xi \in \mathbb{R}^N$, $r \in \mathbb{R}$, and $x \in \Omega$.}
\end{align}
Moreover, we consider the following control of oscillation in the $x$-entry:
\begin{align}\label{H continuity}
\forall \theta_0 >0,\quad \exists r_0>0: \quad \|\beta_F(x,\cdot)\|_{L^\varrho (B_r(x))}\le \theta_0\, r^{{N}/{\varrho}} \;\textrm{ for all } r\le r_0 , \; x\in \overline{\Omega},
\end{align}
where $\beta_F(x,y):=\sup_{X\in \mathbb{S}^N}{ |F(x,0,0,X)-F(y,0,0,X)|}{\|X\|^{-1}}$\, for $x,y\in \overline{\Omega}$. It holds for instance when $F$ satisfies $|F(x,0,0,X)-F(y,0,0,X)|\le \omega (|x-y|)\|X\|$, for all $x,y\in \overline{\Omega}$, $X\in \mathbb{S}^N$.

\smallskip

Finally, we assume that the Dirichlet problems associated to $F_i^*$ and $(F_i)_*$ are uniquely solvable (see Section \ref{section lambda1 scalar}) in the scalar sense together with regularity of solutions in a suitable Sobolev sense. In other words, in terms of Definition \ref{notation} \eqref{notation lambda1 scalar},\eqref{notation W2rho regularity},  we ask that $F_i$ satisfy the following 
\begin{align}\label{H lambdai>0}\tag{H3}
\lambda_1^+(F^*(\vartheta))>0, \qquad F, F^* \,\textrm{ satisfy } \eqref{H continuity},
\end{align}
\vspace{-0.7cm}
\begin{align}\label{H strong}\tag{H4}
\textrm{$F$\, enjoys $W^{2,\varrho}$ regularity.}
\end{align}

Here and onward in the text, the drift $\gamma$ and the zero order  term $\vartheta$  in \eqref{SC} may be unbounded, and this is an advantage of our paper over the usual literature  \cite{Arms2009, BQeq}, even in the scalar case. 
We observe that, in the system, the structure of each $F_i$ with respect to the zero order term in \eqref{SC} could be taken in terms of functions $\varrho_i$, which gives the possibility of prescribing different weights in \eqref{H lambdai>0}, see also Remark \ref{remark monotonicity}. We avoid including so many indexes in order to make the presentation cleaner.

Existence and positiveness in \eqref{H lambdai>0} are verified if $F^*$ is a  proper operator (nonincreasing in $r$) by \cite{regularidade}, and we will see this extends for nonproper operators as well, check Lemma \ref{remark nonproper scalar} ahead. Meanwhile, \eqref{H strong} will hold true if $F$ is a convex (or concave) operator in $X$ and satisfy \eqref{H continuity}, see Lemma \ref{lema convex}
(consequently, under \eqref{H continuity}, $F^*$ fulfills \eqref{H strong}). 
However, \eqref{H strong} also covers some nonconvex operators, for instance the asymptotic recession profiles in \cite{EE} for which a $W^{2,\varrho}$ theory is available. In particular, our results are valid for a class of Isaac's operators which are sufficiently close to a Bellman operator that has good regularity-estimates, see Example \ref{example}. Indeed, in Section \ref{appendix} we prove that they verify \eqref{H strong} under \eqref{HS}, even in the presence of unbounded coefficients.

It is worthwhile to mention that hypothesis \eqref{H strong} is not overly restrective and in fact it is a natural condition when dealing with comparison principles for $L^p$-viscosity solutions, see \cite{regularidade, arma2010}. 
Instead in the universe of $C$-viscosity solutions it is possible to skip it by the price of asking stronger continuity assumptions on the coefficients which is not our intention here, see \cite{Arms2009}.
Fully nonlinear equations with measurable ingredients were introduced in \cite{CCKS} for which a more general $L^p$-viscosity notion of solution is required. It is a modern theory which still develops, and results in such direction with unbounded coefficients are rather involved and delicate, see \cite{KSmpite2007, regularidade, BSvazquez}.

In this paper we treat the optimality of scalar spectral properties of fully nonlinear operators, and we also exploit the differences arising in the case of systems; both on a scenario with possible unbounded drift and weights.

\begin{rmk}
	Note that hypothesis \eqref{H lambdai>0} allows us to treat nonproper operators.
	This is equivalent to ask $\lambda_1^+(F(\vartheta))>0$ when $F$ is a linear operator. For systems, this condition on the operators $F_1$, $F_2$ is somehow required in \cite{LeiteMontenegro, Montenegro} in terms of MP's validity, see Definition \ref{def MP,mP}.
	Instead, in \cite[Theorem 1]{BQsystems} the coupling proposed does not fall upon the nonlinearities but on the operators.
	On the other hand, the regularity assumption \eqref{H strong} for $F_1$, $F_2$ seems to be optimal in the fully nonlinear case, with respect to the previous scalar works \cite{Arms2009, BQeq}, since there is no need to assume neither convexity nor continuity on the data.
\end{rmk}

\subsection{Statement of the main results} 
Let us consider the space $E_\varrho=W^{2,\varrho}(\Omega)\cap C(\overline{\Omega})$. Our main results are in the sequel. 
\begin{theorem}[Existence, simplicity, and asymptotics]\label{Th1 introduction}
Let $\Omega\subset\rN$ be a bounded $C^{1,1}$ domain. Assume $pq=1$, $\tau_1,\tau_2$ satisfy \eqref{H weights}, and $F_1,F_2$ satisfy \eqref{SC}--\eqref{H strong}.
Then there exist two spectral curves
\[
	\Lambda^\pm_1\, (\lambda) \,=\, (\,\lambda,\, \mu_1^\pm (\lambda)\,) \in \real^2,  \;\; \text{ for all } \lambda>0,
\]
in the first quadrant, corresponding to signed eigenfunctions $\varphi_1^\pm, \psi_1^\pm\in E_\varrho$ such that both the pairs $\varphi_1^+,\psi_1^+>0$ and $\varphi_1^-,\psi_1^-<0$  satisfy \eqref{LE} 
in the strong sense.

\smallskip

The eigenfunctions $(\varphi_1^+,\psi_1^+)$ and $(\varphi_1^-,\psi_1^-)$ are unique in the sense that any other eigenfunction $(u^{\pm}, v^\pm)$ corresponding to $\Lambda_1^\pm (\lambda)$ satisfies $u^\pm \equiv t \varphi^\pm$ and $v^\pm \equiv t^p \psi^\pm$ for a suitable $t  \in \mathbb{R}^+$. Furthermore, if $(u,v)$ is a signed solution of \eqref{LE}, then necessarily $(\lambda,\mu)\in \Lambda_1^\pm$.  
\smallskip

Moreover,  $\mu_1^\pm$ is continuous and strictly decreasing with $\lambda$, and the following asymptotic behavior holds
\begin{align}\label{asymptotics}
\textrm{$\mu_1^\pm (\lambda) \to \infty$\; as $\lambda \to 0$, \quad $\mu_1^\pm (\lambda)\to 0$\; as $\lambda \to \infty$.}
\end{align}
\end{theorem}
\begin{rmk}
For the explicit shape of $\Lambda_1^\pm$, see \eqref{eq:thetwocurves} ahead.
\end{rmk}
In what follows we deal with geometric properties of the first spectral curves and their characterization. In what follows, for a parametrized curve $\Sigma=\Sigma (\lambda)=(\lambda,\sigma(\lambda))$, where $\lambda>0$ and $\sigma(\lambda)$ continous, we say that $\Sigma$ is above $\Lambda_1^\pm$  if $\mu_1^\pm(\lambda)<\sigma(\lambda)$ for every $\lambda>0$.

\begin{theorem}[Local isolation]\label{Th isolation Introdu} 
Assume $pq=1$, $\tau_1,\tau_2$ satisfy \eqref{H weights}, $F_1,F_2$ satisfy \eqref{SC}--\eqref{H strong}, and take $\Lambda_1^\pm$ from Theorem \ref{Th1 introduction}.
Then there exists a curve $\Sigma=(\lambda,\sigma(\lambda))$, strictly above $\Lambda_1^\pm$, such that $\sigma$ is strictly decreasing, $\sigma(\lambda)\to \infty$ as $\lambda\to 0$, $\sigma(\lambda)\to 0$ as $\lambda\to \infty$, with the property that:
if  $(\lambda,\mu)$ is an eigenvalue of \eqref{LE}  in the region
\[
\{\,(\lambda,\mu)\in \mathbb{R}^2:\ \lambda>0,\;\, 0<\mu<\sigma(\lambda)\,\},
\]
then necessarily $(\lambda,\mu)\in \Lambda_1^+\cup \Lambda_1^-$. 
 In other words, in the first quadrant, below and slightly above $\Lambda_1^\pm$ there are no other eigenvalues of  \eqref{LE}.
\end{theorem}

Next we see how the region below each curve $\Lambda_1^\pm$ gives a complete characterization of the plane $\real^2$ in terms of maximum and minimum principles, and in terms of the solvability of the associated Dirichlet problem. 
This extends \cite{LeiteMontenegro} to viscosity solutions, and plays the role of the condition $\lambda<\lambda_1^\pm$ in the scalar case. 

\begin{defin}[MP and mP]\label{def MP,mP}
	We say that the maximum principle (MP) holds for \eqref{LE} if any viscosity subsolution of \eqref{LE}, that is, any solution pair  $u,v\in C(\overline{\Omega})$ of
	\begin{align}\label{eq MP}
	F_1 [u]+\lambda \tau_1(x)|v|^{q-1} v  \geq 0 , \;\;\;
	F_2 [v]+\mu\tau_2(x) |u|^{p-1}u \geq 0 \;\;\textrm{ in } \Omega ,\quad\textrm{$u,v\leq 0$ on $\partial\Omega$}
	\end{align}
	satisfies $u, v \le 0$ in $\Omega$. 
	Likewise, we say that the minimum principle (mP) holds for \eqref{LE}
	if $u, v \ge 0$ in $\Omega$ for any viscosity supersolution pair $u,v\in C(\overline{\Omega})$ of
	\begin{align}\label{eq mP}
	F_1 [u]+\lambda \tau_1(x)|v|^{q-1} v  \leq 0 , \;\;\;
	F_2 [v]+\mu\tau_2(x) |u|^{p-1}u \leq 0 \;\;\textrm{ in } \Omega ,\quad\textrm{$u,v\geq 0$ on $\partial\Omega$}.
	\end{align} 
\end{defin}

Let $ \C^+$ be the open region in the first quadrant below $\Lambda_1^+$, and similarly for $\C^-$ associated to $\Lambda_1^-$, namely:
\begin{equation}\label{eq:C+-}
\C^\pm:=\{\, (\lambda,\mu)\in \mathbb{R}^2:\ \lambda>0,\;\, 0<\mu<\mu_1^\pm(\lambda) \,\}.
\end{equation}

\begin{theorem}[Characterization of $\Lambda_1^\pm$] \label{Lambda1pm} Assume $pq=1$, $\tau_1,\tau_2$ satisfy \eqref{H weights}, $F_1,F_2$ satisfy \eqref{SC}--\eqref{H strong}, and let $\Lambda_1^\pm$ be as in Theorem \ref{Th1 introduction}. Then:
\begin{enumerate}[(i)]
	\item $(\lambda,\mu)\in \overline{ \C^+} \setminus \Lambda_1^+$ if, and only if, MP holds for \eqref{LE};
	
	\item $(\lambda,\mu)\in \overline{ \C^-}  \setminus \Lambda_1^-$ if, and only if, mP holds for \eqref{LE}.
\end{enumerate}
\end{theorem}

\smallskip

Let us now consider the Dirichlet problem 
\eqref{Dir lambda mu} in the viscosity sense, for functions $f_1,f_2\in L^\varrho(\Omega)$, $\varrho>N$, with $u,v\in C(\overline{\Omega})$. 

\begin{theorem}[Solvability of the Dirichlet problem]\label{ThDir solvability intro}
Assume $F_1,F_2$ satisfy \eqref{SC}--\eqref{H strong}, $pq=1$, $\tau_1,\tau_2$ satisfy \eqref{H weights}, and let $f_1,f_2\in L^\varrho(\Omega)$. 
\begin{enumerate}[(i)]
\item If $(\lambda,\mu)\in \C^+ \cap \C^-$, then \eqref{Dir lambda mu} is solvable among viscosity solutions.
	
\item If $f_1,f_2\leq 0$ a.e.\ and $(\lambda,\mu ) \in \C^+$, then \eqref{Dir lambda mu} has a unique nonnegative solution pair in $E_\varrho$.
			
\item If $f_1,f_2 \geq 0$ a.e.\ and $(\lambda,\mu) \in \C^-$, then \eqref{Dir lambda mu} has a unique nonpositive solution pair in $E_\varrho$.
\end{enumerate}
\end{theorem}	

In other words, as long as we are in the region $\C^+ \cap \C^-$ we obtain the complete standard solvability of the Dirichlet problem \eqref{Dir lambda mu}. The solvability up to $\C^+\cup \C^-$ may not hold in general (e.g.\ \cite[Theorem 1.8]{BQeq}), in contrast to the case of linear eigenvalues in \cite{LeiteMontenegro}.
However, if the pair $f_1,f_2$ has the ``good" sign, we do obtain solvability in this larger region. 

\begin{rmk}
In particular, Theorem \ref{ThDir solvability intro} applied to $F_1=F_2=\mathcal{L}^+$, $\tau_1=\tau_2=\tau$, $p=q=1$ (scalar case) gives the optimal range of solvability in \cite[Proposition 3.4]{arma2010}; see also Remark p.595 therein where the problem of spectral properties for these operators under unbounded coefficients was left open. See also our Theorem \ref{ABP-MP} in Section \ref{section lambda1 scalar} for a priori bounds of ABP type to solutions produced by Theorem \ref{ThDir solvability intro} in the scalar case, as well as necessary conditions which characterize their validity.
\end{rmk}

In what concerns the optimality of Theorem \ref{ThDir solvability intro}, we show that an \textit{anti-maximum principle} occurs when we move a little bit above the region $\C^+\cup \C^-$. This type of result is essential, for instance, in bifurcation and resonance phenomena \cite{ArcoyaAMP2001, BFQjfaLL, BFQsiam2010}.
A classical reference for it in the linear scalar case is \cite{CPamp} (see also \cite{BiAMP95}), while its fully nonlinear scalar counterpart can be found in \cite[Theorem 2.5]{Arms2009}. Here we extend \cite{Arms2009} to fully nonlinear Lane-Emden systems as follows.

\begin{theorem}[Anti-maximum principle]\label{AMP}
Let $F_1, F_2$ satisfy \eqref{SC}--\eqref{H strong}, $pq=1$, $\tau_1,\tau_2$ satisfy \eqref{H weights},  and $f_i\in L^\varrho(\Omega)$ with $f_i\not\equiv 0$, $i=1,2$, $\varrho>N$. Then there exists a curve $\Gamma=(\lambda,\bar{\gamma}(\lambda))$, depending on $f_1,f_2$, which is above $\Lambda_1^\pm $, where $\bar{\gamma}$ is stricly decreasing, $\bar{\gamma}(\lambda)\to 0$ as  $\lambda\to \infty$, and $\bar{\gamma}(\lambda)\to \infty$  as  $\lambda\to 0$, such that
	\begin{enumerate}[(i)]
		\item if $f_1,f_2 \leq 0$ a.e., $\Lambda_1^+$ is below or coincides with $\Lambda_1^-$, and $(\lambda,\mu)$ is a pair between $\Lambda_1^-$ and $\Gamma$, then	
		any solution pair $u,v\in C(\overline{\Omega})$ of \eqref{Dir lambda mu} satisfies $u,v < 0$ in $\Omega$;
		\item if $f_1,f_2 \geq 0$ a.e., $\Lambda_1^-$ is below or coincides with $\Lambda_1^+$, and $(\lambda,\mu)$ is a pair between $\Lambda_1^+$ and $\Gamma$,
		then any solution pair $u,v\in C(\overline{\Omega})$ of \eqref{Dir lambda mu} verifies $u,v > 0$ in $\Omega$.
	\end{enumerate}
\end{theorem}

We highlight that Theorem \ref{AMP} is new even in the case of the standard Lane-Emden system involving the Laplacian operator, i.e.\ when $F_1=F_2=\Delta$. Up to our knowledge, this is the first result on anti-maximum principle regarding strongly coupled systems. 

We also prove an existence result for the region above $\Lambda_1^+$ and $\Lambda_1^-$ when $p=q=1$. Let us consider the pairs $(\lambda^\pm_1,\lambda^\pm_1)=(\lambda_1^\pm(F_1,F_2),\lambda_1^\pm(F_1,F_2))$ in the intersection of the curve $\Lambda_1^\pm$ with the line $\lambda=\mu$ (cf. Sections \ref{Preliminaries_2} and \ref{section scaling}). Then we define the following quantity: 
\[ 
\lambda_2=\lambda_2(F_1,F_2,\Omega) := \inf \{ \lambda> \max\{\lambda_1^+(F_1,F_2), \lambda_1^-(F_1,F_2) \}: \, (\lambda, \lambda) \text{ is an eigenvalue of \eqref{LE}} \}, 
\]
which could be infinite. Denote by $K$ the first quadrant on the plane $(\lambda,\mu)$. 
\begin{theorem}[The second spectral curve and the Dirichlet problem]\label{Th lambda2 Introdu}
Let $F_1, F_2$ satisfy \eqref{SC}--\eqref{H strong}, $\tau_1,\tau_2$ satisfy \eqref{H weights}, and $pq=1$. 
\begin{enumerate}
\item[(i)]If $\lambda_2 < \infty$, then there exists a curve $\Lambda_2=(\lambda,\mu_2(\lambda))$ lying in the region $K\setminus \overline{ \C^+\cup \C^-}$.  Moreover, $\Lambda_2$ is
such that each point $(\lambda, \mu)$ on $\Lambda_2$ is an eigenvalue for \eqref{LE}. Also, $\mu_2(\lambda)$ is continuous, stricly decreasing, and satisfies
\[ \mu_2(\lambda) \to \infty, \text{ as } \lambda \to 0, \quad \mu_2(\lambda) \to 0 \text{ as } \lambda \to \infty. \]
\item[(ii)] Further, if $p=q=1$, the Dirichlet problem \eqref{Dir lambda mu} is solvable for $f_1,f_2\in L^\varrho(\Omega)$, $\varrho>N$, when $(\lambda, \mu)$ belongs to the region 
\[
\left(\mathcal{C}_2\setminus \overline{ \C^+\cup \C^-}\,\right)
\]
where $\mathcal{C}_2$ the region below $\Lambda_2$ in $K$ if $\lambda_2<\infty$, while $\mathcal{C}_2=K$ if $\lambda_2=\infty$.
\end{enumerate}
\end{theorem}
\begin{rmk}
The explicit parametrization of $\Lambda_2$ is given in \eqref{eq:Lambda2} ahead.
\end{rmk}
This result is an extension for systems of \cite[Theorem 2.4]{Arms2009}. We mention that one may have $\lambda_2=+\infty$ if for instance $F_1,F_2$ are not symmetric, see \cite{Arms2009}. Meanwhile, $\lambda_2<+\infty$ in the scalar case when $F_1=F_2$ is a Pucci's radial operator and $\tau_1=\tau_2=1$, see \cite{EFQPucciradial}. Note that if $\tau_1=\tau_2=1$ then $\lambda_2(\Delta, \Delta)= \lambda_2(\Delta)$ (the second eigenvalue of the Laplacian operator).
In general, finding higher eigenvalues for systems is a difficult issue and it seems that only particular cases involving the Laplacian operator are available. We quote a one dimensional picture displayed in \cite[Section 3]{DenisEderson2010} for $pq=1,$ and \cite{CM} for a higher dimensional scenario when $p=q=1$; both explore the method of reduction by inversion which transforms the second order system into single equation of higher order.
Here instead we use a degree-theoretical approach which allows us to deal with fully nonlinear operators in any dimension when $p=q=1$; the general case $pq=1$ is still open.  

\subsection{Examples and applications}

We start by highlighting that the curves $\Lambda_1^+$ and $\Lambda_1^-$ obtained in Theorem \ref{Th1 introduction} can be different when the operator is not linear, as shows the following example.

\begin{example} In light of \cite{EFQPucciradial}, one may consider the Fucik-like spectrum 
	\begin{center}
		$Lu+\lambda u^+-\frac{\lambda}{\kappa}\,u^-=0$,
	\end{center}
	associated to the linear operator $Lu:=\mathrm{tr}(A(x)D^2u)+\gamma(x)\cdot Du$, where $\kappa>0$ is fixed, which can be viewed as the spectrum of a nonlinear convex or concave operator given by
	\begin{center}
		$F_1[u]:=\max\{Lu, \kappa Lu \} =-\lambda u$\, if $\kappa\ge 1$, \quad $F_2[u]:=\min\{Lu, \kappa Lu \} =-\lambda u$\, if $\kappa\le 1$.
	\end{center}	
	
	Now let us fix $\kappa>1$ and $p,q>0$ such that $pq=1$. 
	From \cite{Montenegro}, there exists a first positive eigenvalue-parameter $\sigma$ (see Section \ref{section scaling}) and an eigenfunction pair $(\varphi,\psi)$ so that
	\begin{align}\label{ex: Fucik linear Mont}
	\textrm{$L \varphi+\sigma\psi^q=0$, \; 
		$L \psi+\sigma\varphi^p=0$, \; $\varphi,\psi>0$ \; in $\Omega$, \qquad $\varphi,\psi=0$ \; on $\partial\Omega$. }
	\end{align}
	Then the pair $\varphi_1^+:=\varphi$, $\psi_1^+:=t\psi$, with $t>0$ to be chosen, solves
	\begin{center}
		$F_1[\varphi^+_1]=\max\{ L\varphi, \,\kappa L\varphi  \}= -\sigma  \psi^q=  -\sigma t^{-q} (\psi^+_1)^q$,\medskip
		
		$F_2[\psi^+_1]=\min\{ tL\psi,\, \kappa t L\psi  \}= - \kappa t\sigma  \varphi^p=  -\kappa t \sigma (\varphi^+_1)^p$,
	\end{center}
	so it is a positive eigenfunction pair if one choses $t:=\kappa^{-\frac{1}{q+1}}$. Moreover, this eigenpair is unique up to scaling by our Theorem \ref{Th1 introduction}, and we conclude
	$\lambda_1^+(F_1,F_2)=\kappa^{\frac{q}{q+1}}\sigma$. \smallskip
	
	Analogously, the pair given by $\varphi_1^-:=-\varphi$, $\psi_1^-:=-s\psi$ solves 
	\begin{center}
		$F_1[\varphi^-_1]=\max\{ -L\varphi ,- \kappa L\varphi \}= \kappa \sigma  \psi^q=  -\kappa\sigma s^{-q}\, |\psi^-_1|^{q-1}\psi^-_1$,\medskip
		
		$F_2[\psi^-_1]=\min\{ -sL\psi, -\kappa s L\psi  \}=  s\sigma  \varphi^p=  -s \sigma \,|\varphi^-_1|^{p-1}\varphi_1^-$,
	\end{center}
and becomes a negative eigenfunction pair when $s:=\kappa^{\frac{1}{q+1}}$ -- again unique up to scaling by Theorem \ref{Th1 introduction}, and 
	$\lambda_1^-(F_1,F_2)=\kappa^\frac{1}{q+1}\sigma$, compare with the scalar case in \cite[Example 3.10]{Arms2009}. \smallskip
	
	Now, since $\kappa>1$, then one always has $\lambda_1^+(F_1,F_2)\neq \lambda_1^-(F_1,F_2)$, for $q\neq 1$, and so by scaling one recovers that the two parallel curves
	$\Lambda^+_1 $ and $\Lambda^-_1$ are different, see Section \ref{section scaling}.
	Furthermore, $\Lambda_1^+$ stays below $\Lambda_1^-$ if $q<1$; while  $\Lambda_1^+$ lies above $\Lambda_1^-$ if $q>1$.
	
On the other hand, is also simple to verify that for $\kappa>1$, $\lambda_1^+(F_1, F_1) = \sigma < \kappa \sigma = \lambda_1^-(F_1, F_1)$ and $\lambda_1^-(F_2, F_2) = \sigma < \kappa \sigma = \lambda_1^+(F_2, F_2)$, for all $p,q>0$ with $pq=1$.
	
	More generally, one can also consider different operators $L_1, L_2$ in \eqref{ex: Fucik linear Mont}. 
\end{example}

The next examples comprise important classes of fully nonlinear operators for which all our results apply, being novelties even in the scalar case in the context of unbounded drift and weight.

\begin{example}\label{example pucci}
Simple prototypes we may have in mind are extremal operators involving Pucci's, for instance $F_1=F_1^*=\mathcal{L}^+$, $F_2=(F_2)_*=\mathcal{L}^-$, with $\vartheta=\ell \theta$, for some $\ell>0$ and $\theta\in L^\varrho(\Omega)$ satisfying $\theta\gneqq 0$ a.e.\ in $\Omega$.  They obviously fulfill \eqref{SC}, \eqref{H homogeneity}, \eqref{H continuity}, and \eqref{H strong}. Moreover, recall $\lambda_1^+(\mathcal{L}^+_0(\theta))>0$
	by \cite{regularidade} (see also our Proposition \ref{prop existence eigenv proper}), and $\lambda_1^+(\mathcal{L}^-(\vartheta))=\lambda_1^-(\mathcal{L}^+(\vartheta))\ge \lambda_1^+(\mathcal{L}^+(\vartheta))$.
	Thus \eqref{H lambdai>0} is verified for $F_1$, $F_2$ if one chooses  $\ell <\lambda_1^+(\mathcal{L}^+_0(\theta))$ as in \eqref{eq ell hatvartheta}, see Lemma \ref{remark nonproper scalar}. 
\end{example}

\begin{example}\label{example} A bit more sophisticated model case arising from control theory are Hamilton-Jacobi-Bellman-Isaac's type operators \cite{Arms2009, BFQsiam2010,BQeq, Eaihpan2019},  with unbounded coefficients, such as
	\begin{align}\label{Isaacs}
	\textstyle F_1[w]=\sup_{s\in \n} \inf_{t\in \n } L_{s,t} (w)\, , \qquad F_2[w]=\inf_{s\in \n} \sup_{t\in \n } L_{s,t}(w), 
	\end{align}	
	wher $L_{s,t}$ for $s,t\in \n$ is a linear operator in the form
	\begin{align}\label{Isaacs Ls,t}
	L_{s,t}[w]= \mathrm{tr} (A_{s,t}(x)D^2 w) +\gamma_{s,t}(x)\cdot Dw +\ell\vartheta_{s,t}(x)w, \quad \ell <\lambda_1^+(\mathcal{L}_0^+(\vartheta)),\qquad  \\
|\gamma_{s,t}| \leq \gamma , \;|\vartheta_{s,t}|\le \vartheta, \; \gamma,\vartheta\in L^\varrho_+(\Omega), \;\alpha I \le A_{s,t}\le \beta I, \; 	A_{s,t}\in C(\overline{\Omega}) \textrm{ uniformly in } s,t\in \n,\nonumber
	\end{align} for $0<\alpha\le\beta$.
	For instance, when $L_{s,t}\equiv L_{0,t}$ for all $s\in \n$ then $F_1$ and $F_2$ are called Bellman operators, which are concave and convex, respectively. In the general case, $F_1$ and $F_2$ are neither convex nor concave, and are called Isaac's operators.
	Again, \eqref{H lambdai>0} holds for $F_1$, $F_2$ under \eqref{Isaacs Ls,t} as in \eqref{eq ell hatvartheta}, see Lemma \ref{remark nonproper scalar}; while we show \eqref{H strong}  for \eqref{Isaacs}--\eqref{Isaacs Ls,t} under \eqref{HS} in Section \ref{appendix}. 
\end{example}

\paragraph{Results in the superlinear and sublinear regimes}
As a byproduct of our arguments, we complement a study on maximum principles in small domains for Lane-Emden systems. Besides being of independent interest \cite{BNV}, they play an important role in symmetry problems \cite{BN, BdaLioSym}, aside from spectral constructions when the domain is not smooth \cite{BQeq}.
It also appeared in \cite{UC} as the main tool to derive a Unique Continuation Principle of radial fully nonlinear type in the case $pq\ge 1$.
For $pq=1$, an explicit form  was previously/independently proved in \cite[Theorem 1.3, Corollary 1.1]{LeiteMontenegro}, under a smallness hypothesis on the weights. Here we instead make it an alternative, by asking either the domain or the weighs to be small. This in particular extends and unifies \cite[Lemma 5.4]{regularidade} (for domains with small measure) and \cite[the maximum principle in Proposition 3.4]{arma2010} (for operators with small weight) even in the scalar case. Furthermore, in what concerns the system, the result is valid not only for $pq=1$ but also for $pq \ge 1$.

\begin{theorem}\label{MP small domain}
Let $F_1, F_2$ satisfy \eqref{SC}. Let $\tau_1,\tau_2\in L^\varrho(\Omega)$, $\varrho>N$, $pq\geq  1$, and $\lambda,\mu\geq 0$. Then the following MP result holds.
\begin{itemize}
\item[(i)] Assume  \eqref{H continuity}, \eqref{H homogeneity} hold for $F_i^*$ with $\lambda_1^+(F^*_i(\vartheta))>0$, $i=1,2$.	
Then, there exists $\varepsilon_0>0$, depending on $N,\varrho,p,q,\alpha,\beta, \lambda,\mu,\| \gamma\|_{L^\varrho(\Omega)}$, $\| \vartheta\|_{L^\varrho(\Omega)}$, $\|\tau_1\|_{L^\varrho(\Omega)}$, $\|\tau_2\|_{L^\varrho(\Omega)}$, $\|u\|_\infty$, $\|v\|_\infty$, $\mathrm{diam}(\Omega)$, and $\lambda_1^+(F_i^*(\vartheta))$, such that if \begin{center}
	either \; $|\Omega|\leq \varepsilon_0$ \; or \; $\min \{\,\|\tau_1\|_{L^N(\Omega)}\,\|\tau_2\|^q_{L^N(\Omega)} , \,\|\tau_1\|^p_{L^N(\Omega)}\,\|\tau_2\|_{L^N(\Omega)} \}\le \epsilon_0$,
\end{center}
then any viscosity subsolution pair $u,v\in C(\overline{\Omega})$ of
	\begin{align*}
	F_1[u]+\lambda\tau_1(x)|v|^{q-1}v \geq  0, \quad F_2[v]+\mu\tau_2(x)|u|^{p-1}u \geq  0 \textrm{ in } \Omega,\quad	u,v \leq  0 \textrm{ on } \partial\Omega,
	\end{align*}
satisfies $u,v\leq 0$ in $\Omega$.  
\item[(ii)] Assume  \eqref{H continuity}, \eqref{H homogeneity} hold for $(F_i)_*$ with $\lambda_1^+((F_i)_*(\vartheta))>0$, $i=1,2$.
Then, there exists $\varepsilon_0>0$, depending on $N,\varrho,p,q,\alpha,\beta, \lambda,\mu,\| \gamma\|_{L^\varrho(\Omega)}$, $\| \vartheta\|_{L^\varrho(\Omega)}$, $\|\tau_1\|_{L^\varrho(\Omega)}$, $\|\tau_2\|_{L^\varrho(\Omega)}$, $\|u\|_\infty$, $\|v\|_\infty$, $\mathrm{diam}(\Omega)$, and $\lambda_1^+((F_i)_*(\vartheta))$, such that if 
\begin{center}
	either \; $|\Omega|\leq \varepsilon_0$ \; or \;  $\min \{\,\|\tau_1\|_{L^N(\Omega)}\,\|\tau_2\|^q_{L^N(\Omega)}, \,\|\tau_1\|^p_{L^N(\Omega)}\,\|\tau_2\|_{L^N(\Omega)}\}\le \epsilon_0$,
\end{center}
then for any $u,v\in C(\overline{\Omega})$ viscosity supersolution of
\[
 F_1[u]+\lambda\tau_1(x)|v|^{q-1}v \leq  0,\quad  F_2[v]+\mu\tau_2(x)|u|^{p-1}u \leq  0 \text{ in  } \Omega, \quad u,v\geq 0\text{ on }\partial\Omega,
\]
one has $u,v\geq 0$ in $\Omega$.
\end{itemize}
If $pq=1$, then $\varepsilon_0$ does not depend on a bound from above of the $L^\infty$ norms of $u,v$.
\end{theorem}

On the other hand, the eigenvalue problem furnishes unique solvability in the sublinear regime.

\begin{theorem}[Sublinear regime]\label{sublinear}
Assume $F_1,F_2$ verify \eqref{SC}--\eqref{H strong}.  Let $f_i\in L^\varrho(\Omega)$ with $f_i\le 0$ a.e.\ in $\Omega$, $i=1,2$, and $p,q>0$ such that $pq<1$. Then the problem \eqref{Dir lambda mu} is uniquely solvable among positive viscosity solutions for all $\lambda,\mu>0$.
\end{theorem}

Problems of this nature, for instance involving the Laplacian operator, have been studied in \cite[Theorem 3]{DalmassoUniq} \cite{Montenegro}, and \cite[Theorem 7.1]{EdTran2018}. We also mention that uniqueness results imply that solutions inherit all symmetries of the problem. For example, if the operator and its domain are radially symmetric, then so is the solution. Furthermore, uniqueness simplifies the dynamics of evolution problems, and in many cases provides global stability properties of equilibrium.

\subsection{Structure of the paper}
The rest of the paper is organized as follows. In Section \ref{Preliminaries} we recall some preliminary facts and definitions. In particular, we define the notion of principal eigenvalues for the system \eqref{LE}.

Section \ref{section lambda1 scalar} is devoted to the study of the scalar case, namely we prove how principal eigenvalues relate to maximum principles of ABP type, from which we obtain Theorem \ref{MP small domain} as a consequence. 

In Section \ref{section aux} we prove some auxiliary results for systems, in particular sufficient conditions that imply uniqueness of solutions to systems (up to scaling), and a priori bounds for the first eigenvalue.

Section~\ref{section pq=1}  addresses the main properties of the first eigenvalue problem when $pq=1$. It contains the proofs of Theorems \ref{Th1 introduction}, \ref{Th isolation Introdu}, \ref{Lambda1pm} and \ref{ThDir solvability intro} (i).

In Section \ref{sec:antimax} we deal with the anti-maximum principle, proving Theorem \ref{AMP}.

In Section  \ref{section second curve} we treat the second eigenvalue problem, namely Theorem \ref{Th lambda2 Introdu}.

Section \ref{section pq<1} is dedicated to solvability of the Dirichlet problem when the functions $f_i$ have a sign. In this spirit we develop a unified proof for both Theorem \ref{ThDir solvability intro} (ii)--(iii) and Theorem \ref{sublinear}. 

Finally, Section \ref{appendix} is devoted to  $W^{2,\varrho}$ regularity of Isaac's operators in Example \ref{example}.

\section{Preliminaries}\label{Preliminaries}

In this section we begin by recalling some of the different notions of viscosity solutions and their equivalence under regularity of the data. We then recall important results such as the Alexandrov-Bakelman-Pucci (ABP for short) maximum principle, the strong maximum principle and Hopf's lemma. In the second part we introduce the notion of principal eigenvalues and comment on the scaling properties of the system \eqref{Dir lambda mu}. All these results are used  throughout the paper.

\subsection{Some known results}\label{Preliminaries_1}

Let us start by recalling the definition of Pucci's operators
\begin{align}\label{def Pucci}
\textstyle \mathcal{M}^+_{\alpha,\beta}(X):=\sup_{\alpha I\leq A\leq \beta I} \mathrm{tr} (AX)\,,\quad \mathcal{M}^-_{\alpha,\beta}(X):=\inf_{\alpha I\leq A\leq \beta I} \mathrm{tr} (AX),
\end{align}
and of viscosity and strong solutions in what follows.
\begin{defin}\label{def Lp-viscosity sol}
	Let $f\in L^\varsigma_{\textrm{loc}}(\Omega)$ for some $\varsigma\geq N$, and $F$ an operator satisfying \eqref{SC}. 
We say that $u\in C(\Omega)$ is an \textit{$L^\varsigma$-viscosity subsolution} $($resp. supersolution$)$ of $F[u]=f(x)$ in $\Omega$ if whenever $\phi\in  W^{2,\varsigma}_{\mathrm{loc}}(\Omega)$, $\varepsilon>0$, and $\mathcal{O}\subset\Omega$ open are such that
	\begin{align*}
	F(x,u(x),D\phi(x),D^2\phi (x))-f(x)  \leq -\varepsilon\qquad
	( F(x,u(x),D\phi(x),D^2\phi (x))-f(x)  \geq \varepsilon )
	\end{align*}
for a.e. $x\in\mathcal{O}$, then $u-\phi$ cannot have a local maximum $($minimum$)$ in $\mathcal{O}$. 
In this case we also say that $u$ is a \emph{viscosity solution of the inequality} $F[u]\ge f(x)$ (resp.\ $F[u]\le f(x)$) in $\Omega$.

	A \textit{strong subsolution} (resp.\ supersolution) belongs to $W^{2,\varsigma}_{\mathrm{loc}}(\Omega)$ and satisfies the inequality $F[u]\ge f(x)$ (resp.\ $F[u]\le f(x)$) at almost every point $x\in \Omega$.  
	
	In each situation, a \textit{solution} $u \in C(\Omega)$ is meant to be both subsolution and supersolution.
\end{defin}

We now comment on the equivalence of these definitions. 
\begin{itemize}
\item The notions of $L^\varrho$-viscosity and strong solutions are equivalent whenever the solution belongs to the space $W^{2,\varrho}_{\mathrm{loc}}(\Omega)$, see \cite[Theorem 3.1, Proposition 9.1]{KSweakharnack}.
\item Under \eqref{SC}, the concepts of $L^\varrho$ and $L^N$ viscosity solutions are also equivalent whenever $\varrho>N$ and $f\in L^\varrho_{\mathrm{loc}}(\Omega)$, see \cite[Proposition 2.9]{tese} (observe that $L^\varrho_{\mathrm{loc}}(\Omega)\subset L^N_{\mathrm{loc}}(\Omega)$).
\end{itemize}
Thus, given $f\in L^\varrho_{\mathrm{loc}}(\Omega)$, throughout the text, we say simply \textit{viscosity solution} of $F[u]=f$ to mean an $L^N$-viscosity solution. This in turn is equivalent to be a strong solution when hypothesis \eqref{H strong} is in force.

\begin{rmk} In order to unify the notation, we always assume
$f,\vartheta\in L^\varrho$ by means of producing $C^{1,\alpha}$ solutions under \eqref{H continuity} as in Proposition \ref{C1,alpha regularity estimates geral} (such strategy was also employed in \cite{Arms2009} to treat nonconvex operators), despite sometimes this integrability can be relaxed to $L^N$.
\end{rmk}

Next we recall the ABP maximum principle for proper operators with unbounded drift (for a proof see \cite[Proposition 2.8]{KSite}). Recall from \eqref{def Lpm rho} the notation
\[
\mathcal{L}_0^\pm[u]:=\mathcal{L}_0^\pm(x,Du,D^2u)=\mathcal{M}^\pm (D^2u)\pm \gamma(x)|Du|, \qquad  \mathcal{L}^\pm[u]:=\mathcal{L}_0^\pm[u]\pm \vartheta (x)|u|.
\]

\begin{prop}[ABP]\label{ABPproper}
Let $f \in L^N(\Omega)$, $\gamma\in L_{+}^\varrho(\Omega)$ for $\varrho>N$, and $u \in C(\overline{\Omega})$ be a viscosity solution of $\mathcal{L}_0^{+}[u] \geq  f(x)$ in $\Omega^+$ (resp.\ $\mathcal{L}^{-}_0[u] \leq  f(x)$ in $\Omega^-$),
where $\Omega^\pm =\Omega\cap \{\pm u>0\}$.
Then
\begin{align}\label{eq:ABP-MP-mP}
\textstyle \max_{\overline{\Omega}} u \leq \max_{\partial \Omega} u^+ +C\,\|f^-\|_{L^N(\Omega)}  
 \quad (\mathrm{resp.}\;  \min_{\overline{\Omega}} u \geq \min_{\partial \Omega} (-u^-) -C\,\|f^+\|_{L^N(\Omega)}),
\end{align}
for a universal constant $C=C(N, \alpha,\beta, \|\gamma\|_\varrho, \mathrm{diam}(\Omega))>0$, which is bounded if theses quantities are bounded from above. We denote this constant by $C_A$.
\end{prop}

\begin{rmk}\label{rmk ABP proper}
Recall that $F$ is called proper if $F (x,r,\xi, X)\le F(x,s,\xi,X)$ for $r\ge s$. Therefore, the previous statement can be applied to proper operators, since for instance 
\[
\mathcal{L}^+_0[u]\ge F(x,0,Du,D^2u)\ge F(x,u,Du,D^2u)=F[u]\;\; \textrm{ in }\Omega^+.
\]
\end{rmk}

A consequence of ABP is the following result on the stability of viscosity solutions (see \cite[Theorem 4]{arma2010}, which is based on \cite[Theorem 3.8]{CCKS}).

\begin{prop}[Stability]\label{stability}
	Let $F$, $F_k$ be operators satisfying \eqref{SC}, $f, \, f_k\in L^\varrho(\Omega)$. Let $u_k\in C(\Omega)$ be a viscosity solution of
	$F_k[u_k]\geq f_k(x) $ in $\Omega$  ({resp}.\ $\leq f(x)$) for all $k\in \n$.
	Suppose  $u_k\to u$ in $L_{\mathrm{loc}}^\infty  (\Omega)$ as $k\to \infty$ and that, for each ball $B\subset\subset \Omega$ and $\varphi\in W^{2,\varrho}(B)$, setting
	\begin{align*}
	g_k(x):=F_k(x,u_k,D\varphi,D^2\varphi)-f_k(x) \,\,\,\,\, \text{and}\,\,\,\,
	g(x):=F(x,u,D\varphi,D^2\varphi)-f(x),
	\end{align*}
	we have
	$ \| (g_k-g)^+\|_{L^\varrho(B)} \to 0$ as $ k\to \infty.$ (resp. $(\| (g_k-g)^-\|_{L^\varrho(B)}) \to 0 $ as $k\to 0$).
Then $u$ is a viscosity solution of $F[u]\geq f(x)$ (resp.\ $\leq f(x)$) in $\Omega$.
\end{prop}

Next we recall some important results concerning the strong maximum principle and the Hopf lemma from \cite{B2016}.
We often refer to them simply by \textit{SMP} and \textit{Hopf} along the text.

\begin{prop}[SMP]\label{SMP}
	Let $\Omega$ be a $C^{1,1}$ domain and $u$ a viscosity solution of $\mathcal{L}^-[u]\leq 0$, $u\geq 0$ in $\Omega$, where $\gamma, \vartheta\in L^\varrho_+(\Omega)$. Then either $u>0$ in $\Omega$ or $u\equiv 0$ in $\Omega$.
\end{prop}

\begin{prop}[Hopf]\label{Hopf}
	Let $\Omega$ be a $C^{1,1}$ domain and $u$ a viscosity solution of $\mathcal{L}^-[u]\leq 0$, $u> 0$ in $\Omega$, where $\gamma, \vartheta\in L^\varrho_+(\Omega)$. If $u(x_0)=0$ for some $x_0\in\partial\Omega$, then $\varliminf_{t\to 0^+} {u(x_0+t\nu)}/{t} >0$, where $\nu$ is the interior unit normal vector to $\partial\Omega$ at $x_0$.
\end{prop}

In \cite{B2016}, Propositions \ref{SMP} and \ref{Hopf} are proved for $\vartheta\equiv 0$, but the same proofs there work for any coercive operator, in particular for $\mathcal{L}^-=\mathcal{L}_0^--\vartheta (x)u$\, since $u\ge 0$; see also \cite{BSvazquez}.

To conclude we recall $C^{1,\alpha}$ regularity estimates for equations with unbounded drift from \cite{regularidade}. Since Theorem 1 therein was stated for bounded zero order terms, we briefly explain how to deduce them also for merely $L^\varrho$-integrable ones.

\begin{prop}
\label{C1,alpha regularity estimates geral}
Assume $F$ satisfies \eqref{SC}, $f  \in L^\varrho (\Omega)$, $\varrho>N$, and $\Omega\subset\rN$ is a bounded domain.
Let $u$ be a viscosity solution of $F[u]=f(x)$ in $\Omega$. Then, there exists $\alpha\in (0,1)$ and $\theta_0=\theta_0 (\alpha)$, depending on $N,\varrho,\lambda,\Lambda,\|\gamma\|_{L^\varrho(\Omega)}$, such that if \eqref{H continuity} holds for all $r\leq \min \{ r_0, \mathrm{dist} (x,\partial\Omega)\}$, for some $r_0>0$ and for all $x \in \Omega$, this implies that $u\in C^{1,\alpha}_{\mathrm{loc}} (\Omega)$ and for  any subdomain $\Omega^\prime \subset\subset \Omega$,
	\begin{align*}
	\|u\|_{C^{1,\alpha}(\overline{\Omega^\prime})} \leq C \,\{\, \| u \|_{L^{\infty} (\Omega)}  + \|f \|_{L^p (\Omega)} \}
	\end{align*}
where $C$ depends on $\,r_0,N,\varrho,\lambda,\Lambda,\alpha,  \| \gamma \|_{L^\varrho (\Omega)},\| \vartheta \|_{L^\varrho (\Omega)},\mathrm{diam} (\Omega)$, $\mathrm{dist} (\Omega^\prime,\partial\Omega)$.
	\vspace{0.02cm}
	
If in addition, $\partial\Omega\in C^{1,1}$ and $u\in C(\overline{\Omega})\cap C^{1,\tau} (\partial \Omega )$, then there exists $\alpha\in (0,\tau)$ and $\theta_0 =\theta_0 (\alpha)$, depending on $N,\varrho,\lambda , \Lambda , \|\gamma\|_{L^\varrho(\Omega)}$, so that if \eqref{H continuity} holds for some $r_0>0$ and for all $x \in \overline{\Omega}$, this implies that $u\in C^{1,\alpha}(\overline{\Omega})$ and 
	\begin{align*}
	\|u\|_{C^{1,\alpha}(\overline{\Omega})} \leq C\, \{ \,\| u \|_{L^{\infty} (\Omega)} + \|f \|_{L^p (\Omega)}  + \|u\|_{C^{1,\tau} (\partial\Omega)} \}
	\end{align*}
	where $C$ depends on $\,r_0,n,p,\lambda,\Lambda,\alpha$,$\| \gamma \|_{L^\varrho (\Omega)},\| \vartheta \|_{L^\varrho (\Omega)},\mathrm{diam} (\Omega), \partial\Omega$.
\end{prop}

\begin{proof} We first observe that our structural hypothesis \eqref{SC} takes into account a Lipschitz modulus of continuity as the zero order term. Moreover, when $\mu=0$ in \cite[Theorem 1]{regularidade}, one can perform a simpler rescaling of variable $W=N(0)$ as in \cite[Remark 3.4]{regularidade}, but instead we use directly \eqref{H continuity} (for $\beta_F$ in place of $\bar{\beta}_F$ there) and set $\widetilde{u}(x):=\frac{u(\sigma x)}{W}$ in \cite[Claim 3.2]{regularidade}. This allows us to achieve the regularity desired for $\vartheta\in L^\varrho(\Omega)$, with the estimates depending on $\|\vartheta\|_{L^\varrho(\Omega)}$.
\end{proof}

\subsection{Principal scalar eigenvalues for proper operators}

In the scalar case, for an operator $F$ satisfying \eqref{SC}, \eqref{H homogeneity}, we set
\begin{center}
	$\lambda_1^\pm\,(F(\vartheta))=\sup\left\{ \lambda \in \mathbb{R}\, , \, \Phi^\pm_\lambda\neq \emptyset\right  \},  $
\end{center}
where\vspace{-0.5cm}
\begin{align*}
&\Phi^+_\lambda=\left\{ \phi:\ \phi>0 \textrm{ in }\Omega,\; 	F[\phi] + \lambda\vartheta (x) \phi 
\leq 0 \textrm{ in }\Omega \right\},\\
&\Phi^-_\lambda=\left\{\phi:\ \phi<0 \textrm{ in }\Omega,\; 
F[\phi] +\lambda \vartheta (x)\phi 
\geq 0 \textrm{ in }\Omega \right\}.
\end{align*}
Our goal is to show that these suprema are achieved accordingly to Definition \ref{notation} \eqref{notation lambda1 scalar}, i.e.\ there exist eigenfunctions $u^\pm$ such that $-F[u]=\lambda_1^\pm u^\pm$ in $\Omega$ -- for instance when $F=F_i$, $ i=1,2$.

The first step is to deduce it, in light of \cite{regularidade}, for proper operators $F$ with unbounded drift and weight for which it holds \eqref{SC}, \eqref{H homogeneity}, \eqref{H continuity}, and \eqref{H strong}.
	
\begin{prop}\label{prop existence eigenv proper}
Let $\Omega\subset\rN$ be a bounded $C^{1,1}$ domain, $\tau\in L^\varrho(\Omega)$, $\tau\gneqq 0$, $\varrho>n$, where $F$ is a proper operator satisfying \eqref{SC}, \eqref{H homogeneity}, \eqref{H continuity}, and \eqref{H strong}, for $\gamma,\vartheta \in L^\varrho_+(\Omega)$. Then $\lambda_1^\pm>0$ and $F$ has two signed eigenfunctions $\varphi_1^\pm\in C^{1,\alpha}(\overline{\Omega})$ so that
		\begin{align*} 
\textstyle		F[\varphi_1^\pm]+\lambda_1^\pm \tau (x) \varphi_1^\pm =0 , \quad \pm \varphi_1^\pm >0 \textrm{ in } \Omega, \quad	\varphi_1^\pm = 0 \textrm{ on }\partial\Omega, \quad \max_{\overline{\Omega}}\,(\pm \varphi_1^\pm) =1.
		\end{align*}
	\end{prop}	

\begin{proof}
Let us first observe that Proposition \ref{prop existence eigenv proper} is already proved in \cite[Theorem 5.2]{regularidade} when $\vartheta$ is bounded. We stress that $C^{1,\alpha}$ regularity estimates in Proposition \ref{C1,alpha regularity estimates geral} hold for unbounded $\vartheta$.
On the other hand, we note that the solvability asked in \cite[hypothesis (H) of Theorem 5.2]{regularidade} is now ensured due to \eqref{H strong}, since in this case unique solvability of the Dirichlet problem comes from \cite[Theorem 1 (i), (ii)]{arma2010}.
Moreover, the existence result on first eigenvalues does not require the drift nor the zero order term to be bounded. Indeed, the bound \cite[(5.8)]{regularidade} is replaced by Lemma 5.7 there, with the blow-up argument comprising an unbounded zero order term as well; see ahead Step 2 in the proof of our Proposition \ref{boundedness eig BQeq}, by taking $u=v$, $\tau_1=\tau_2$.
\end{proof}

Hence, for such $F$, the following ordering holds
\begin{align}\label{relation lambda*}
\lambda_1^+(F^*(\vartheta))=\lambda_1^-(F_*(\vartheta))\le \lambda_1^+(F(\vartheta)), \lambda_1^-(F(\vartheta)) \le \lambda_1^-(F^*(\vartheta))=\lambda_1^+(F_*(\vartheta)),
\end{align}
since $F^*$ is convex and $F_*$ is concave, see \cite[Proposition 4.2]{Arms2009} and \cite[Lemma 1.1]{BQeq}.
\begin{rmk}\label{remark monotonicity}
	The following monotonicity property with respect to the weight holds:
	\begin{align}\label{monotonicity}
	\textrm{if \; $\vartheta_1\leq \vartheta_2$ \; a.e.\ in $\Omega$ \; then \; $\lambda_1^+(F(\vartheta_1))\ge \lambda_1^+(F(\vartheta_2))$.}		\end{align}
	They are instrumental in risk-sensitive control and probabilistic arguments, see \cite{controlJMPA2019, Djaironotes81}.	
\end{rmk}

\subsection{Definition of principal eigenvalues for systems}\label{Preliminaries_2}

Inspired by \cite{BNV, BMS99, BQeq}, we define the notion of principal eigenvalues for the system \eqref{LE} as follows:
\[
\lambda_1^\pm=\lambda_1^\pm\,(F_1,F_2)=\lambda_1^\pm\,(F_1(\tau_1),F_2(\tau_2)):=\sup\left\{ \lambda \in \mathbb{R}\, , \, \Psi^\pm_\lambda\neq \emptyset\right  \},
\]
where 
\smallskip
\begin{align*}
& \Psi^+_\lambda=\left\{ (\varphi, \psi); \;\varphi,  \psi>0 \textrm{ in }\Omega,\; 
	F_1[\varphi] + \lambda\tau_1(x) \psi^{q} \le 0, \; F_2 [\psi]+\lambda\tau_2(x)\varphi^{p} 
	\leq 0 \textrm{ in }\Omega \right\},
	\\
	&\Psi^-_\lambda=\left\{ (\varphi, \psi); \;\varphi,  \psi<0 \textrm{ in }\Omega,\; 
	F_1[\varphi] +\lambda\tau_1(x) |\psi|^{q-1}\psi \ge 0, \; F_2 [\psi]+\lambda\tau_2(x) |\varphi|^{p-1}\varphi  
	\geq 0 \textrm{ in }\Omega \right\},
\end{align*}
\smallskip
with inequalities holding in the $L^N$-viscosity sense, for functions $\varphi, \psi \in C(\overline{\Omega})$. When necessary, we will also highlight the dependence of $\lambda_1^\pm$ on $\Omega$. Observe that
\begin{align}\label{def G}
\lambda_1^\pm (G_1,G_2)=\lambda_1^\mp (F_1,F_2),\;
\textrm{ for }\, G_i(x,r,p,X)=-F_i(x,-r,-p,-X).
\end{align}

\begin{rmk}\label{lambda1 F1,F2 ge 0}
By \eqref{relation lambda*} and hypothesis \eqref{H lambdai>0} on $F=F_i$, we have $\lambda_1^\pm(F_i(\vartheta))>0$, $i=1,2$. 
Then we infer that $\lambda_1^\pm (F_1,F_2)\geq 0$. Indeed, to fix the ideas let us consider the $\lambda_1^+$ case. Taking the positive eigenfunctions $\varphi_1^+$ and $\psi_1^+$ associated to $\lambda_1^+(F_1(\vartheta))>0$ and $\lambda_1^+(F_2(\vartheta))>0$ respectively, we have  $(\varphi_1^+,\psi_1^+)\in \Psi_0^+$, from which the desired bound follows. 
\end{rmk}

Let us also denote 
\begin{align}\label{def m1,M1}
m_1=\min \{ \lambda_1^+(F_1, F_2),  \lambda_1^-(F_1, F_2) \}, \;\;\quad M_1=\max \{ \lambda_1^+(F_1, F_2),  \lambda_1^-(F_1, F_2) \}.
\end{align}
\subsection{Scaling and asymptotic behavior}\label{section scaling}

In this section we show some equivalent forms of problem \eqref{LE} obtained by means of a suitable scaling, and how to build a spectral curve starting from a scalar-like eigenvalue. These observations are crucial in the proof of our main results, as often it will be convenient to take $\lambda=\mu$ in \eqref{LE}. 

Take $p,q>0$ with $pq=1$, and consider:
\begin{align}\label{hyperbola mu=lambda}
-F_1[u]= \lambda \tau_1(x) |v|^{q-1}v, \quad -F_2 [v]=\lambda \tau_2(x)|u|^{p-1}u \quad\textrm{ in } \;\Omega, \quad u=v=0 \textrm{ on } \partial \Omega.
\end{align}
Let us check that to study this system for $\lambda>0$ is, in a way, equivalent to study \eqref{LE} for $(\lambda,\mu)$ in the first quadrant.

Assume we have an eigenfunction pair $(u_0,v_0)$ associated to an eigenvalue $\lambda_0>0$, i.e. $(u_0,v_0)\neq(0,0) $ viscosity solution to
\begin{align}\label{eigenpair lambda,mu}
-F_1[u_0]= \lambda_0 \tau_1(x) |v_0|^{q-1}v_0, \quad -F_2 [v_0]= \lambda_0 \tau_2(x)|u_0|^{p-1}u_0 \quad\textrm{ in } \;\Omega \quad u_0=v_0=0 \textrm{ on } \partial \Omega.
\end{align}
for some $\lambda_0>0$. We infer that this implies the existence of a curve of eigenvalues of the form $(\lambda, \mu)\in \mathbb{R}^+\times \mathbb{R}^+$, with associated eigenfunctions $u,v$ such that
\begin{align}\label{hyperbola}
-F_1[u]= \lambda \tau_1(x) |v|^{q-1}v, \quad -F_2 [v]=\mu\tau_2(x) |u|^{p-1}u \quad\textrm{ in } \;\Omega \quad u=v=0 \textrm{ on } \partial \Omega.
\end{align}
(i.e, they solve \eqref{LE}). Indeed, given $\lambda>0$, set 
\begin{equation}\label{eq:scaling}
\textstyle u=u_0,\quad v=\frac{\lambda_0^{p}}{\lambda^p} v_0,\quad \text{ and } \quad \mu=\frac{\lambda_0^{p+1}}{\lambda^p}. 
\end{equation}
By the homogeneity assumption \eqref{H homogeneity}, we see that \eqref{hyperbola} is satisfied.
In other words, given $\lambda_0>0$,  \eqref{hyperbola mu=lambda} produces a spectral curve $\Lambda_{\lambda_0}$ parametrized by
\begin{equation}\label{eq:curve_Lambda}
\textstyle \Lambda_{\lambda_0}(\lambda) = (\lambda, \mu (\lambda)), \quad \text{ where } \;\mu (\lambda)=\frac{\lambda_0^{p+1}}{\lambda^p}, \; \lambda>0.
\end{equation}
Observe that $\lambda\mapsto \mu(\lambda)$ is one-to-one, $\mu(\lambda)\to 0$ as $\lambda\to \infty$, $\mu(\lambda)\to \infty$ as $\lambda\to 0$. Moreover, $\Lambda_{\lambda_0}\cap \Lambda_{\lambda_0'}$ if $\lambda_0\neq \lambda_0'$, and $\mathbb{R^+}\times \mathbb{R}^+=\cup_{\lambda_0} \Lambda_{\lambda_0}$.

From these comments one derives that, via the relation \eqref{eq:scaling}, to study  \eqref{LE} for $(\lambda,\mu)\in \mathbb{R}^+\times \mathbb{R}^+$  is equivalent to study
\eqref{hyperbola mu=lambda}, and that statements in the Introduction can be equivalently written in terms of these scalings. Since in what follows we are going to consider almost exclusively the first quadrant of the plane $(\lambda, \mu)$ (with the exception of Theorem \ref{Lambda1pm}), we will equivalently write \eqref{LE} in the form \eqref{hyperbola mu=lambda}.

In particular, ahead in Section \ref{sec:firstmaintheorem} we will prove that the principal eigenvalues $\lambda_1^\pm=\lambda_1^\pm(F_1,F_2)$, defined in the previous section, exist and are positive, being associated with positive solutions of 
\[
-F_1[u]= \lambda_1^\pm \tau_1(x) |v|^{q-1}v, \quad -F_2 [v]=\lambda_1^\pm \tau_2(x)|u|^{p-1}u \quad\textrm{ in } \;\Omega, \quad u=v=0 \textrm{ on } \partial \Omega.
\]
From this one builds two spectral curves 
\begin{equation}\label{eq:thetwocurves}
\textstyle \Lambda^\pm_1\, (\lambda) \,=\, (\,\lambda,\, \mu_1^\pm (\lambda)\,),  \quad \text{ where }\; \mu^{\pm}_1 (\lambda)= \frac{(\lambda_1^\pm )^{p+1}}{\lambda^p}, \;\; \text{ for all } \lambda>0. 
\end{equation}
Each $\mu_1^\pm(\lambda)$ is  strictly decreasing as a function of $\lambda$, satisfying the asymptotic behavior \eqref{asymptotics} stated in Theorem \ref{Th1 introduction}. Moreover, in view of the proof of Theorem \ref{Lambda1pm}, observe that saying that $(\lambda, \mu)$ lies for instance below the curve $\Lambda_1^+$ and in the first quadrant is completely equivalent to saying that 
\[ \textstyle
\mu<\frac{(\lambda_1)^{p+1}}{\lambda^p}\;\; \Leftrightarrow\;\; \lambda_0:=(\mu \lambda^p)^\frac{1}{p+1}<\lambda_1^+.
\]

\begin{rmk}\label{remark scaling}
Just to justify other scaling prototypes we might find in the literature, instead of \eqref{hyperbola mu=lambda} we could also have written either
\begin{align*}
-F_1[u]=  \tau_1(x) |v|^{q-1}v, \quad -F_2 [v]=\lambda \tau_2(x)|u|^{p-1}u \quad\textrm{ in } \;\Omega, \quad u=v=0 \textrm{ on } \partial \Omega;
\end{align*} 
or \vspace{-0.5cm}
\begin{align*}
-F_1[u]= \lambda \tau_1(x) |v|^{q-1}v, \quad -F_2 [v]=\tau_2(x)|u|^{p-1}u \quad\textrm{ in } \;\Omega, \quad u=v=0 \textrm{ on } \partial \Omega;
\end{align*}
which are both equivalent to \eqref{hyperbola mu=lambda} whenever we are in the first quadrant.
	
On the other hand, we point out that  we may reparametrize the curve \eqref{eq:curve_Lambda} as
\[
\textstyle \Lambda_{\lambda_0}(a) =(\lambda(a),\mu (a)), \qquad \text{ where } \mu (a)=a\lambda(a), \quad \text{ for }  a=\frac{\mu}{\lambda}=\frac{\lambda_0^{p+1}}{\lambda^{p+1}}. 
\]
This way one recovers the notation and asymptotic behavior from \cite{Montenegro},
\begin{center}
$\lambda(a)= \frac{\lambda_0 }{a^{1/(p+1)}} \to 0$\, as $a\to +\infty$, \quad 
$\lambda(a)\to +\infty$\, as $a\to 0$,
\smallskip

$\mu(a)= a\lambda= a^{\frac{p}{p+1}} \lambda_0 \to +\infty$\, as $a\to +\infty$, \quad 
$\mu(a)\to 0$\, as $a\to 0$.\\
\end{center}
 \end{rmk}

\section{The scalar case with unbounded coefficients}\label{section lambda1 scalar}

We first recall that, in the case of proper operators with unbounded drift and weights, existence of positive principal eigenvalues $\lambda_1^\pm$ with associated eigenfunctions $\varphi_1^\pm$  is proved in Proposition \ref{prop existence eigenv proper}. 
In this section we start by extending the existence of eigenvalues for  nonproper operators.

\begin{lem}\label{remark nonproper scalar}
Set
$F_0[u]:=F[u]-\vartheta(x)u$, where $F$ satisfies \eqref{SC}, \eqref{H homogeneity}, \eqref{H continuity}, and \eqref{H strong}.
Then the quantity defined by\vspace{-0.3cm}
\begin{align}\label{def lambda1 nonproper}
\lambda_1^+(F({\vartheta}))\,:=\,\lambda_1^+(F_0({\vartheta}))-1 
\end{align}
is the first eigenvalue associated to a positive eigenfunction of the
scalar Dirichlet problem 
		$F[u]+\lambda \vartheta(x) u=0$ in $\Omega$,  $u=0$ on $\partial \Omega$.
Moreover, if $C_A$ is the ABP-constant in Proposition \ref{ABPproper}, then
\begin{align}\label{cotaB lambda1}
\lambda^+_1(F(\vartheta))\ge\, \frac{1}{\,C_A\,\|\vartheta\|_{L^N(\Omega)}}-1.
\end{align}
In particular, assumption  \eqref{H lambdai>0} is verified whenever $\|\vartheta\|_{L^N(\Omega)}<\frac{1}{C_A}$.

An analogous result holds for $\lambda_1^-(F(\vartheta))=\lambda_1^+(G(\vartheta))$ by applying it to $G$, see \eqref{def G}.
\end{lem}
\begin{proof}
Notice that $F_0$ is a proper operator, i.e.\ $F_0 (x,r,\xi, X)\le F_0(x,s,\xi,X)$ for $r\ge s$, since
\begin{center}
$F_0 (x,r,\xi, X)-F_0(x,s,\xi,X)\leq \vartheta (x) |r-s|-\vartheta(x)(r-s)=2\vartheta (x) (r-s)^-$ \, for $r,s\in \real$.
\end{center}
We evoke the existence and positivity of the first eigenvalue $\lambda_1^+(F_0(\hat{\vartheta}))$ for the proper fully nonlinear operator $F_0$ with unbounded drift $\gamma$ and weight $\hat{\vartheta}$ from \cite{regularidade}. Hence, by the definition of scalar eigenvalue, 
one derives the first statement. \smallskip

Now, if one writes $\vartheta = \ell \hat{\vartheta}$, for $\ell=\|\vartheta\|_{L^N(\Omega)}$ and $\|\hat{\vartheta}\|_{L^N(\Omega)} =1$, then by definition of $\lambda^+_1$, one deduces $\ell \lambda_1^+(F_0({\vartheta}))=\lambda_1^+(F_0(\hat{\vartheta}))$, and so
\begin{align}\label{eq ell hatvartheta}
\lambda_1^+(F(\vartheta))>0  \;\;\Leftrightarrow\;\; \ell<\lambda_1^+(F_0(\hat{\vartheta})).
\end{align}
Let us check that \eqref{eq ell hatvartheta} is verified if $\|\vartheta\|_{N}$ is sufficiently small. 
We claim that $\lambda_1^+(F_0(\theta))\geq C>0$ uniformly in $\theta$ whenever $\|\theta\|_{L^N(\Omega)}$ is fixed.
Indeed, since $\lambda_0=\lambda_1^+(F_0(\theta))$ is well defined and positively attained by \cite{regularidade}, then there exists a positive eigenfunction $\phi_0$ related to $\lambda_0$ such that 
\begin{center}
$F_0[\phi_0]=-\lambda_0\, \theta (x) \phi_0$, \; $\phi_0>0$\, in $\Omega$, \quad $\phi_0=0$\, on $\partial\Omega$.
\end{center} Since $F_0$ is a proper operator, then 
$\mathcal{L}^+_0[\phi_0]\ge -\lambda_0\, \theta (x)\phi_0$\,  in $\Omega^+$ (see Remark \ref{rmk ABP proper}). Therefore, ABP (Proposition \ref{ABPproper}) and $\phi_0>0$ yield the existence of a universal constant $C_A$ such that \begin{align}\label{ABP lambda1}
\textstyle \sup_{\Omega}\phi_0\le C_A\, \lambda_0\, \sup_\Omega \phi_0 \, \|\theta\|_{L^N(\Omega)} .
\end{align}
In particular we achieve \eqref{eq ell hatvartheta} by taking  $\theta=\hat{\vartheta}$ and $\ell C_A<1$. Equivalently, for $\theta=\vartheta$ we derive \eqref{cotaB lambda1} due to relation \eqref{def lambda1 nonproper}.	
\end{proof}

\begin{rmk}
The bound from below \eqref{cotaB lambda1} extends and improves \cite[Proposition 3.3]{arma2010} to the context of unbounded coefficients. In particular, with our spectral tools, the proof of \cite[Theorem 1]{arma2010} can now be considerably shorten, see also Remark in p.595 there.
\end{rmk}

We now show how scalar principal eigenvalues are related to the validity of the Alexandrov-Bakelman-Pucci estimate in the scalar case, in the following sense.

\begin{defin}\label{def ABP-MP}
Let $F$ satisfy \eqref{SC}. We say that ABP-MP (resp.\ ABP-mP) holds for $F$ in $\Omega$ if whenever $f\in L^N(\Omega)$ and $u\in C(\overline{\Omega})$ viscosity solution of $F[u]\ge f(x)$ (resp.\ $F[u]\le f(x)$) in $\Omega$, then 
\begin{align}\label{eq:ABP-MP-mP CB}
\textstyle \max_{\overline{\Omega}} u \leq C_B\{\max_{\partial \Omega} u^+ +\|f^-\|_{L^N(\Omega)} \} 
\;\; (\mathrm{resp.}\;  \min_{\overline{\Omega}} u \geq C_B\{\min_{\partial \Omega} (-u^-) -\,\|f^+\|_{L^N(\Omega)}\}),
\end{align}
for some positive constant $C_B$ not
depending on the norm of $u$. 
\end{defin}

\begin{theorem}\label{ABP-MP}
Let $\Omega$ be a bounded $C^{1,1}$ domain. Assume \eqref{SC}, \eqref{H continuity} on $F$, \eqref{H continuity}, \eqref{H homogeneity} on $F^*$.
\begin{enumerate}[(i)]
	\item If $\lambda_1^+(F^*(\vartheta))=\lambda_1^-(F_*(\vartheta))>0$ then ABP-MP holds for $F$ in $\Omega$, with $C_B$ depending on $N, \varrho,\alpha, \beta, \|\gamma\|_{\varrho}, \|\vartheta\|_{\varrho}, \mathrm{diam}(\Omega), \lambda_1^+(F^*(\vartheta))$. 
	On the other hand, if $F$ satisfies \eqref{H homogeneity}, \eqref{H strong}, and $\lambda_1^+(F(\vartheta)) \le 0$ then ABP-MP does not hold for $F$; 
	
	\item If $\lambda_1^-(F^*(\vartheta))=\lambda_1^+(F_*(\vartheta))>0$ then ABP-mP holds for $F$ in $\Omega$, with $C_B$ depending on $N, \varrho,\alpha, \beta, \|\gamma\|_{\varrho},\|\vartheta\|_{\varrho}, \mathrm{diam}(\Omega),\lambda_1^-(F^*(\vartheta))$. On the other side, if $F$ satisfies \eqref{H homogeneity}, \eqref{H strong}, and $\lambda_1^-(F(\vartheta)) \le 0$ then ABP-mP does not hold for $F$.
\end{enumerate}
\end{theorem}

Note that, in order prove ABP-mP and ABP-MP, we do not need impose $F$ verifying \eqref{H homogeneity}. This is good since, for instance, one may take $F^*$ to be $\mathcal{L}^+$, which always satisfies \eqref{H homogeneity} even when $F$ does not. Of course notice that, if $F$ satisfies \eqref{H homogeneity}, then this is also the case for $F^*$.  

\begin{lem}\label{lema convex}
Let $\Omega\in C^{1,1}$ be a bounded domain. If $F$ is either a convex or concave operator in the $X$-entry, for which it holds \eqref{SC} and \eqref{H continuity}, then $F$ satisfies \eqref{H strong}. In particular, $\lambda_1^+(F(\vartheta))$ as in Lemma \ref{remark nonproper scalar} is well defined for convex (or concave) operators satisfying only \eqref{SC}, \eqref{H homogeneity}, and \eqref{H continuity}.
\end{lem}

\begin{proof}
By \cite[Theorem 5.3]{Swiech2020} we already know that $F$ enjoys $W^{2,\varrho}$
interior regularity estimates. Thus it is enough to obtain the global statement.
Let $u$ be a viscosity solution of $F[u]=f(x)$ in $\Omega$, where $f\in L^\varrho(\Omega)$, $u=\psi$ on $\partial\Omega $ for some $\psi\in W^{2,\varrho}(\Omega)$. Then by the local regularity we know that $u$ is a strong solution. Thus, the $C^{1,\alpha}$ global regularity in Proposition \ref{C1,alpha regularity estimates geral} and the proof of Nagumo's lemma in \cite[Lemma 4.4]{regularidade} imply the desired global regularity and estimates.
\end{proof}

\begin{proof}[Proof of Theorem \ref{ABP-MP}]
We only show item (i), since (ii) is analogous. Assume $u\in C(\overline{\Omega})$ is a viscosity solution of $F[u] \geq  f(x)$ in $ \Omega$.
We first notice that, if $\lambda_1^+(F(\vartheta)) \le 0$, then ABP-MP is not satisfied. Indeed, in this situation we obtain
\[
F[\varphi_1^+]  =- \lambda_1^+\vartheta(x)\varphi_1^+ \ge 0\textrm{ in } \Omega,
\]
with $\varphi^+_1=0$ on $\partial\Omega$, but $\varphi_1^+ >0$ in $\Omega$. 
Consequently, ABP-MP does not hold in general.

\smallskip

Now set $\lambda_1^+:=\lambda_1^+(F^*(\vartheta))>0$. Let us show that this is a sufficient condition for ABP-MP.

\smallskip

Step 1) Let us check that, if there exists a solution $ \psi \in W^{2,\varrho}(\Omega)\cup C^1(\overline{\Omega})$ of
$F^* [\psi] \leq 0$ in $\Omega$ so that
$\psi > 0$ in $\overline{\Omega}$,  $\psi=1$ on $\partial\Omega$, and $\psi\in [a,b]$ for some universal constants $0<a<b$,
then $F$ satisfies ABP-MP in $\Omega$, for a constant that depends also on $a=\inf_{\Omega} \psi$ and $b=\sup_{\Omega} \psi$. 

\smallskip

Set $D:=\|\psi\|_{C^1(\overline{\Omega})}$.
Note that $v=\frac{u}{\psi}$ is a viscosity solution of $F_\psi [v] \ge f(x)$ in $\Omega$, where
\begin{align}\label{def Fpsi}
F_\psi (x,r,\xi, X):= F(x,r \psi,\, r D \psi  + \psi \xi, \, rD^2\psi  + \psi X +2 D\psi \otimes \xi).
\end{align}
The operator $F_\psi$ satisfies \eqref{SC} with ellipticity constants $a\alpha$, $b\beta$, with drift term $\gamma_\psi (x)=(2D+b)\gamma (x)$. Furthermore, $F_\psi$ is a proper operator: indeed, by applying \eqref{SC} for $F$, and \eqref{H homogeneity} for $F^*$, one finds
\begin{center}
$F_\psi (x,r,\xi,X)-F_\psi(x,r,\eta,Y)\le \mathcal{M}^+_{a\alpha,\,b\beta}\,(X-Y)+\gamma_\psi(x)|\xi-\eta|$,\vspace{0.35cm}
\\
$F_\psi (x,r,\xi,X)-F_\psi (x,s,\xi,X)\leq F^*[(r-s)\psi]=(r-s)F^*[\psi]\le 0 $\; for $r\geq s$. 
\end{center}
Now ABP for proper operators with unbounded coefficients (Proposition \ref{ABPproper} and subsequent remark) produces the estimate \eqref{eq:ABP-MP-mP} for $v=u/\psi$. In addition, $u\leq 0$ in the set where $v\le 0$, and $u=v\psi \le bv$ in the set where $v>0$, and so $u$ satisfies \eqref{eq:ABP-MP-mP CB}.

\medskip

Step 2) Now we prove that $\lambda_1^+>0$ yields the existence of a function $\psi$ as in Step 1.

Note that there exists a neighborhood of $\partial\Omega$ such that $\varphi_1^+$ attains its global maximum outside it. Moreover, as in \cite[Lemma 4.5]{BQeq}, by the Lipschitz estimate (see \cite[Theorem 2.3]{B2016} for a version of \cite[Proposition 4.9]{BQeq} for unbounded coefficients) we may take this neighborhood  $\mathcal{U}$ depending only on $N,\lambda,\Lambda,\|\gamma\|_{\varrho}, \|\vartheta\|_N, \lambda_1^\pm$, i.e.\ uniform with respect to the class of equations we consider. 
Indeed, $F^*[\varphi^+_1]+\lambda_1^+\vartheta(x)\varphi_1^+=0$ in $\Omega$, and so $\varphi_1^+$ is a viscosity positive solution of $\mathcal{L}_0^+[\varphi^+_1]\ge -(1+\lambda_1^+)\vartheta (x)\varphi^+_1$ in $\Omega$, $\varphi^+_1=0$ on $\partial\Omega$, and by \cite[Theorem 2.3]{B2016} we get 
\begin{align*}
\textstyle 1=\varphi^+_1(x_0)=\max_{\Omega} \varphi^+_1 \le C (1+\lambda_1^+)\|\vartheta\|_{N}\,\mathrm{dist}(x_0,\partial\Omega) \Rightarrow\,  \mathrm{dist}(x_0,\partial\Omega) \geq (C(1+\lambda_1^+)\|\vartheta\|_{N})^{-1},
\end{align*}
where $C$ is a universal positive constant depending only on $n,\varrho,\alpha,\beta, \|\gamma\|_{L^\varrho(\Omega)}, \Omega$.
We then take a compact set $K\subset (\rN\setminus \mathcal{U})$ 
such that $\varphi_1^+$ attains its maximum equal to $1$ in $K$ and 
\begin{center}
	$|\Omega \setminus K| \le\varepsilon :=(2C_A \|\vartheta \|_{L^\varrho(\Omega)})^\frac{-N\varrho}{\varrho-N}$
\end{center} where
$C_A$ is the constant in Proposition \ref{ABPproper}. Since $F^*$ satisfies \eqref{SC}, by \cite[Theorem 1.1(ii)]{arma2010} we may consider $w\in C(\overline{\Omega})$ a viscosity solution of the Dirichlet problem
\begin{align}\label{eq:claim visc}
\textrm{$F^*(x,0,0,D^2 w) + \gamma (x)|Dw| =f(x)$ \;in $\Omega$, \quad $w=0$ \,on $\partial\Omega$,}
\end{align}
where $f(x)= -2 \vartheta (x)$ in $\Omega\setminus K$ and $f(x)=0$ in $K$. 
Note that $w\in W^{2,\varrho} (\Omega)$ by Lemma \ref{lema convex}.

Next, we apply ABP (Proposition \ref{ABPproper}) and H\"older's inequality to find \begin{center}
	$0<w\le 2C_A \|\vartheta \|_{L^N(\Omega\setminus K)}\le 2C_A |\Omega\setminus K|^{{1}/{N}-{1}/{\varrho}}\,\|\vartheta \|_{L^\varrho(\Omega\setminus K)} \le 2C_A \varepsilon^{\frac{\varrho-N}{N\varrho}} \|\vartheta \|_{L^\varrho(\Omega)}=1$\; in $\Omega$.
\end{center} Then $w$ also solves, in the strong sense,
\begin{center}
$F^* [w+1]  \le  F^*(x,0,0,D^2 w)+ \gamma (x)|Dw|+ \vartheta (x)(w+1) 
 =  \vartheta (x)(w-1)
\le 0$ \; in $\Omega\setminus K$.
\end{center}
Now we infer that Harnack inequality gives us $\varphi_1^+\ge \eta$ on $K$, for some $\eta >0$. 
In fact, since $\varphi^+_1$ is a positive solution of the inequalities $\mathcal{L}_0^+[\varphi^+_1]+(1+\lambda_1^+)\vartheta(x)\varphi_1^+ \ge 0$ and $ \mathcal{L}^-[\varphi^+_1]\le 0$ in $\Omega$, this is a combination of the Local Maximum Principle for the nonproper operator $\mathcal{L}_0^+ +(1+\lambda_1^+)\vartheta(x)$ with unbounded zero order term (which is obtained from \cite[Theorem 2.5]{B2016} through a Moser type argument, see details in \cite[proof of Theorem 2.2]{tese}) followed by the Weak Harnack inequality for the proper operator $\mathcal{L}^-=\mathcal{L}^-_0-\vartheta $ (since $\varphi^+_1>0$) with unbounded coercive term, see \cite[Theorem 2.1]{BSvazquez}. This produces a positive constant $\eta$ depending on $n,\varrho, \alpha, \beta,\Omega, \|\gamma\|_{L^\varrho(\Omega)}, \|\vartheta\|_{L^\varrho (\Omega)}$ and $\lambda_1^+$. 

Now we set $A\lambda_1^+ \eta  = 2$ and $\psi = 1 + w + A\varphi_1^+$. Thus $1\le \psi \le 2+A=b$ in $\Omega$, $\psi=1$ on $\partial\Omega$, and $\psi$ is a strong solution of
\begin{align*}
F^* [\psi] \le  F^* [w+1]+ F^*[A\varphi_1^+] 
\le f(x)+\vartheta (x)(w+1)-A\lambda_1^+\vartheta(x) \varphi_1^+=: h(x) \le 0 \; \textrm{ in } \Omega,
\end{align*}
since $h(x)=\vartheta(x)(w+1-A\lambda_1^+ \eta)\le 0$\, in $K$. 
Note that $a=\inf_{\Omega}\psi=1$, and $b=\sup_{\Omega}\psi$ depend only upon the same constants of ABP (see for instance \cite{KSite}) since $\psi=1$ on $\partial\Omega$.
In conclusion, $F$ satisfies ABP-MP in $\Omega$, which proves the theorem.
\end{proof}

\begin{rmk}
If $F^*$ is a proper operator then Proposition \ref{ABP-MP} reduces to the usual ABP (Proposition \ref{ABPproper}). Moreover, in this case $\psi \equiv 1$ verifies the conditions in Step 1, since $F^*[1] \leq F^*[0]= 0$.
\end{rmk}

Let us now discuss some applications of Theorem \ref{ABP-MP}.
In what concerns the Dirichlet problem, the following solvability in the scalar case will be essential for the solvability of the system.

\begin{theorem}\label{Dirichlet scalar}
Let $\Omega\in C^{1,1}$ be a bounded domain. Assume \eqref{SC}, \eqref{H lambdai>0} on  $F$, and \eqref{H homogeneity} on  $F^*$. Let $f\in L^\varrho (\Omega)$, $\varrho>N$. Then there exists a viscosity solution $u\in  C^{1,\alpha}(\overline{\Omega})$ of the problem 
\begin{align}\label{solv scalar}
F(x, u, Du, D^2u)=f(x) \text{ in } \Omega, \;\;\quad u=0\text{ on } \partial \Omega .
\end{align}

Further, if \eqref{solv scalar} possesses a strong solution $u\in W^{2,\varrho}_{\mathrm{loc}}(\Omega)$, then $u$ is the unique solution of \eqref{solv scalar} in the class of viscosity solutions. In particular, \eqref{solv scalar} is uniquely solvable under \eqref{H strong}.
\end{theorem}

\begin{proof}
By Step 2 in the proof of Theorem \ref{ABP-MP} we know that there exists a function $\psi \in W^{2,\varrho}(\Omega)$ with $\psi>0$ in $\overline{\Omega}$ such that $\psi$ solves $F^*[\psi]\le 0$ in $\Omega$ in the strong sense.
Next, by Step 1 in that proof, one may define \eqref{def Fpsi}, from which we see that solving \eqref{solv scalar} is equivalent to solve $F_\psi[v]=f(x)$ in $\Omega$, where $F_\psi$ is a proper operator satisfying \eqref{SC}.
In turn, the existence of a viscosity solution $v\in C(\overline{\Omega})$ to $F_\psi[v]=f(x)$ comes from \cite[Theorem 1(ii), case $\mu=0$]{arma2010}; and the uniqueness in the presence of a strong solution follows by \cite[Theorem 1(iii), case $\mu=0$]{arma2010}. The regularity assertions are an immediate consequence of the $C^{1,\alpha}$  results in Proposition \ref{C1,alpha regularity estimates geral}.
\end{proof}

As a nontrivial application of Theorem \ref{ABP-MP} to systems, we prove MP and mP for either domains with small measure or weights with small $L^N$-norm.

\begin{proof}[Proof of Theorem \ref{MP small domain}] We only prove item $(i)$, as $(ii)$ is completely analogous. We start choosing $A>0$ such that
$
\lambda\mu^q\,  C_B^{1+q} \leq A,
$
where $C_B>0$ is the universal constant in Definition \ref{def ABP-MP}, which depends upon $N,\alpha,\beta,\mathrm{diam}(\Omega)$, $\|\gamma\|_{L^\varrho(\Omega)}$, 
in addition to $\lambda_1^+(F_i^*(\vartheta))$.

Note that $(u,v)$ is a pair of viscosity solutions to
	\begin{center}
$-F_1^* [u]\leq \lambda \tau_1 (x)(v^+)^q $, \quad $-F_2^* [v]\leq \mu \tau_2 (x)(u^+)^p $ \; in $\Omega$.
	\end{center}
Hence, Theorem \ref{ABP-MP} just proved, applied to the scalar equation for $u$ and $v$, yields
\begin{align}\label{abp uv}
\textstyle	\sup_{\Omega}u \leq  C_B\lambda\, \| \tau_1\|_{L^N(\Omega)} \sup_\Omega (v^+)^q ,
\quad	\sup_{\Omega}v \leq C_B\,\mu\, \| \tau_2\|_{L^N(\Omega)} \sup_\Omega (u^+)^p .
	\end{align}
If we had either $u\le 0$ or $v\le 0$ in $\Omega$, then by \eqref{abp uv} we would obtain $u,v\le 0$ in $\Omega$, and the proof is done. We then assume that both $u$ and $v$ assume their positive maxima in $\Omega$. 
Then, 
\begin{align*}
\textstyle 	\sup_{\Omega}(u^+)
 &\leq  C_B^{1+q} \,\lambda \mu^q\, \|\tau_1\|_{L^N(\Omega)}\,\|\tau_2\|_{L^N(\Omega)}^q\,   \sup_{\Omega}(u^+)^{pq} \\
 &\textstyle\leq  A\, \|\tau_1\|_{L^N(\Omega)}\,\|\tau_2\|_{L^N(\Omega)}^q\,   \|u\|_\infty^{pq-1} \,\sup_{\Omega}(u^+).
\end{align*}

For MP with small weights $\tau_1 , \tau_2\in L^N_+(\Omega)$, we choose $\varepsilon_0>0$ with $\|\tau_1\|_{L^N(\Omega)}\,\|\tau_2\|_{L^N(\Omega)}^q\le \varepsilon_0$ so that $A\varepsilon_0\,   \|u\|_\infty^{pq-1}\le 1/2$. Then, $u\le 0$ in $\Omega$, and so does $v$  by \eqref{abp uv}. 
Upon performing the above argument with $\sup_\Omega (v^+)$, one can assume instead $\|\tau_1\|^p_{L^N(\Omega)}\,\|\tau_2\|_{L^N(\Omega)}\le \varepsilon_0$. 

On the other hand, for $\tau_1 , \tau_2\in L^\varrho_+(\Omega)$, say $\|\tau_1\|_{L^\varrho(\Omega)}\,\|\tau_2\|_{L^\varrho(\Omega)}^q\le W$, we apply Holder inequality to obtain
	\begin{align*}
\textstyle 	\sup_{\Omega}(u)^+\leq  AW\, |\Omega|^{(\frac{1}{N}-\frac{1}{\varrho})(1+q)} \|u\|_{\infty}^{pq-1} \sup_{\Omega}(u^+).
	\end{align*}
Then we pick $\varepsilon_0>0$ with $|\Omega|\le \varepsilon_0$ such that $ AW \varepsilon_0^{(\frac{1}{N}-\frac{1}{\varrho})(1+q)}  \|u\|_{\infty}^{pq-1} \leq 1/2$, from which we also derive $u,v\le 0 $ in $\Omega$.  

The argument for supersolutions is analogous, by considering the negative parts. In any case, observe that, if $pq=1$, then $\varepsilon_0$ can be chosen independently of the $L^\infty$-norm of $u$ and $v$.
\end{proof}

\section{Auxiliary results for systems}\label{section aux}

In this section we consider some fundamental results which appear throughout the text. 
We start with an instrumental proposition to our analysis of uniqueness results.

\begin{prop}\label{th4.1 BQeq}
	Let $F_1, F_2$ satisfy \eqref{SC}, \eqref{H homogeneity}, \eqref{H lambdai>0}. 
	Let $pq=1$, $\lambda, \mu \geq 0$, and $(u_1, v_1)$, $(u_2, v_2)$ in $C(\overline{\Omega})$ be viscosity solutions of 
	\begin{align*}
	\left\{
	\begin{array}{rclcc}
	F_1 [u_1]+\, \lambda \tau_1(x) v_1^{q} &\le&0 &\mbox{in} & \;\Omega \\
	F_2 [v_1]+\,\mu \tau_2(x) u_1^p &\le&0 &\mbox{in} & \;\Omega \\
	u_1\, ,\; v_1\;&>& 0 &\mbox{in} & \;\Omega 
	\end{array}
	\right.
	\;\;,\;\;\;
	\left\{
	\begin{array}{rclcc}
	F_1[u_2]+\,\lambda \tau_1(x)|v_2|^{q-1}v_2 &\ge&0 &\mbox{in} & \;\Omega \\
	F_2 [v_2]+\,\mu \tau_2(x) |u_2|^{p-1}u_2 &\ge&0 &\mbox{in} & \;\Omega \\
	u_2\, ,\; v_2\;&\le& 0 &\mbox{on} & \;\partial \Omega.
	\end{array}
	\right.
	\end{align*}
	In addition, assume that 
	\begin{align}\label{H no MP}
	\textrm{either\; $u_2(x_0)>0$\; or\; $v_2 (x_0)>0$,\;\;\; for some $x_0 \in \Omega$;}
	\end{align}
	and that one of the  pairs of solutions is in $E_\varrho=W^{2,\varrho}(\Omega)\cap C(\overline{\Omega})$. Then $u_1 \equiv tu_2$ and $v_1 \equiv t^p v_2$ in $\Omega$ for some $t >0$. 
	
\smallskip

Analogously, if $pq=1$, $\lambda, \mu \geq 0$, and $(u_1, v_1)$, $(u_2, v_2)$ in $C(\overline{\Omega})$ satisfy
\begin{align*}
\left\{
\begin{array}{rclcc}
F_1 [u_1]+\, \lambda \tau_1(x) |v_1|^{q-1}v_1 &\ge&0 &\mbox{in} & \;\Omega \\
F_2 [v_1]+\,\mu \tau_2(x) |u_1|^{p-1}u_1 &\ge&0 &\mbox{in} & \;\Omega \\
u_1\, ,\; v_1\;&<& 0 &\mbox{in} & \;\Omega 
\end{array}
\right.
\;\;,\;\;\;
\left\{
\begin{array}{rclcc}
F_1[u_2]+\,\lambda \tau_1(x)|v_2|^{q-1}v_2 &\le&0 &\mbox{in} & \;\Omega \\
F_2 [v_2]+\,\mu \tau_2(x) |u_2|^{p-1}u_2 &\le&0 &\mbox{in} & \;\Omega \\
u_2\, ,\; v_2\;&\ge& 0 &\mbox{on} & \;\partial \Omega.
\end{array}
\right.
\end{align*}
with either
$u_2(x_0)<0$ or $v_2 (x_0)<0$, for some $x_0 \in \Omega$; and that one of the  pairs of solutions is in $E_\varrho$. Then $u_1 \equiv tu_2$ and $v_1 \equiv t^{p} v_2$ in $\Omega$ for some $t >0$. 
\end{prop}

\begin{rmk}
	The assumption \eqref{H no MP} on the pair $(u_2,v_2)$ means that the system \eqref{LE} does not satisfy the maximum principle for $(\lambda, \mu)$. 
\end{rmk}

\begin{proof}
Let us prove the first statement, since the other one is carried out similarly.

	We observe that \eqref{H no MP} implies $\sup_\Omega u_2>0$ and $\sup_\Omega v_2>0$. 
	Indeed, say $u_2(x_0)>0$, which yields $\sup_\Omega u_2>0$. If we had $v_2\leq 0$ in $\Omega$ then $-F_1^*[u_2]\leq 0$ in $\Omega$, $u_2\leq 0$ on $\partial\Omega$, so we would get $u_2\leq 0$ by ABP-MP (since we assume $F_1$ satisfies \eqref{H lambdai>0}, the first eigenvalue of $F_1^*$ is positive and so Theorem \ref{ABP-MP} can be applied). This yields a contradiction.
		
	Observe that for each compact set $K\subset\Omega$, with $x_0\in K$, there exists $s_K$ such that $u_1 > s_K u_2$, $v_1 > s_K^p v_2$ in $K$. This comes from the fact that $\min u_1$, $\min {v_1}$, $\max u_2$, $\max v_2$ are positive over $K$.
	
	Next we claim that $u_1 \geq  s u_2$, $v_1 \geq s^p v_2$ in a neighborhood of $\partial\Omega$ for some small $s>0$. 
	
	It is enough to prove the first inequality; the second one is analogous.
	Notice that $u_1-su_2\geq 0$ on $\partial\Omega$, for all $s>0$. Fix $\hat{x}\in \partial\Omega$.
	If $u_1(\hat{x})-su_2(\hat{x})> 0$, then by continuity of $u_1$ and $u_2$ up to the boundary there exists a neighborhood of $\hat{x}$, namely $\hat{B}$, such that $u_1-su_2> 0$ in $\Omega\cap \hat{B}$.
	Assume then $u_1(\hat{x})-su_2(\hat{x})= 0$ for some $s$.
	Thus $u_1(\hat{x})=u_2(\hat{x})=0$. 
	Let us look at the quantities
	\begin{center}
		$A_i =\varliminf_{t \to 0^+} \frac{u_i(\hat{x}+t\nu)-u_i(\hat{x})}{t}$, \; $i=1,2$,
	\end{center}
	where $\nu$ is the interior unit normal vector to $\partial\Omega$ at $\hat{x}$. 
	Hopf's lemma for viscosity solutions (Proposition \ref{Hopf}) yields $A_1>0$.  If we had $A_2\leq 0$, then $A_1-sA_2>0$ for all $s>0$. Otherwise, if $A_2>0$ then we may pick some small $\hat{s}>0$ such that $A_1-\hat{s} A_2>0$. Recall that one of the solutions pair is in $E_\varrho$, then one of the $A_i$'s is the normal derivative of $u_i$ at $\hat{x}$. Since $u_1(\hat{x})-su_2(\hat{x})= 0$, this is enough to ensure that $u_1-su_2>0$ in $\Omega\cap \hat{B}$.
	A covering argument then concludes the claim.
	
	Therefore one obtains the existence of some $s_0>0$ such that $u_1 \geq  s_0 u_2$ and $v_1 \geq s_0^p v_2$ in $\Omega$, for all $s\leq s_0$. In particular, the following set is nonempty,
	\begin{center}
		$ S=\{ s>0 \, : \, u_1 > su_2, \, v_1 > s^pv_2 \textrm{\, in }\Omega \} $, 
	\end{center}
	and the quantity $s_*=\sup S$\, is well defined. We have $s_*<+\infty$ by \eqref{H no MP}.
	
	\smallskip
	
	Notice that $w:=u_1 -  s_* u_2\geq 0$ and $z:=v_1 - s_*^p\, v_2\geq 0$ in $\Omega$.
	Moreover, $s_*^p\, |u_2|^{p-1}u_2\leq u_1^p $ and $s_* |v_2|^{q-1}v_2\le v_1^q$ in $\Omega$.
	By using also \eqref{SC}, \eqref{H homogeneity}, and $\lambda, \mu \geq 0$, one sees that $w,z$ satisfy
	\begin{align}\label{eq w,z}
	(F_1)_*[w] \le F_1 [u_1] - s_*\, F_1 [u_2] \le -\lambda \tau_1(x)v_1^{q} + \lambda  \tau_1(x)\,s_*\,|v_2|^{q-1}v_2 \le  0, \nonumber\\
	(F_2)_*[z]   \le F_2 [v_1] - s_*^p\, F_2 [v_2] \le - \mu \tau_2(x) u_1^p + \mu \tau_2(x) \,s_*^p \,|u_2|^{p-1}u_2 \le  0,
	\end{align}
	in the viscosity sense in $\Omega$, with $w,z\ge 0$ on $\partial\Omega$.
	Whence, by applying twice Theorem \ref{ABP-MP} and SMP for scalar equations we get either $w>0$ or $w\equiv 0$ in $\Omega$; and either $z>0$ or $z\equiv 0$ in $\Omega$.
	
	Notice that $w> 0$ is equivalent to $z> 0$ by \eqref{eq w,z}. In other words, one has either $w,z>0$ or $w,z\equiv 0$. Under the latter we are done. 
	
	Suppose on the contrary that $w,z>0$ in $\Omega$.
	Now we may reproduce the preceding argument with the pair $(u_2,v_2)$ replaced by $(w,z)$ in order to conclude the existence of some small $\varepsilon>0$ such that
	$u_1 > (s_*+\varepsilon) u_2$ in $\Omega$. But this contradicts the definition of $s_ *$ as the supremum of $S$.
\end{proof}

\begin{corol}\label{cor uniqueness}
Let $F_1, F_2$ satisfy \eqref{SC}, \eqref{H homogeneity}, \eqref{H lambdai>0}. 
Let $pq=1$, $\lambda \geq 0$, and let $(u_1, v_1)$, $(u_2, v_2)$ be viscosity solutions of 
\begin{align*}
\left\{
\begin{array}{rclcc}
F_1 [u_1]+\, \lambda \tau_1(x) |v_1|^{q-1}v_1 &=&0 &\mbox{in} & \;\Omega \\
F_2 [v_2]+\, \lambda \tau_2(x)|u_1|^{p-1}u_1 &=&0 &\mbox{in} & \;\Omega \\
u_1\, ,\; v_1\;&=& 0 &\mbox{on} & \;\partial\Omega 
\end{array}
\right.
\;\;,\;\;\;
\left\{
\begin{array}{rclcc}
F_1[u_2]+\, \Lambda\tau_1(x) |v_2|^{q-1}v_2 &=&0 &\mbox{in} & \;\Omega \\
F_2 [v_2]+\, \Lambda \tau_2(x)|u_2|^{p-1}u_2 &=&0 &\mbox{in} & \;\Omega \\
u_2\, ,\; v_2\;&=& 0 &\mbox{on} & \;\partial \Omega
\end{array}
\right.
\end{align*}
with one of the solutions pair in $E_\varrho$. Suppose that either $u_i,v_i>0$ in $\Omega$ or  $u_i,v_i<0$ in $\Omega$, $i=1,2$. 
Then $\lambda=\Lambda$, and $u_1 \equiv tu_2$, $v_1 \equiv t^p v_2$ in $\Omega$ for some $t >0$.
\end{corol}

We conclude the section presenting a priori bounds for the first eigenvalue.

\begin{lem}\label{lemma bound c=1}
Let $F_1, F_2$ satisfy \eqref{SC}, \eqref{H homogeneity}, \eqref{H lambdai>0} with $\tau_i(x)\geq \delta$ a.e.\ in $B_R$ for $i=1,2$, for some $B_R\subset\subset\Omega$. then
	$$\lambda_1^\pm (F_1(\tau_1),F_2(\tau_2),\Omega) \leq \delta^{-1} \,\lambda_1^\pm (F_1(1),F_2(1),B_R). $$
\end{lem}

\begin{proof}
We work on the $\lambda^+_1$ case, since for $\lambda_1^-$ it is just a question of replacing $F_i$ by $G_i(x,r,p,X)=-F_i(x,-r,-p,-X)$, recall \eqref{def G}.
Observe that $\lambda_1^\pm(F_i, \Omega) \leq \lambda_1^\pm (F_i, B_R)$, by definition.
Also, both quantities are nonnegative by Remark~\ref{lambda1 F1,F2 ge 0}. Hence, given \[
\mathcal{A}:=\{\lambda\in \mathcal{R},\ \Psi_\lambda^\pm (\Omega)\neq \emptyset\},\qquad \mathcal{B}:=\{\lambda\in \mathcal{R},\ \Psi_\lambda^\pm (B_R)\neq \emptyset\},
\]
it is enough to see that $\mathcal{A}\cap \{\lambda\geq 0\} \subset\mathcal{B}/\delta \cap \{\lambda\geq 0\}$, since
$$
\lambda_1^+(F_1(\tau_1), F_2(\tau_2),\Omega)=\sup_\mathcal{A} \lambda = \sup_{\mathcal{A}\cap \{\lambda\geq 0\}} \lambda\; ,\;\;  \lambda_1^+(F_1(1),F_2(1), B_R)=\sup_\mathcal{B} \lambda = \sup_{\mathcal{B}\cap \{\lambda\geq 0\}} \lambda,
$$
as settled before in Section \ref{Preliminaries_2}. Let $\lambda\in \mathcal{A}\cap \{\lambda\geq 0\}$, then there exist positive functions $\varphi,\psi\in C(\overline{\Omega})$ solving
$F_1[\varphi]+\lambda \tau_1(x)\psi^q \leq 0$, 	$F_2[\psi]+\lambda \tau_2(x)\varphi^p \leq 0$  in $\Omega$,
in the viscosity sense. Hence, $(\varphi,\psi)$ is a positive viscosity solution of 
$F_1[\varphi]+\lambda \delta\psi^q \leq 0$, $F_2[\psi]+\lambda \delta\varphi^p \leq 0$  in $B_R$, and so $\delta\lambda\in \mathcal{B}$.
\end{proof}

\begin{prop} \label{boundedness eig BQeq}
Suppose $F_1,F_2$ verify \eqref{SC}--\eqref{H strong}, $pq=1$, and $\tau_1\geq \delta>0$, $\tau_2\geq \delta>0$ a.e.\ in $B_R\subset\subset\Omega$, for some $R\leq 1$.
Let $\lambda>0$ and $(u,v)$, with $uv>0$ in $\Omega$, be a solution of 
\[
-F_1[u]= \lambda \tau_1(x) |v|^{q-1}v, \quad -F_2 [v]=\lambda \tau_2(x)|u|^{p-1}u \ \textrm{ in } \;\Omega, \quad u=v=0 \textrm{ on } \partial \Omega.
\]
Then $\lambda\leq C$, for a positive constant $C$ depending on $N,\alpha,\beta,\Omega, \|\gamma\|_{L^\varrho(\Omega)}$, $\|\vartheta\|_{L^\varrho(\Omega)}$, $\delta$, and $R$.
\end{prop}

\begin{proof}
We use some constructions from \cite{BNV, Montenegro, regularidade}.  By Lemma \ref{lemma bound c=1}, it is enough to prove the result for $\tau_i \equiv 1$, $i=1,2$. Let us consider $u>0$ and $v>0$ in $\Omega$;  for the case with $u<0$ and $v<0$ in $\Omega$ it is sufficient to apply the case with positive sign to $G_i$  in place of $F_i$, see \eqref{def G}. 

\vspace{0.02cm}

\textit{Step 1)} We first consider a bounded drift and zero order term $\gamma\in L^\infty_+(\Omega)$, namely $|\gamma(x)|\leq \gamma$ and $|\vartheta|\le \vartheta$ for a.e.\ $x\in \Omega$.
Let us prove in what follows that there exists  $C>0$ depending only on $N,\alpha,\beta,R,\gamma $, and $\vartheta$, such that
\begin{center}
	$\lambda_1^\pm (F_1(1),F_2(1),B_R)\leq\frac{C}{ R^2}.$
\end{center}
Moreover, if $\gamma, \vartheta=0$ then the constant does not depend on $R$, for all $R>0$ such that $B_R\subset \Omega$.

\vspace{0.02cm}

Note that the function $U (x)=(R^2-|x|^2)^2$ is a positive strong solution of
$F_1[U]+ \frac{C_0}{ R^2}  U \geq 0$ in $B_R$, see \cite{BNV}.
Next we consider the strong solution of
\begin{center}
$F_2[V]+U^p=0$ in $B_R$, \quad $V=0$ on $\partial B_r$,
\end{center}
given for instance in \cite{BQeq}. Here $V>0$ in $B_R$ by Theorem \ref{ABP-MP} and SMP,  since $\tau_2>0$ a.e.\ and $U>0$ in $B_R$.
Also, $\partial_\nu V>0$ on $\partial\Omega$ by Hopf, where $\nu$ is the interior unit normal.
Now, since $pq=1$, without loss of generality we may assume $q\leq 1$.
Thus $U^{1/q}\in C^1(\overline{\Omega})$, and we can pick up some $a>0$ large enough so that 
\begin{center}
$U^{1/q}\leq a V$ \;\;in $B_R$.
\end{center}
Therefore $U$ becomes a strong solution of $F_1[U]+\frac{a^q C_0}{ R^2}V^q \geq 0$ in $B_R$.

Set $C=a^q C_0 >1$.
Suppose by contradiction that there exists some $\lambda>\frac{C}{ R^2}$ such that $\varphi,\psi\in\Psi^+_\lambda(B_R)\neq\emptyset$, i.e.\ let $(\varphi,\psi)\in C(\overline{B}_R)$ be positive viscosity solutions of $F_1[\varphi]+\lambda \psi^q\leq 0$ and $F_2[\psi]+\lambda \varphi^p\leq 0$ in $B_R$. 
They also solve
$F_1[\varphi]+\frac{C}{ R^2}\psi^q\leq0$, $F_2[\psi]+\frac{C}{ R^2} \varphi^p\leq 0$ in $B_R$.
\smallskip

Now we apply Proposition \ref{th4.1 BQeq} to obtain that $\varphi = tU$ and $\psi=t^p V$ in $B_R$ for some $t>0$. However, this is not possible since $\varphi >0$ on $\partial B_R\subset\Omega$, while $U =0$ on $\partial B_R\,$.

\smallskip

\textit{Step 2)} In general case we assume $\gamma, \vartheta\in L^\varrho_+(\Omega)$. Let us show that there exists a universal constant such that
$ \lambda^{1-\frac{N}{\varrho}}\le C\, ( \|\gamma\|_{L^\varrho(\Omega)} +\|\vartheta\|_{L^\varrho(\Omega)} )^2.
$

\smallskip

We argue by contradiction. Suppose there exist sequences $\gamma_k, \vartheta_k\in L^\infty_+ (\Omega)$ with $\|\gamma_k\|_{L^\varrho(\Omega)}+\|\vartheta_k\|_{L^\varrho(\Omega)}\leq C$, but $\|\gamma_k\|_{L^\varrho(\Omega)} +\|\vartheta_k\|_{L^\infty (\Omega)}\rightarrow +\infty$, and let the respective eigenvalue problem
\begin{center}
$F_1^k [u_k]+\lambda_k v_k^{q}=0$, \; $F_2^k [v_k]+\lambda_k u_k^{p}=0$, \;\; $u_k,v_k>0$ \; in $B_R$, \;\; $u_k,v_k=0$ on $\partial B_R$,
\end{center}
in the viscosity sense, such that $\lambda_k\to +\infty$ as $k\to \infty$, where $F_i^k$ is a fully nonlinear operator satisfying \eqref{SC}--\eqref{H strong} for the respective $\gamma_k$ and $\vartheta_k$.  Up to using the rescaling $(u,v) \longmapsto (tu, t^pv)$ with $t= 1/\|u\|_{\infty}$, we may assume $\max_{\overline{\Omega}}\,u_k=1$, $\max_{\overline{\Omega}}\,v_k =  c_0$ for all $k\in\n$. Say  $\max_{\overline{\Omega}}u_k=u_k(x_0^k)$ for $x_0^k\in\Omega$. Then, $x_0^k\rightarrow x_0\in \overline{\Omega}$ as $k\rightarrow +\infty$, up to a subsequence.
	
	\vspace{0.1cm}

Since $B_R$ is a convex domain we know that $x_0\in\Omega$. Let $2\rho=\mathrm{dist}(x_0,\partial B_R)>0$, so $x_0^k\in B_\rho (x_0)$ for all $k\geq k_0$.
	Set $r_k={\lambda_k}^{-{1}/{2}}$ and $U_k(x)=U_k(x_0^k+r_k x)$, $V_k(x)=V_k(x_0^k+r_k x)$. Thus, $(U_k,V_k)$ is a viscosity solution pair of
\begin{align}\label{eq Uk,Vk}
\textrm{$\tilde{F}_1^k \,[U_k]+ V_k^{q}=0$, \; $\tilde{F}_2^k\, [V_k]+U_k^{p}=0$, \quad $U_k,V_k>0$ \; in $\widetilde{B}_k:=B_{\rho/{r_k}}(0)$,}
\end{align}
where $\widetilde{F}_i^k(x,r,p,X)=r^2_k\, F_i^k(x_0^k+r_k x,r,p/{r_k},X/{r_k^2})$ satisfies \eqref{SC}--\eqref{H strong} for $\tilde{\gamma_k}$ and $\tilde{\vartheta}_k$, where $\tilde{\gamma}_k(x)=r_k \,\gamma_k(x_0^k+r_k x) $ and $\tilde{\vartheta}_k=r_k^2 \,\vartheta_k(x_0^k+r_k x)$. Then one has
\begin{center}
$\|\tilde{\gamma}_k\|_{L^\varrho}=r_k^{1-\frac{N}{\varrho}}\|\gamma_k\|_{L^\varrho (\Omega)}\to 0$ \; and \;  $\|\tilde{\vartheta}_k\|_{L^\varrho(\Omega)}=r_k^{2-\frac{N}{\varrho}}\|\vartheta_k\|_{L^\varrho (\Omega)}\to 0$ \quad as $k\to \infty$.
\end{center} 
Also, $\sup_{\widetilde{B}_k}U_k=U_k(0)=1$ and $\sup_{\widetilde{B}_k}V_k=V_k(0)\le c_0$  for all $k\in\n$ with $B_R(0)\subset\subset \widetilde{B}_k$ for large $k$, for any fixed $R>0$.
	By $C^{1,\alpha}$ regularity estimates (Proposition \ref{C1,alpha regularity estimates geral}) we have $U_k,V_k\in C_{\mathrm{loc}}^{1,\alpha}$ and 
	\begin{align*}
	\|U_k\|_{C^{1,\alpha}(\overline{B}_R(0))}, \;	\|V_k\|_{C^{1,\alpha}(\overline{B}_R(0))}\leq  C,
	\end{align*}
since the constant depends only on a bound from above on the $L^\varrho$-norm of the coefficient $\gamma_k$, which is uniformly bounded. Hence, by compact embeddings, we have that there exists $U,V\in C^1(\overline{B}_R(0))$ such that $U_k\to U$, $V_k\to V$ as $k\to +\infty$, up to a subsequence. Doing the same for each ball $B_R(0)$, for every $R>0$, it yields $U_k\rightarrow U$, $V_k\to V$ in $L^{\infty}_{\mathrm{loc}}(\rN)$.
By applying a stability argument (Proposition \ref{stability}) in each ball, one gets that  $U,V$ is a viscosity solution of 
\begin{center}
$J_1(x,D^2U)+V^q=0$, \; $J_2(x,D^2V)+U^p=0$ \; in $\rN$,
\end{center}
for some measurable operators $J_1,J_2$ satisfying \eqref{SC} with $\vartheta$ and $\gamma$ equal to zero. The operators $J_i$ are obtained by using Arzela-Ascoli theorem, since $\widetilde{F}_i^k$ are $(\alpha,\beta)$--uniformly elliptic, with zero and first order coefficients converging to zero.
Further,  $U(0)=1$ implies $U>0$ in $\rN$ by SMP, and so $V>0$ in $\rN$. Thus, by Step 1,
\begin{center}
$1\leq \lambda_1^+ (J_1,J_2,B_r)\leq \frac{C_0}{r^2}$ for all $r>0$,
\end{center}
where $C_0$ does not depend on $r$. One derives a contradiction by letting $r\rightarrow +\infty$.
\end{proof}

\section{The first eigenvalue problem}\label{section pq=1}

In this section we investigate the main properties of the first eigenvalue problem. We stress once again that we can equivalently consider \eqref{LE}, \eqref{hyperbola mu=lambda} if positive parameters are taken into account, recall Section \ref{section scaling}. In particular, we recall that the existence of a principal eigenvalue for problem \eqref{hyperbola mu=lambda}  implies the existence of a spectral curve for problem \eqref{LE}. Moreover,  $(\lambda, \mu)$ is below the curve $\Lambda_1^\pm$ if and only if $\lambda_0 < \lambda_1^\pm(F_1,F_2)$, where $\lambda_0$ is defined through the identity $\mu=\frac{\lambda_0^{p+1}}{\lambda^p}$. Therefore, in what follows, we will always reduce ourselves to \eqref{hyperbola mu=lambda} by exploiting this scaling. 

\subsection{Existence and simplicity}\label{sec:firstmaintheorem}

We start by recalling a well known result from degree theory, see \cite{BR} for the proof.

\begin{prop}\label{KR}
Let $(\mathcal{E}, \norm{ \cdot})$ be a Banach space. Let $T \colon \mathbb{R}^+_0 \times \mathcal{E} \to \mathcal{E}$ be a completely continuous operator such that $T(0, u)=0$ for all $u \in \mathcal{E}$; then there exists an unbounded, connected component $\mathcal{C}$ of $\mathbb{R}^+ \times \mathcal{E} $ of solutions of $u=T(\mu, u)$ and starting from $(0, 0)$. 
\end{prop}

We first show the following result.
\begin{prop}\label{KR nonnegative weight}
Let $F_1, F_2$ satisfy \eqref{SC}--\eqref{H strong}, $pq=1$, and assume \eqref{H weights}. Then there exist positive numbers $\sigma_1^\pm$, and signed functions $\varphi_1^\pm, \psi_1^\pm$ in $E_\varrho=W^{2,\varrho}(\Omega)\cap C(\overline{\Omega})$ such that
\[
\left\{
\begin{array}{rclcc}
F_1 [\varphi_1^\pm]+\sigma_1^\pm\,\tau_1(x)(\psi_1^\pm)^q &=&0 &\mbox{in} & \;\Omega \\
F_2 [\psi_1^\pm]+\sigma_1^\pm \, \tau_2(x) (\varphi_1^\pm)^p &=&0 &\mbox{in} & \;\Omega \\
\pm\varphi_1^\pm\, ,\;\pm \psi_1^\pm\;&>& 0 &\mbox{in} & \;\Omega \\
\varphi_1^\pm\,\;= \,\;\psi_1^\pm\;&=& 0 &\mbox{on} & \;\partial\Omega ,
\end{array}
\right.
\]
\end{prop}
\begin{proof} We prove the $\sigma_1^+$ case; for $\sigma_1^-$ is analogous by replacing the operator $F_i$ by $G_i$, see \eqref{def G}.

\textit{Step 1)} Assume that $\tau_i(x)>0$ a.e.\ in $\Omega$ for $i=1,2$.

In this case we adapt Krein-Rutman's theorem, exploiting ideas from \cite{BR, Quaas2004}. 
Let us consider the Banach space $E=  C^1_0(\bar \Omega)$.
Define $\mathcal{F}_i=-F_i^{-1}\circ \tau_i$ in $E$, i.e.\ 
\begin{center}
	$\mathcal{F}_i \,u=U \;\;\Leftrightarrow \;\;\, -F_i\,[U]=\tau_i(x) u$\; in $\Omega$, \;\;$U=0$\; on $\partial\Omega$, \;\;$i=1,2$,
\end{center}
in the viscosity sense. Thus $\mathcal{F}_i : E \to E$ is well defined and completely continuous. 
Indeed, this comes from \eqref{H lambdai>0}, \eqref{H strong} for $F_i$, and Theorem \ref{Dirichlet scalar}.

Consider the closed positive cone $K:=\{w\in E, \, w\geq 0\}$. Note that $K$ is solid, that is, it has nonempty interior.
Then $\mathcal{F}_i$ is strictly positive with respect to $K$, in the sense that $\mathcal{F}_i\,(K\setminus\{0\})\subset \mathrm{int} K$.
This is due to ABP-MP, SMP, and Hopf for scalar equations (see Theorem \ref{ABP-MP}, and Propositions \ref{SMP}, \ref{Hopf}), since the weight $\tau_i$ is strictly positive in $\Omega$.

We first choose $w_0 \in K \setminus \{ 0 \}$. We can choose $M>0$ such that $M\mathcal{F}_2(w_0) \ge w_0$, $M\mathcal{F}_2(w_0^p) \ge w_0^p$ and $M\mathcal{F}_1(w_0) \ge w_0$. In fact, this choice is possible because $K$ is solid and each $\mathcal{F}_i$ is a strongly positive operator with respect to $K$.
Fix $\varepsilon >0$ and define $T_{\varepsilon} : \mathbb{R}^+ \times K\times K \to K\times K$ as
\[\textstyle{ T_{\varepsilon} (\mu, u, v)=\left(\mu \mathcal{F}_1(v^q) + \varepsilon \mu \mathcal{F}_2(w_0), \mu \mathcal{F}_2(u^p)+ \varepsilon \mu \mathcal{F}_1(w_0)\right).} \]
By Proposition \ref{KR}, there exists $\mathcal{C}_{\varepsilon}$, an unbounded connected component of solutions of $(u, v)=T_{\varepsilon}(\mu, u, v)$ which contains $(0,0,0)$. 
We claim that $\mathcal{C}_{\varepsilon} \subset [0, M] \times K\times K$. Indeed, let $(\mu, u, v) \in  \mathcal{C}_{\varepsilon}$. In particular, 
\[ \textstyle{ u=\mu \mathcal{F}_1(v^q) + \varepsilon \mu \mathcal{F}_2(w_0) \quad  \textrm{and}\quad  v= \mu \mathcal{F}_2(u^p)+ \varepsilon \mu \mathcal{F}_1(w_0).} \]
Then,
\[ \textstyle{u\ge \mu \varepsilon \mathcal{F}_2(w_0) \ge \frac{\mu}{M} \varepsilon w_0, } \]
which implies that $u^p \ge \frac{\mu^p}{M^p} \varepsilon^p w_0^p$. By applying $\mathcal{F}_2$, and using the comparison principle, see Theorem \ref{ABP-MP}, we have 
\[ \textstyle{\mathcal{F}_2(u^p) \ge \frac{\mu^p}{M^p} \varepsilon^p \mathcal{F}_2(w_0^p) \ge \frac{\mu^p}{M^p} \varepsilon^p  \frac{1}{M} w_0^p.} \] 
Moreover, 
$v \ge \mu \mathcal{F}_2(u^p) \ge \frac{\mu^{p+1}}{M^{p+1}} \varepsilon^p w_0^p, $
whence $v^q \ge \frac{\mu^{q+1}}{M^{q+1}}\varepsilon w_0$ since $pq=1$. Now, applying $\mathcal{F}_1$, 
\[\textstyle{ \mathcal{F}_1(v^q) \ge \frac{\mu^{q+1}}{M^{q+1}}\varepsilon \mathcal{F}_1(w_0) \ge \frac{\mu^{q+1}}{M^{q+1}}\varepsilon \frac{1}{M} w_0,}  \]
and therefore $u \ge \mu \mathcal{F}_1(v^q) \ge \left( \frac{\mu}{M}\right)^{q+2} \varepsilon w_0$. Upon iteration, one gets 
\[ \textstyle{u \ge \left( \frac{\mu}{M} \right)^{\alpha_k} \varepsilon w_0 \;\; \textrm{ for all $k \ge 2$,\quad where\; $\alpha_k=k(q+1)+1$. }}\]
This gives us $\mu \le M$, and the claim is proved. 

Therefore, there exists $(\mu_{\varepsilon}, u_{\varepsilon}, v_{\varepsilon}) \in \mathcal{C}_{\varepsilon}$ such that $\norm{(u_{\varepsilon}, v_{\varepsilon})}_{\infty}=1$. 
By compactness of $\mathcal{F}_1$ and $ \mathcal{F}_2$, Theorem \ref{ABP-MP}, and SMP, by taking $\varepsilon \to 0$ we conclude that there exists $\sigma_1>0$, in addition to $u_1, v_1$ such that $\norm{(u_1, v_1)}_{\infty}=1$, 
\begin{center}
	$ u_1 =\sigma_1 \mathcal{F}_1(v_1^q) \quad  \textrm{and}\quad
 v_1=\sigma_1 \mathcal{F}_2(u_1^p).$
\end{center} 
The eigenfunctions belong to the interior of the positive cone $K$, by the positivity of the operators $\mathcal{F}_i$ for $i=1,2$. Then the result follows in the case of positive weights.

\smallskip

\textit{Step 2)} In the general case we argue by approximation.
Since $\tau_1,\tau_2>0$ a.e.\ in a common set of positive measure -- assumption \eqref{H weights} -- then there exists some $\delta>0$ such that $| \{\tau_1\geq \delta\}\cap \{\tau_2\geq \delta\} | >0$. 
Say $\tau_1,\tau_2\geq \delta$ a.e.\ in some ball $B_R\subset\subset\Omega$, $R\leq 1$. 
Take $\varepsilon\in (0,1)$ and define $\tau_i^\varepsilon:=\tau_i+\varepsilon >0$ in $\Omega$. By \textit{Step 1} we obtain the existence of $\sigma_1^\varepsilon>0$ and $\varphi_1^\varepsilon, \psi_1^\varepsilon\in C^{1}(\overline{\Omega})$ such that
\begin{align*}
F_1[\varphi_1^\varepsilon]+\sigma_1^\varepsilon \,\tau_1^\varepsilon(x) (\psi_1^\varepsilon)^q=0, \quad
F_2[\psi_1^\varepsilon]+ \sigma_1^\varepsilon \, \tau_2^\varepsilon(x) (\varphi_1^\varepsilon)^p = 0  \;\; \textrm{ in } \Omega \\
	\varphi_1^\varepsilon, \psi_1^\varepsilon>0  \;\;\textrm{ in } \Omega, \quad
	\varphi_1^\varepsilon , \psi_1^\varepsilon = 0 \;\;\textrm{ on } \partial\Omega,\quad  \textstyle{\max_{\overline{\Omega}}\,\varphi_1^\varepsilon=1},\;\; \textstyle{\max_{\overline{\Omega}}\,\psi_1^\varepsilon=1}.
\end{align*}
Then, by Proposition \ref{boundedness eig BQeq},
$0<\sigma_1^\varepsilon\leq C_0$ for all $\varepsilon\in (0,1)$. So $\sigma_1^\varepsilon\rightarrow\sigma_1 \in [0,C_0]$ up to a subsequence. Then, applying $C^{1,\alpha}$ global regularity (Proposition \ref{C1,alpha regularity estimates geral}), yields
	\begin{align*}
	\|\varphi_1^\varepsilon\|_{C^{1,\alpha}(\overline{\Omega})} \leq C\,\{\,\|\varphi_1^\varepsilon\|_{L^\infty(\Omega)} +\sigma_1^\varepsilon\, \|\tau_1^\varepsilon\|_{L^\varrho(\Omega)} \,\|\varphi_1^\varepsilon \|^q_\infty \,\}
	\leq C \,(\|\tau_1\|_{L^\varrho(\Omega)} +1) \,\} \leq C, 
	\end{align*}
and analogously $\|\psi_1^\varepsilon\|_{C^{1,\alpha}(\overline{\Omega})} \leq C$. Hence the compact inclusion $C^{1,\alpha}(\overline{\Omega})\subset C^1(\overline{\Omega})$ yields $\varphi_1^\varepsilon \rightarrow \varphi_1$ and $\psi_1^\varepsilon \rightarrow \psi_1$ in $ C^1(\overline{\Omega})$, up to a subsequence, with $\max_{\overline{\Omega}}\,\varphi_1=1$, $\max_{\overline{\Omega}}\,\psi_1=1$, $\varphi_1, \psi_1\geq 0$ in $\Omega$, and $\varphi_1,\psi_1=0$ on $\partial\Omega$. 
Since $\tau_i^\varepsilon\rightarrow \tau_i$ in $L^\varrho(\Omega)$ as $\varepsilon\rightarrow 0$, by stability of viscosity solutions we derive that $\varphi_1$, $\psi_1$ is an $L^N$-viscosity solution pair of $F_1[\varphi_1]+\sigma_1 \tau_1(x)\psi_1^q=0$, $F_2[\psi_1]+\sigma_1 \tau_2(x)\varphi_1^p=0$ in $\Omega$, which allows us to apply $C^{1,\alpha}$ regularity again to obtain that $\varphi_1,\psi_1\in C^{1,\alpha}(\overline{\Omega})$.

Next, by Theorem \ref{ABP-MP} and SMP, we have that $\varphi_1>0$ in $\Omega$; similarly $\psi_1>0$. 
Moreover, we must have $\sigma_1>0$. Indeed, $\sigma_1=0$ produces $F_1[ \varphi_1]= 0$ in $\Omega$ which in turn would give us $\varphi_1\leq 0$ in $\Omega$ by MP, since we are assuming $\lambda_1^+(F_1^*)>0$. Using the regularity property in \eqref{H strong}, $\varphi_1$ and $\psi_1$ turn out to be strong solutions.  
\end{proof}

\begin{proof}[Proof of Theorem \ref{Th1 introduction}]
\textit{Existence.} We preliminarily notice that $\lambda_1^+$ is finite due to Proposition \ref{boundedness eig BQeq}. Take $\sigma_1^\pm$ as in Proposition \ref{KR nonnegative weight}.
Let us prove that $\sigma_1^\pm =\lambda_1^\pm$. 
By definition of $\lambda_1^+$, we have $\sigma_1^+ \le \lambda_1^+$. Suppose there exists $\varepsilon>0$ such that $\sigma_1^+ < \lambda_1^+ - \varepsilon$. Then, by definition of $\lambda_1^+$, we can take $\varphi, \psi>0$ such that 
\[ F_1[\varphi] + (\lambda_1^+ - \varepsilon)\tau_1(x)\psi^{q} \le 0, \quad F_2[\psi] +(\lambda_1^+ - \varepsilon) \tau_2(x) \varphi^{p} \le 0. \]
Since $ F_1[ \varphi_1^+] + (\lambda_1^+ - \varepsilon)\tau_1(x) (\psi_1^+)^q > F_1[\varphi_1^+] +\sigma_1^+\tau_1(x) (\psi_1^+)^q = 0$, then
by Proposition \ref{th4.1 BQeq} (and up to suitable rescaling as in Section \ref{section scaling}) we have $\varphi= t \varphi_1^+$ and $\psi=t^p \psi_1^+$ for a suitable $t>0$, which is a contradiction. Similarly, we obtain $\sigma_1^-=\lambda_1^-$. 
\smallskip

\textit{Simplicity.}
Let us prove the case of $\lambda_1^+$, since for $\lambda_1^-$ is analogous. Assume that $(u,v)$ is another eigenfunction corresponding to $\lambda_1^+$, with $u\not\equiv 0$ or $v\not\equiv 0$. 
 
In the case that $u$ or $v$ attains a positive maximum in $\Omega$, then we consider the positive eigenfunctions pair $(\varphi_1^+,\psi^+_1)$ of the operators $(F_1,F_2)$. Then we apply Proposition \ref{th4.1 BQeq} to both pairs $u,v$ and $\varphi_1^+,\psi^+_1$ to obtain that $u\equiv t \varphi_1^+$ and $v\equiv t^p \psi^+_1$ in $\Omega$ for a suitable $t  \in \mathbb{R}$. 

If, in turn, $u\leq 0$ and $v\leq 0$ in $\Omega$, then $u<0$ and $v<0$ in $\Omega$ by the strong maximum principle for scalar equations. Notice that if one between $u$ and $v$ is $\equiv 0$, then also the other one has to be null. Then we consider the negative eigenfunctions pair $(\varphi_1^-,\psi_1^-)$ of the operators $(F_1,F_2)$.
Therefore we use the fact that $\lambda_1^-$ is the only positive eigenvalue corresponding to a negative eigenfunction.
In other words, we apply Corollary \ref{cor uniqueness} to obtain that $\lambda_1^+=\lambda_1^-$, as well as $u\equiv t \varphi^-_1$ and $v\equiv t^p \psi_1^-$ in $\Omega$. 
\end{proof}

\subsection{Maximum principles}

\begin{proof}[Proof of Theorem \ref{Lambda1pm}] We prove the statement regarding MP; the case mP is analogous.

\smallbreak

Step 1) Let us first assume that $(\lambda,\mu)\in \mathbb{R}^+\times \mathbb{R}^+$. Then, via the scaling \eqref{eq:scaling}, we reduce the problem to check if MP holds for
\begin{equation}\label{Theorem1.3_aux}
-F_1[u]= \lambda_0 \tau_1(x) |v|^{q-1}v_0, \quad -F_2 [v]= \lambda_0 \tau_2(x)|u|^{p-1}u_0 \quad\textrm{ in } \;\Omega \quad u=v=0 \textrm{ on } \partial \Omega.
\end{equation}
if and only if $\lambda_0<\lambda_1^+$. We first notice that, if $\lambda_0 \geq  \lambda_1^+$, then MP is not satisfied. Indeed, when $\lambda_0\geq \lambda_1^+$ we get 
\[
F_1[\varphi_1^+] + \lambda_0\tau_1(x)(\psi_1^+)^q =- \lambda_1^+\tau_1(x)(\psi_1^+)^q +\lambda_0 \tau_1(x)(\psi_1^+)^q\ge 0, \quad
F_2[\psi_1^+] +\lambda_0  \tau_2(x) (\varphi_1^+)^p \ge 0 \textrm{ in } \Omega,
\]
with $\varphi^+_1,\psi_1^+=0$ on $\partial\Omega$, but $\varphi_1^+ >0,\, \psi_1^+>0$ in $\Omega$.

Now we prove that MP holds if $0 <  \lambda_0 < \lambda_1^+$.
Let $(u, v)$ be a viscosity solution of 
\[
	F_1 [u]+\lambda_0 \tau_1(x)|v|^{q-1} v  \geq 0 , \;\;\;
	F_2 [v]+\lambda_0\tau_2(x) |u|^{p-1}u \geq 0 \;\;\textrm{ in } \Omega ,\quad\textrm{$u,v\leq 0$ on $\partial\Omega$}.
\]
 Assume by contradiction that one between $u,v$ has a positive maximum in $\Omega$.
Observe that $\varphi_1^+,\psi_1^+$ is a strong solution pair of
\begin{center}
$F_1 [\varphi_1^+]+\, \lambda_0 \tau_1(x)(\psi_1^+)^q \le 0$,\;\; $F_2 [\psi_1^+]+\lambda_0 \tau_2(x)(\varphi_1^+)^p \le 0$\; in $\Omega$,\;\; $\varphi_1^+,\psi_1^+>0$ in $\Omega$.
\end{center}
Then Proposition \ref{th4.1 BQeq} yields $u=t\varphi_1^+$, $v=t^p \psi_1^+$ in $\Omega$, which contradicts $\lambda_0 \ne \lambda_1^+$. So $u, v \le 0$ in $\Omega$. 

\smallbreak

Step 2) Let us now check that, if $(\lambda,\mu)\notin \mathbb{R}^+\times \mathbb{R}^+$, then MP does not hold. If $\lambda=0$ or $\mu=0$, then we reduce ourselves to the scalar case, and the conclusion follows from \cite{BQeq} in the case of bounded drifts, while for unbounded ones this is a new contribution of this paper, see Section \ref{section lambda1 scalar}. 
The only thing which is left to prove is that if at least one between $\lambda, \mu$ is negative, then the MP and the mP do not hold. Assume for instance $\lambda <0$. Consider the eigenvalue problem
\begin{align*}
F_1 [\tilde \varphi]+\lambda_0 \tau_1(x)\tilde \psi^q  =0 ,\quad
(F_2)_*[ \tilde \psi] +\, \lambda_0 \tau_2(x)\tilde \varphi^p =0 \textrm{ in } \Omega,
\end{align*}
where $\tilde \varphi, \tilde \psi>0$ in $\Omega$,  $\tilde \varphi, \tilde \psi=0$ on $\partial\Omega$, $\lambda_0=\lambda_1^+(F_1, (F_2)_*)$. Let $\bar{\lambda}=1/s^q$, $\bar{\mu}=s \lambda$. Hence, choosing a suitable $s>0$, we have $\bar{\lambda}>0$ and $\bar{\mu}<0$ such that  $\bar{\mu} \le -\lambda_0$. Hence, 
\[
\begin{cases}
F_1[\tilde \varphi] + \bar{\lambda}\tau_1(x) | \tilde \psi|^{q-1}(-\tilde{\psi})= F_1[\tilde \varphi] - \bar{\lambda}\tau_1(x) (\tilde \psi)^q=-\tau_1(x)\,(\lambda_0 +\bar{\lambda} ) \, (\tilde \psi)^q\le 0, \\
F_2[-\tilde \psi] + \bar{\mu} \tau_2(x) (\tilde \varphi)^p \le -(F_2)_*[\tilde \psi] + \bar{\mu} \tau_2(x)(\tilde \varphi)^p=\tau_2(x)\,(\lambda_0 + \bar{\mu} )\,  \varphi)^p \le 0.
\end{cases}
\]
However, $-\tilde \psi <0$ in $\Omega$.
\end{proof}

\subsection{The Dirichlet problem for $\lambda<m_1$}\label{section lambda<m1}

Now we focus on the nonhomogeneous Dirichlet Lane-Emden problem, proving item $(i)$ of Theorem \ref{ThDir solvability intro}. We postpone the proof of items $(ii)$ and $(iii)$ to Section \ref{section pq<1} ahead, as they turn out to follow from an adaptation of the proof of Theorem \ref{sublinear}. Recall $m_1, M_1$ from \eqref{def m1,M1}.
We consider functions $f_i\in L^\varrho(\Omega)$. Via the scaling \eqref{eq:scaling}, we reduce our problem to the study of system
\begin{align}\label{Dir}
\left\{
\begin{array}{rclcc}
F_1 [u]+\lambda\tau_1(x) |v|^{q-1}v  &=&f_1(x) &\mbox{in} & \;\Omega, \\
F_2 [v]+\lambda \tau_2(x)|u|^{p-1}u &=&f_2(x) &\mbox{in} & \;\Omega, \\
u\,\;= \,\;v\;&=& 0 &\mbox{on} & \;\partial\Omega.
\end{array}
\right.
\end{align}
and the condition $(\lambda,\mu)\in \mathcal{C}_1^+\cap \mathcal{C}_1^-$ for \eqref{LE} translates to $0<\lambda<m_1$ for \eqref{Dir}.

\begin{proof}[Proof of Theorem \ref{ThDir solvability intro} (i)]
We sketch the proof in light of \cite{LeiteMontenegro}. One first obtains the following a priori bounds.
\begin{claim}\label{claim Dir}
Let $(u, v)$ be a viscosity solution of \eqref{Dir}. Then 
\begin{center}
	$ \norm{u}_{\infty} + \|{v}\|_{\infty}^p \le C \{\,\|{f_1}\|_{L^N(\Omega)}+\|{f_2}\|_{L^N(\Omega)}\} , $ \; for all $0 < \lambda < m_1$.
\end{center}
\end{claim}

In order to see this, let us assume by contradiction that there exist solutions such that
\[ \textstyle{\norm{u_k}_\infty + \norm{v_k}^p_\infty > k\,\{\,\|{f_1^k}\|_{\varrho}+\|{f_2^k}\|_{\varrho}\}. }\]
We normalize as in \cite[Section 5]{LeiteMontenegro} to get new functions $\tilde u_k, \tilde v_k, \tilde f_1^k, \tilde f_2^k$  such that
\[ \textstyle{ \norm{\tilde u_k}_\infty + \norm{\tilde v_k}^p_\infty =1, \quad \|{\tilde f_1^k}\|_\varrho,\; \|{\tilde f_2^k}\|_\varrho < 1/k. }\]
The limit functions $\tilde u, \tilde v$ solve \eqref{LE} for $\lambda=\mu$, and they are nonzero due to our normalization. Hence, they are an eigenfunction, however $\lambda < m_1$, a contradiction. So Claim \ref{claim Dir} is true.

\smallskip 

Define $H(t, u, v)=(U, V)$ as the viscosity solution to the problem 
\begin{align*}
\left\{
\begin{array}{rclcc}
F_1\, [U]+\tau_1(x) t\lambda |v|^{q-1}v  &=&tf_1(x) &\mbox{in} & \;\Omega ,\\
F_2\, [V]+\tau_2(x) t\lambda |u|^{p-1}u &=&tf_2(x)&\mbox{in} & \;\Omega, \\
U\,\;= \,\;V\;&=& 0 &\mbox{on} & \;\partial\Omega,
\end{array}
\right.
\end{align*}
in the space $C^1(\overline{\Omega})$.
By \eqref{SC}-\eqref{H strong} for $F_1,F_2$, together with regularity-estimates $C^{1,\alpha}$ theory of viscosity solutions, the map $H$ is well defined and completely continuous. Observe that $t\lambda < \min \{ \lambda_1^-, \lambda_1^+ \}$ for any $t \in [0, 1]$. By using this and the estimates above, we have that any $(u, v)$ which satisfies $H(t, u, v)=(u, v)$ is bounded by a constant not depending on $t$. By degree fixed point theory we get a solution to our problem. 
\end{proof}

\subsection{Local isolation}

\begin{proof}[Proof of Theorem \ref{Th isolation Introdu}]
Once again, via the scaling \eqref{eq:scaling}, we reduce the problem to the study of \eqref{LE} with $\lambda=\mu$, that is \eqref{hyperbola mu=lambda}. Recall $M_1, m_1$ from \eqref{def m1,M1}. For $0<\lambda<M_1$ then the system \eqref{LE} with $\mu=\lambda$ satisfies MP or mP depending on whether $M_1=\lambda_1^+$ or $\lambda_1^-$. Suppose that $u,v \in C(\bar{\Omega})$ is a nontrivial eigenfunction pair associated to the eigenvalue $\lambda$, and let $M_1=\lambda_1^-$. 
If $\lambda < \lambda_1^-$ then \eqref{LE} with $\lambda=\mu$ satisfies mP. Therefore $u,v \geq 0$ in $\Omega$, and $u,v > 0$ in $\Omega$ by Hopf. By Proposition \ref{th4.1 BQeq} one has $\lambda=m_1= \lambda_1^+$. Analogously we deduce that if $M_1= \lambda^+_1$ then $\lambda = \lambda^-_1$.
In particular, no eigenfunctions exist on the left of $M_1$, except from the one corresponding to $m_1$.

Let us show that there is a neighborhood on the right of $M_1$ where eigenfunctions do not exist.
We assume w.l.g.\ that $M_1=\lambda_1^-$ (the other case is analogous). Let us take a sequence of positive eigenvalues $\lambda_n =\lambda_1^- +\varepsilon_n$ related to normalized eigenfunctions $(u_n, v_n)$ such that $\varepsilon_n \to 0$. 
Then by stability of viscosity solutions (Proposition \ref{stability}), $(u_n, v_n) \to (u, v)$ eigenfunction related to $\lambda_1^-$. 
Then, by Proposition \ref{th4.1 BQeq} one concludes $u=t \varphi_1^-$ and $v=t^p \psi_1^-$ for some $t >0$.  Note that by Krein-Rutman theorem $\varphi_1^-,\psi_1^-$ belong to the interior of the cone of negative solutions.  Then $u_n, v_n<0$ for large $n$. This implies by Proposition \ref{th4.1 BQeq} that $\lambda_n = \lambda_1^-$, which is a contradiction. 
\end{proof}

\section{The anti-maximum principle}\label{sec:antimax}

In this section we go along with the validity of the maximum principle, this time when it fails and takes the form of an anti-maximum principle. We move towards the proof of Theorem \ref{AMP}. Our approach relies on some arguments in \cite{Arms2009, IY}. Once again, we reduce ourselves to one of the equivalent forms of \eqref{LE}, and in particular we will always assume that the system is written in the form \eqref{hyperbola mu=lambda}, see Section \ref{section scaling}. 
 
We start with an auxiliary nonexistence result. We assume $F_i$ satisfies \eqref{SC}--\eqref{H lambdai>0}, $i=1,2$.

\begin{lem}\label{PDEnonexistence}
Let $\tau_1,\tau_2\in L^\varrho_+(\Omega)$ for some $\varrho>N$, and $f_i \in L^\varrho(\Omega)$ such that $f_i \not\equiv 0$. 
\begin{enumerate}[(i)]
\item If $\lambda \geq \lambda_1^+$ and $f_i \leq 0$ then there is no nonnegative solution $u,v\in C(\overline{\Omega})$ of
\begin{align}\label{eq PDEnonexistence1}
F_1 [u]+\lambda\tau_1(x) |v|^{q-1}v  \leq f_1, \;\;\;
F_2 [v]+\lambda \tau_2(x)|u|^{p-1}u \leq f_2\;\; \textrm{ in } \;\Omega,\quad u,v\geq 0 \textrm{ on } \partial\Omega.
\end{align}
If in addition $\lambda_1^+ \leq \lambda \leq \lambda_1^-$, the problem \eqref{eq PDEnonexistence1} does not possess a solution $u,v\in C(\overline{\Omega})$.

\item On the other hand, if $\lambda \geq \lambda_1^-$ and $f_i \geq 0$ there is no nonpositive solution $u,v \in C(\overline{\Omega})$ of
\begin{equation}\label{eq PDEnonexistence2}
F_1 [u]+\lambda\tau_1(x) |v|^{q-1}v  \geq f_1, \;\;\;
F_2 [v]+\lambda \tau_2(x)|u|^{p-1}u \geq f_2\;\; \textrm{ in } \;\Omega,\quad u,v\leq 0 \textrm{ on } \partial\Omega.
\end{equation}
Moreover, if $\lambda_1^- \leq \lambda \leq  \lambda_1^+$, the problem \eqref{eq PDEnonexistence2} does not possess a solution $u,v \in C(\overline{\Omega})$.
\end{enumerate}
\end{lem}

\begin{proof} We prove just $(i)$, since the case $(ii)$ is analogous. Assume by contradiction that there exists a solution pair $u,v\geq 0$ of \eqref{PDEnonexistence} with $\lambda \geq \lambda_1^+$ and $f_i \leq 0$. 
Note that the case $u,v \equiv 0$ is not allowed by the hypothesis $f_i \not \equiv 0$.
Thus, assume that $u\not \equiv 0$ or $v\not \equiv 0$. Thus by SMP we have $u,v > 0$ in $\Omega$. The definition of $\lambda^+_1$ implies $\lambda \leq \lambda^+_1$, from which we deduce $\lambda=\lambda^+_1$. By Proposition \ref{th4.1 BQeq} we obtain that $u=t\varphi^+_1$ and $v=t^p\psi_1^+$ for some $t>0$. However, this leads to $f_i \equiv 0$, $i=1,2$, which contradicts the hypoteses.

Now assume in addition that $\lambda \leq \lambda^-_1$, and on the contrary that there exists a solution $u,v$ of \eqref{eq PDEnonexistence1}. By what we just proved, either $u$ or $v$ must be negative somewhere in $\Omega$. Applying Proposition \ref{th4.1 BQeq} we get that $u\equiv t\varphi^-_1$ and $v\equiv t^p\psi^-_1$ in $\Omega$ for some $t > 0$. This yields $\lambda = \lambda^-_1$, and so $f_i \equiv 0$, $i=1,2$, again a contradiction.
\end{proof}

\begin{proof}[Proof of Theorem \ref{AMP}]
We prove only $(i)$; $(ii)$ is similar. In order to get a contradiction, suppose that there exist values $\lambda_k$ above $\lambda_1^-$,  and $u_k, v_k\in C(\overline{\Omega})$ such that $\lambda_k  \to \lambda_1^-$, and $u_k,v_k\in C(\overline{\Omega})$ satisfying
\begin{align}\label{AMB-false-k}
\left\{
\begin{array}{rclcc}
F_1 [u_k]+\lambda_k\tau_1(x) |v_k|^{q-1}v_k  &=&f_1(x) &\mbox{in} & \;\Omega \\
F_2 [v_k]+\lambda_k\tau_2(x)|u_k|^{p-1}u_k &=&f_2(x) &\mbox{in} & \;\Omega \\
u_k\,\;= \,\;v_k\;&=& 0 &\mbox{on} & \;\partial\Omega
\end{array}
\right.
\end{align}
such that at least one between $u_k, v_k$ is nonnegative somewhere in $\Omega$. It then turns out that $u_k$ or $v_k$ is nonnegative somewhere for infinite $k$'s, say $u_k$. Thus, take such $x_k\in \Omega$ where $u_k$ attains a nonnegative maximum at $x_k$. In particular, for this subsequence one has
	\begin{equation}\label{sequenceantimax1}
\textstyle{	u_k(x_k) \geq 0 \quad \mbox{and} \quad Du_k(x_k) = 0.}
	\end{equation}
By taking a further subsequence we may assume $	x_k \rightarrow x_0$ for some $x_0\in \overline{\Omega}$. We claim that
	\begin{equation}\label{unbounded uk or vk}
\textstyle{\mbox{either}\quad\sup_{k} \| u_k \|_{L^\infty(\Omega)} = +\infty \quad\mbox{ or } \;\quad\sup_{k} \| v_k \|_{L^\infty(\Omega)} = +\infty.}
	\end{equation}
Otherwise, if $ \| u_k \|_{L^\infty(\Omega)} , \| v_k \|_{L^\infty(\Omega)}\le C$ for all $k$, then by $C^{1,\alpha}$ estimates and compact inclusion we obtain $u,v \in C^1(\overline{\Omega})$ such that $u_k \rightarrow u$  and $v_k \rightarrow v$ uniformly in $\overline{\Omega}$. By stability of viscosity solutions we may pass to limits in \eqref{AMB-false-k}, then $u,v$ become solution of the problem
\begin{align*}
F_1 [u]+\lambda_1^- \tau_1(x) |v|^{q-1}v  =f_1(x) , \quad
F_2 [v]+\lambda_1^-\tau_2(x)|u|^{p-1}u =f_2(x)\;\; \mbox{ in }  \;\Omega , \quad u=v=0 \mbox{ on }  \;\partial\Omega.
\end{align*}
But this contradicts Lemma \ref{PDEnonexistence} (i), whence \eqref{unbounded uk or vk} is proved.

Let us assume w.l.g. that $\norm{u_k}_{\infty} \to \infty$ (up to a subsequence), and let us define $\theta_k=\|u_k\|_\infty$. We also consider the rescaling $u_k = \theta_k U_k$ and $v_k = \theta_k^{p}\, V_k$. Since the operators $F_1,F_2$ are positively $1$-homogeneous,
$$
\textstyle -F_1 [U_{k}] = - \frac{F_1 [u_{k}]}{\theta_{k}}  
  = \frac{1}{\theta_{k}} \{ \lambda_k\tau_1(x) |v_k|^{q-1}v_k-f_1(x) \}  = \lambda_k\tau_1(x)  |V_k|^{q-1}V_k-\frac{f_1(x)}{\theta_k},
$$
and
$$
\textstyle -F_2 [V_{k}]=  -\frac{F_2 [v_{k}]}{\theta_{k}^{p}} 
= \frac{1}{\theta_{k}^{p}}\{ \lambda_k \tau_2(x) |u_k|^{p-1}u_k-f_2(x)\}  = {\lambda_k\tau_2(x)} |U_k|^{p-1}U_k -\frac{f_2(x)}{\theta_k^p} ,
$$
with $U_k=V_k=0$ on $\partial\Omega$.
Again by $C^{1,\alpha}$ estimates, our construction, and taking a subsequence if necessary, we may assume 
$U_k \rightarrow U$, $V_k \rightarrow V$ in $ C^1(\overline{\Omega})$ for some $U,V\in C^1(\overline{\Omega})$. Notice that $\norm{U_k} =1$, and hence the RHS in the last equality is uniformly bounded, therefore by Theorem \ref{ABP-MP} $V_k$ is also bounded. 
Passing to limits via stability of viscosity solutions, 
we deduce that $U,V$ is a solution of
\begin{align*}
F_1 [U]+\lambda_1^- \tau_1(x) |V|^{q-1}V  =0 , \quad
F_2 [V]+\lambda_1^-\tau_2(x)|U|^{p-1}U =0\;\; \mbox{ in }  \;\Omega , \quad U=V=0 \mbox{ on }  \;\partial\Omega.
\end{align*}
If instead $\norm{v_k}_{\infty} \to \infty$ we argue similarly, by defining $\theta_k=\norm{v_k}_{\infty}^{q}$.

\vspace{0.15cm}

Anyway, our construction produces $\| U\|_\infty =1$ or $\| V\|_\infty =1$.
Notice that if either $V\equiv 0$ or $U\equiv 0$ in $\Omega$, then $U\equiv V \equiv 0$ in $\Omega$ which produces a contradiction. W.l.g. suppose $\|U\|_{\infty}=1$ and fix $x_1 \in \Omega$ such that $U(x_1) \neq 0$. By Proposition \ref{th4.1 BQeq} we conclude that $U\equiv t\varphi$ and $V\equiv t^p \psi$ for some $t>0$, where $\varphi = \varphi_1^+,\, \psi = \psi_1^+$ if $U(x_1)>0$ (since $\lambda_1^- \geq \lambda_1^+$), while  $\varphi = \varphi_1^-,\, \psi = \psi_1^-$ if $U(x_1)<0$. 
Let us finish the proof by showing that both cases are not admissible.  

First, if $U(x_1)<0$, then $U,V < 0$ in $\Omega$. Using \eqref{sequenceantimax1} we deduce that $U(x_0) = 0$, and $x_0\in \partial \Omega$. However, $DU(x_0) = 0$, in violation of Hopf lemma, and so $U(x_1)<0$ fails to be true.

Finally, if $U(x_1)>0$, then $U,V> 0$ in $\Omega$. Note that $U_k,V_k>0$ in any given compact set $K\subset\Omega$, for $k$ suitably large.
We claim that for $k$ sufficiently large we have $U_k,V_k \geq  0$ in $\Omega$. 
Indeed, we consider $K\subset\Omega$ such that $|\Omega\setminus K|<\varepsilon_0$, where $\varepsilon_0$ is the constant of Theorem \ref{MP small domain}. Thus $U_k,V_k\ge 0$ in $\partial(\Omega\setminus K)$, and so in $\Omega\setminus K$ by Theorem \ref{MP small domain} since $\lambda_1^-(F_1^*(\vartheta)),\,\lambda_1^-(F_2^*(\vartheta))>0$ and $f_1,f_2\le 0$.
Hence we derive a contradiction with Lemma \ref{PDEnonexistence}, from where $U(x_1)>0 $ is also impossible.
\end{proof}

\section{The second eigenvalue problem}\label{section second curve}

This section is dedicated to the proof of Theorem \ref{Th lambda2 Introdu}. Throughout this section, we will assume $F_i$ both satisfy \eqref{SC}--\eqref{H strong}.
We start recalling 
\[ \lambda_2=\lambda_2(F_1,F_2,\Omega) = \inf \{ \lambda> M_1: \, \lambda \text{ is an eigenvalue of \eqref{hyperbola mu=lambda}} \}, \]
with $M_1=\max \{ \lambda_1^+(F_1, F_2),  \lambda_1^-(F_1, F_2) \}$, as in \eqref{def m1,M1}. By Theorem \ref{Th isolation Introdu} we know that $\lambda_2>M_1$.
Note that one may have $\lambda_2=+\infty$, for instance if $F_1,F_2$ are not symmetric, see \cite{Arms2009}. However, when $\lambda_2$ is finite then it is in fact an eigenvalue, as it is shown below. 

\begin{lem}\label{lambda2 is igenvalue}
	If $\lambda_2< +\infty$, then there is a nontrivial solution pair $\varphi_2 , \psi_2\in  E_\varrho$ of
\begin{equation}\label{eq 2}
F_1[\varphi_2]+\lambda_2\tau_1(x)|\psi_2|^{q-1}\psi_2 =0 ,\; F_2[\psi_2]+\lambda_2\tau_2(x)|\varphi_2|^{p-1}\varphi_2 =0\textrm{ in } \Omega \;\;
\varphi_2,\psi_2 = 0 \textrm{ in } \partial \Omega.
\end{equation}
\end{lem}
\begin{proof}
Take a sequence $\lambda_k \to \lambda_2$ of eigenvalues, with corresponding eigenfunctions $u_k,v_k \in E_\varrho$, with $\| u_k \|_{\infty} = 1$, and solving, in the viscosity sense,
\begin{equation}\label{eq 2k}
F_1[u_k]+\lambda_k\tau_1(x)|v_k|^{q-1}v_k =0 ,\; F_2[v_k]+\lambda_k\tau_2(x)|u_k|^{p-1}u_k=0\textrm{ in } \Omega \;\;
u_k, v_k = 0 \textrm{ in } \partial \Omega.
\end{equation}
By $C^{1,\alpha}$ regularity-estimates for scalar equations (Proposition \ref{C1,alpha regularity estimates geral}) we have 
$\| u_k \|_{C^{1,\alpha}(\overline{\Omega})} , \| v_k \|_{C^{1,\alpha}(\overline{\Omega})}\leq C.$
Thus, up to a subsequence, $u_k,v_k$ converge to functions $\varphi_2,\psi_2 \in C^1(\overline{\Omega})$. Since $\| \varphi_2 \|_{\infty}= 1$ we may pass to the limit in \eqref{eq 2k} via stability of viscosity solutions. Then $\varphi_2,\psi_2$ is a nontrivial pair of solutions to the problem \eqref{eq 2}. Observe that $\varphi_2,\psi_2\in E_\varrho$ by hypothesis \eqref{H strong}.
\end{proof}

\begin{proof}[Proof of Theorem \ref{Th lambda2 Introdu}-(i)] From the previous lemma and by scaling as in Section \ref{section scaling}, a second spectral curve $\Lambda_2$ is produced in the first quadrant if $\lambda_2<+\infty$:
\begin{equation}\label{eq:Lambda2}
\textstyle \Lambda_2\, (\lambda) \,=\, (\,\lambda,\, \mu_2(\lambda)\,),  \quad \text{ where }\; \mu_2 (\lambda)= \frac{(\lambda_2 )^{p+1}}{\lambda^p}, \;\; \text{ for all } \lambda>0. 
\end{equation}
Notice that $\Lambda_2$ and the curve originating from $M_1$ cannot intersect. This is a consequence of $M_1<\lambda_2$ together with the definition of the curves, given by \eqref{eq:thetwocurves} and  \eqref{eq:Lambda2}.
\end{proof}
\smallskip

Now we turn to the Dirichlet problem \eqref{Dir} for $\lambda\in ( M_1,\lambda_2)$ in the case $p=q=1$, proving Theorem \ref{Th lambda2 Introdu}-(ii).

We closely follow \cite[\S 5]{Arms2009}, highlighting only the main differences. We define a homotopy between \eqref{Dir lambda mu} and the corresponding Dirichlet problem for the Laplacian with constant weights. 
For each $0 \leq s \leq 1$, define the fully nonlinear operators
\begin{equation}\label{eq:F-homotopy}
F^s_i(x,r,\eta,X) = \beta s\, \mathrm{tr}(X) + (1-s) F_i(x,r,\eta,X), \;\;i=1,2.
\end{equation}
It is easy to verify that $F^s_i$ satisfies \eqref{SC}-\eqref{H strong}. Note that
\begin{center}
$
\lambda^\pm_{1,s}:=\lambda^\pm_1(F_1^s(\tau_1^s),F^s_2(\tau_2^s),\Omega)  < \infty,
$
\end{center}
where $\tau_i^s=s+(1-s)\tau_i$ due to our Proposition \ref{boundedness eig BQeq}. Analogously to \cite[Lemmas 5.4 \& 5.5]{Arms2009}, one sees that the maps $s \mapsto \lambda^+_{1,s}$ and $s \mapsto \lambda^-_{1,s} $ are continuous on $[0,1]$, while the map 
\begin{center}
	$s \mapsto \lambda_{2,s}=\lambda_2 (F^s_1(\tau_1^s),F^s_2(\tau_2^s),\Omega)$ 
\end{center}
is lower semi-continuous on $[0,1]$. We also have the following analog of \cite[Proposition 5.6]{Arms2009}.

\begin{lem} \label{lem M}
Let $M_1 < \lambda < \lambda_2.$ Then there is a continuous function $\mu : [0,1] \to \real$ such that $\mu_0 = \lambda$ and
	$\max\{ \lambda_{1,s}^-, \lambda_{1,s}^+ \} < \mu_s < \lambda_{2,s} $ for all $s\in [0,1].$
Furthermore, for every constant $M>0$, there is a constant $C>0$ such that for any $f_1,f_2 \in L^\varrho(\Omega)$ satisfying $\|f_1 \|_\varrho ,\|f_2\|_\varrho \leq M$, for all $0\leq s \leq 1$, and for any solution $u,v\in C(\overline{\Omega})$ of the Dirichlet problem
\begin{align}\label{Dirs}
\left\{
\begin{array}{rclcc}
F_1^s \,[u]+\mu_s\tau_1^s(x) v  &=&f_1(x) &\mbox{in} & \;\Omega, \\
F_2^s\, [v]+\mu_s\tau_2^s(x)u &=&f_2(x) &\mbox{in} & \;\Omega, \\
u\,\;= \,\;v\;&=& 0 &\mbox{on} & \;\partial\Omega,
\end{array}
\right.
\end{align}
we have the estimate
$	\| u \|_{C^{1,\alpha}(\overline{\Omega})} , \| v \|_{C^{1,\alpha}(\overline{\Omega})}< C( 1+ \max\{\|f_1 \|_\varrho,\|f_2 \|_\varrho\} ).
$
\end{lem}

\begin{proof} The existence of $\mu_s$ follows from the continuity statements that precede the lemma.  As for the $C^{1,\alpha}$ estimates, they follow once we obtain $L^\infty$ bounds. Assuming the conclusion is false, we have the existence of $M>0$ and sequences $0\leq s_k\leq 1$, $f_{1k},f_{2k}\in L^\varrho(\Omega)$ with $\|f_{1k}\|_\varrho, \|f_{2k}\|_{\varrho}\le M$, and $u_k,v_k\in C(\overline \Omega)$ solutions of 
\begin{align*}
\left\{
\begin{array}{rclcc}
F_1^s \,[u_k]+\mu_{s_k}\tau_1^{s_k}(x) v_k  &=&f_{1k}(x) &\mbox{in} & \;\Omega, \\
F_2^s\, [v_k]+\mu_{s_k}\tau_2^{s_k}(x)u_k &=&f_{2k}(x) &\mbox{in} & \;\Omega, \\
u_k\,\;= \,\;v_k\;&=& 0 &\mbox{on} & \;\partial\Omega,
\end{array}
\right.
\end{align*}
such that  (without loss of generality) 
\[
\frac{\|v_k\|_\infty}{1+ \max\{\|f_{1k} \|_\varrho,\|f_{2k} \|_\varrho\}}\leq \frac{\|u_k\|_\infty}{1+ \max\{\|f_{1k} \|_\varrho,\|f_{2k} \|_\varrho\}}\to +\infty \qquad \text{ as } k\to +\infty.
\]
By letting 
\[
\tilde u_k:=\frac{u_k}{\|u_k\|_\infty},\quad  \tilde v_k:=\frac{v_k}{\|u_k\|_\infty},
\]
we have $\|\tilde u_k\|_\infty=1$, $\|\tilde v_k\|_\infty\leq 1$ and 
\[
F_1^s[\tilde u_k]+\mu_{s_k} \tau_1^{s_k} \tilde v_k= \frac{f_{1k}}{\|u_k\|_\infty},\quad F_2^s[\tilde v_k]+\mu_{s_k} \tau_1^{s_k} \tilde u_k= \frac{f_{2k}}{\|u_k\|_\infty} \text{ in } \Omega.
\]
Observing that the right hand sides converge to $0$ in $L^\varrho$, by $C^{1,\alpha}$ estimates we can pass to the limit an  obtain the existence of $\bar s \in [0,1]$ for which $\mu_{\bar s}$ is an eigenvalue, a contradiction. \end{proof}

Let us fix $f_i\in L^\varrho(\Omega)$, and set $E=C^1(\overline{\Omega})^2$. We define a map $\mathcal{A}_{ s}=\mathcal{A}_{f_1,f_2,s}: E\times [0,1] \to E$ by $\mathcal{A}_{ s}(u,v) = (U,V)$,
where $(U,V)\in E$ is the unique strong solution of the Dirichlet problem 
\begin{align}\label{Dir U,V}
\left\{
\begin{array}{rclcc}
F_1^s \,[U]+\mu_s\tau_1^s(x) v  &=&f_1(x) &\mbox{in} & \;\Omega, \\
F_2^s\, [V]+\mu_s\tau_2^s(x)u &=&f_2(x) &\mbox{in} & \;\Omega, \\
U\,\;= \,\;V\;&=& 0 &\mbox{on} & \;\partial\Omega.
\end{array}
\right.
\end{align}
Observe that the map is well defined by Theorem \ref{Dirichlet scalar}. 

\begin{lem}\label{lem:nice-homotopy}
	The map $\mathcal{A}_{s}$ is a homotopy of completely continuous transformations on $E$.
\end{lem}
\begin{proof}
	For each $s \in [0,1]$, whenever $(U,V) = \mathcal{A}_{s}(u,v)$
	\begin{equation}\label{eq:nice-homotopy-est}
\textstyle	\| U \|_{C^{1,\alpha}(\overline{\Omega})} < C( \max_{s\in [0,1]} |\mu_s| \, \| v \|_{L^\varrho(\Omega)} + \| f_1 \|_{L^\varrho(\Omega)}  ) \leq C( 1 + \| v \|_{L^\infty(\Omega)}),
	\end{equation}
and also $\| V \|_{C^{1,\alpha}(\overline{\Omega})} <  C ( 1 + \| u \|_{L^\infty(\Omega)} ).$ Thus the operator $(u,v) \mapsto \mathcal{A}_{ s}(u,v)$ is completely continuous for each fixed $s \in [0,1]$.
Let us show that for each constant $R>0$, the map $(u, v, s) \mapsto \mathcal{A}_{s}(u, v)$ is uniformly continuous on the set 
$B^E_R (0)\times [0,1].$
By contradiction assume there exist $\varepsilon > 0$, and sequences of numbers $s_k, t_k \in [0,1]$ and functions $u_k,v_k\in C^1(\overline{\Omega})$ such that
$| s_k - t_k | \rightarrow 0$, $\|(u_k, v_k) \|_{E} <R$	but
	\begin{equation}\label{nice-homotopy-1}
	\| (U_k,V_k) - (\tilde{U}_k,\tilde{V}_k) \|_{E} \geq \varepsilon,
	\end{equation}
where $(U_k,V_k) = \mathcal{A}_{ s}(u_k,v_k,s_k)$ and $(\tilde{U}_k,\tilde{V}_k) = \mathcal{A}_{ s}(u_k,v_k,t_k)$. By \eqref{eq:nice-homotopy-est} and up to a subsequence, we can find $s\in [0,1]$, functions $(u,v) \in E$ and $(U,V),(\tilde{U},\tilde{V}) \in E$ such that
$s_k \rightarrow s$, $t_k \rightarrow s$, $	v_k \rightarrow v $ uniformly on $ \overline{\Omega}$, $(U_k,V_k) \rightarrow (U,V) $ in $ E$, $(\tilde{U}_k,\tilde{V}_k) \rightarrow (\tilde{U},\tilde{V})$ in $E$. Passing to limits, we deduce that $U,V$ and $\tilde{U}, \tilde{V}$ are both solutions of the problem
\begin{align*}
\left\{
\begin{array}{rclcc}
F_1^s \,[U]+\mu_s\tau_1^s(x) v  &=&f_1(x) &\mbox{in} & \;\Omega, \\
F_2^s\, [V]+\mu_s\tau_2^s(x)u &=&f_2(x) &\mbox{in} & \;\Omega, \\
U\,\;= \,\;V\;&=& 0 &\mbox{on} & \;\partial\Omega.
\end{array}
\right.
\end{align*}
By the uniqueness properties of the operators $F_1^s,F_2^s$ (Theorem \ref{Dirichlet scalar}) one has $U = \tilde{U}$, $V=\tilde{V}$, which contradicts \eqref{nice-homotopy-1}. This completes the proof.
\end{proof}

Now we look at the following operator $\mathcal{B}_{s}=\mathcal{B}_{f_1,f_2,s} : E \to E$ by $\mathcal{B}_{s}(u,v) =(u,v)-  \mathcal{A}_{s}(u,v).$
For $(u,v) \in E$, set $(w,z):=\mathcal{A}_s(u,v)$, then $(U,V) = \mathcal{B}_{s}(u,v)$ is equivalent to $w=u-U$ and $z=v-V$ solving the Dirichlet problem
\begin{align}\label{Dir w,z}
\left\{
\begin{array}{rclcc}
F_1^s \,[w]+\mu_s\tau_1^s(x) v  &=&f_1(x) &\mbox{in} & \;\Omega, \\
F_2^s\, [z]+\mu_s\tau_2^s(x)u &=&f_2(x) &\mbox{in} & \;\Omega, \\
w\,\;= \,\;z\;&=& 0 &\mbox{on} & \;\partial\Omega.
\end{array}
\right.
\end{align}
Our goal is to show the existence of a solution $(u,v) \in E$ of the equation $\mathcal{B}_{0}(u,v) = 0,$ where $0 = (0,0) \in E$. This will be accomplished by showing that $\deg( \mathcal{B}_{0},B_R^E,0) \neq 0$ for some ball $B_R^E \subseteq E$, and then appealing Leray-Schauder degree theory.

\begin{lem}\label{lem:degree-B-1}
Let $R:=1+C( 1+ \max\{\|f_1 \|_\varrho,\|f_2 \|_\varrho\} )$, for $C$ as in Lemma \ref{lem M}. Then \begin{center}
	$\deg ( \mathcal{B}_{1} , B_R^E(0) , 0) = \pm 1.$
\end{center}
\end{lem}
\begin{proof}
We will show that $\mathcal{B}_1=I-\mathcal{A}_{1}$\, is bijective. This is equivalent to prove that for any $(U,V) \in E$ there exists a unique solution $(w, z)$ to the following
\[
\beta \Delta w+\mu_1 z + \mu_1 V  =f_1, \quad
\beta \Delta z+\mu_1 w + \mu_1 U = f_2  \;\textrm{ in } \;\Omega, \quad 
w\,= \,z= 0 \;\textrm{ on }  \;\partial\Omega. 
\]
namely to the following
\begin{equation}\label{sys inject}
\beta \Delta w+\mu_1 z=g_1, \quad
\beta \Delta z+\mu_1 w  = g_2  \;\textrm{ in } \;\Omega, \quad 
w\,= \,z= 0 \;\textrm{ on }  \;\partial\Omega,
\end{equation}
for any $g_1, g_2$ functions in $L^\varrho(\Omega)$. Recall that $\mu_1$ satisfies $\lambda_1(\beta \Delta)=\lambda_1(\beta \Delta, \beta \Delta) < \mu_1 < \lambda_2(\beta \Delta, \beta \Delta)=\lambda_2(\beta \Delta)$. We first consider the case $g_1, g_2 \in C_c^{\infty}(\Omega)$. 
Take a basis for $L^2$ given by positive eigenvectors $\varphi_i$ of the operator $\beta \Delta$ with related eigenvalues $\lambda_i$, for $i=1, \dots, N$. Then we can write
\[ \textstyle w= \sum_i a_i \varphi_i, \quad z = \sum_i b_i \varphi_i \]
for some coefficients $a_i, b_i$. Then, if $w, z$ satisfy the system \eqref{sys inject}, we conclude that 
\[ \textstyle - \sum_i \lambda_i a_i \varphi_i + \mu_1 \sum_i b_i \varphi_i = \sum_i \langle f_1, \varphi_i \rangle \varphi_i \]
and similarly
\[ \textstyle - \sum_i \lambda_i b_i \varphi_i + \mu_1 \sum_i a_i \varphi_i = \sum_i \langle f_2, \varphi_i \rangle \varphi_i. \]
Therefore, by the first equation for any $i$ we need 
\[\textstyle  b_i= \frac 1{\mu_1} \langle f_1, \varphi_i \rangle + \frac {\lambda_i}{\mu_1} a_i \]
 and putting this information into the second equation
 \[\textstyle -\frac{\lambda_i}{\mu_1} \langle f_1, \varphi_i \rangle + \frac {-\lambda_i^2}{\mu_1} a_i + \mu_1a_i = \langle f_2, \varphi_i \rangle \]
which implies
\[\textstyle a_i= \frac{1}{\mu_1^2-\lambda_i^2}\, \{ \,\mu_1\, \langle f_2, \varphi_i \rangle + \lambda_i\, \langle f_1, \varphi_i \rangle\,\}  \]
since $\mu_1 \ne \lambda_i$. This proves that \eqref{sys inject} has a (unique) solution if $g_1, g_2 \in C_c^{\infty}(\Omega)$. The general case follows by approximation arguments, by recalling that $C_c^\infty(\Omega)$ is dense in $L^\varrho(\Omega)$. 

The fact that the solution is unique can be proved considering the problem
\[ \beta\Delta w+\mu_1 z=0, \quad
\beta \Delta z+\mu_1 w  = 0  \;\textrm{ in } \;\Omega, \quad 
w\,= \,z= 0 \;\textrm{ on }  \;\partial\Omega,
\] 
and recalling that $\lambda_1(\beta\Delta,\beta \Delta) < \mu_1 < \lambda_2(\beta\Delta, \beta\Delta)$.
\end{proof}

\begin{proof}[Conclusion of the proof of Theorem \ref{Th lambda2 Introdu}-(ii)]
Taking, as before $R:=1+C( 1+ \max\{\|f_1 \|_\varrho,\|f_2 \|_\varrho\} )$, we have by homotopy invariance of the degree that
$\deg ( \mathcal{B}_{0}, B_R^E(0) , 0) = \pm 1.$ Therefore there exists $(u,v)\in B_R^E(0)$ such that  $\mathcal{A}_{0}(u,v)=(u,v)$,
and this gives us the desired existence result. Notice that $0 \not\in \mathcal B_s(\partial B_R^E(0))$ due to the a priori estimates we get in Lemma \ref{eq:nice-homotopy-est}.  
\end{proof}

\section{The signed Dirichlet problem}\label{section pq<1}

In this section we draw some attention to the
Dirichlet problem \eqref{Dir lambda mu}, or equivalently to \eqref{hyperbola mu=lambda}, see Section \ref{section scaling}, in the case the functions $f_1,f_2$ have the ``good" sign.

When $pq=1$ we have seen that the Dirichlet problem \eqref{Dir} is solvable for any $\lambda <m_1$, with $m_1$ as in \eqref{def m1,M1}, independently of sign on $f_1,f_2$. Now we turn to the case $\lambda\in (m_1,M_1)$.

Moreover, a variation of such argument applies to show that the Dirichlet problem is uniquely solvable among positive viscosity solutions in the sublinear regime $pq<1$.
We start with the latter, by proving Theorem \ref{sublinear}.

\begin{proof}[Proof of Theorem \ref{sublinear}]
We borrow some ideas from \cite[Section 5]{E} and \cite[Theorem 4.1]{Montenegro}, based on a Krasnoselskii \cite{Krasnoselskii} type argument.

	\textit{Step 1) Existence:}
	Let us consider the eigenvalue problem obtained in Theorem \ref{Th1 introduction} for the operators $(F_1)_*$ and $(F_2)_*$, i.e.\ we take the strong solution pair $\varphi, \psi\in W^{2,p}_{\mathrm{loc}}(\Omega)$ of
	\begin{align}\label{eigenvalue L-}
	\left\{
	\begin{array}{rclcc}
	(F_1)_* [\varphi] + \lambda_1^+ \tau_1(x)\psi^{1/p}=0 &\mbox{in} & \;\Omega \\
	(F_2)_* [\psi] + \lambda_1^+ \tau_2(x) \varphi^p=0 &\mbox{in} & \;\Omega \\
	\varphi \, , \;\, \psi\, &>& 0 &\mbox{in} & \;\Omega \\
	\varphi\;\,= \,\;\psi\, &=& 0 &\mbox{on} & \;\partial\Omega .
	\end{array}
	\right.
	\end{align}
	Notice that the operators $(F_i)_*$ are homogeneous if $F_i$ are homogenous, namely satisfy \eqref{H homogeneity}. Moreover, $\lambda_1^+(((F_i)_*)^*)=\lambda_1^-(((F_i)_*)_*)=\lambda_1^-((F_i)_*) = \lambda_1^+(F_i^*)>0$, hence also \eqref{H lambdai>0} is true. Finally, 
	$(F_i)_*$ satisfy \eqref{H strong} as they are concave operators with $(F_i)_*(\cdot, 0, 0, 0)=0$, see Lemma \ref{lema convex}. 
	\smallskip
For the sake of convenience we now look at the Dirichlet problem
\begin{align}\label{Dir lambda,lambda}
\left\{
\begin{array}{rclcc}
F_1 [u]+\tau_1(x) |v|^{q-1}v  &=&f_1(x) &\mbox{in} & \;\Omega, \\
F_2 [v]+\tau_2(x)|u|^{p-1}u &=&f_2(x) &\mbox{in} & \;\Omega, \\
u\,\;= \,\;v\;&=& 0 &\mbox{on} & \;\partial\Omega.
\end{array}
\right.
\end{align}
Observe that \eqref{Dir lambda mu} and \eqref{Dir lambda,lambda} are equivalent up to scaling, since $pq<1$, see Section \ref{section scaling}.
	
	We first construct a subsolution pair $(u_0,v_0)=(\epsilon\varphi, \; \epsilon^k \psi)$  to \eqref{Dir lambda,lambda}, say for $f_i \in L^\varrho (\Omega)$ with $f_i\le 0$ a.e.\ in $\Omega$. We use \eqref{eigenvalue L-} and pick up some positive constants $\varepsilon, k$ to be chosen. 
	For this, we write in the a.e.\ sense,
	\begin{align}\label{eqF1}
	-F_1[u_0] \leq -(F_1)_*[\varepsilon \varphi]=  \lambda_1^+ \epsilon \tau_1 \psi ^{1/p} =\lambda_1^+  \tau_1 (\epsilon^p \psi)^{1/p} \leq \tau_1 v_0^q -f_1(x),
	\end{align}
	and \vspace{-0.5cm}
	\begin{align}\label{eqF2}
	-F_2[v_0] \leq -(F_2)_*[\varepsilon^k \psi]= \lambda_1^+ \epsilon^k  \tau_2 \varphi^p =  \lambda_1^+\tau_2 (\epsilon^{k/p} \varphi)^p \leq \tau_2 u_0^p -f_2(x).
	\end{align}
	The choice of $\epsilon $ is made in order to have $\lambda_1^+ \epsilon^{1-k q} \leq \|\psi\|_{\infty}^{(pq-1)/p}$ in \eqref{eqF1}, for some $1- k q > 0$.
	In addition, for \eqref{eqF2} we require $\lambda_1^+ \epsilon^{k - p} \leq 1$, with $k>p$. Then, by diminishing $\epsilon$ if necessary, it is enough to choose $k \in (p,{1}/{q})$.
	
	Next, for each $n\geq 0$ we define recursively $(u_{n+1},v_{n+1})$ as the unique strong solution of 
	$$
	-F_1[u_{n+1}] = \tau_1(x) v_n^q -f_1(x)\, ,\;\; -F_2 [v_{n+1}] = \tau_2(x) u_n^q -f_2(x)\;\textrm{ in } \Omega\, , \quad u_{n+1}=v_{n+1}=0\;\, \textrm{ on }\partial\Omega.
	$$
Note that by \eqref{H lambdai>0} and ABP-mP for scalar equations (Theorem \ref{ABP-MP}, since $\lambda_1^- (F_i^*)\ge \lambda_1^+(F_i^*)>0$) one has $u_{n+1}$, $v_{n+1}\geq 0$ in $\Omega$, for all $n\geq 0$.
	
Now we infer that the sequences $(u_n)$ and $(v_n)$ are monotone nondecreasing. 
	This is accomplished via a monotone iterations technique.
	Indeed,
	\[
	F_1^*[u_{n} - u_{n+1}] \geq -F_1[u_{n+1}] +F_1[u_n] = \tau_1(x)(v_n^q - v_{n-1}^q) , \quad F_2^* [v_{n} - v_{n+1}]  \geq \tau_2(x)(v_n^p - v_{n-1}^p),
	\]
	for all $n \geq 1$. For $n=0$ one has
	\[
	F_1^* [u_0 - u_1] \geq -F_1[u_1] +F_1[u_0] \geq \tau_1(x)(v_0^q - v_0^q) =0, \quad 
	F_2^* [v_0 - v_1]  \geq  \tau_2(x)(u_0^p - u_0^p)  =0 .
	\]
	This implies $u_1 \geq u_0$ and $v_1 \geq v_0$ in $\Omega$, by \eqref{H lambdai>0} and ABP-MP for scalar equations (Theorem \ref{ABP-MP}), as desired.

Next we claim that
	$(u_n,\; v_n)$ is bounded in $L^\infty \times L^\infty$.
	To see this we use the blow up method. Suppose $\alpha_n = \|u_n\|_\infty \to \infty$ in order to get a contradiction. 
	We know that $(\alpha_n)$ is a nondecreasing sequence.
	Hence we define the rescaled pair
	\begin{center}
		$u_n = \alpha_n U_n$\quad\; and\;\quad $ v_n = \alpha_n^{Q} V_n$\,,\qquad  where\; $Q=\frac{p+1}{q+1}$.
	\end{center}
	
	Properties of $F_1^*$, $F_2^*$ and \eqref{SC} for $F_1$, $F_2$ give us
	$$
	-F_1^*[U_{n+1}] 
	\leq  -\frac{F_1 [u_{n+1}]}{\alpha_{n+1}}  = \frac{\tau_1(x) v_n^q}{\alpha_{n+1}} -\frac{f_1(x)}{\alpha_{n+1}}  \leq \tau_1(x) \alpha_n^{\frac{pq -1}{q+1}} V_n^q - f_1^n(x),
	$$
	and analogously,
	$$
	-F_2^* [V_{n+1}]\leq  -\frac{F_2 [v_{n+1}]}{\alpha_{n+1}^{Q}} 
	= \frac{\tau_2(x) u_n^p}{\alpha_{n+1}^{Q}}-\frac{f_2(x)}{\alpha_{n+1}^{Q}}  \leq \tau_2(x) \alpha_n^{\frac{pq -1}{q+1}} U_n^p -f_2^n(x),
	$$
where $f_1^n(x)={f_1(x)}\alpha_{n+1}^{-1}$ and $f_2^n(x)={f_2(x)}{\alpha_{n+1}^{-Q}}$ for all $n\ge 1$.
	Since the last RHS converges to zero in $L^\varrho$, then by \eqref{H lambdai>0} and ABP-MP we obtain that $V_{n+1} \to 0$. So, again ABP-MP for the first equation  yields $U_{n+2} \to 0$, which derives a contradiction.

	Therefore, since the sequences $u_n^q$ and $v_n^p $ are uniformly bounded from above and from below, by $C^\alpha$ regularity estimates $u_n,v_n\in C^\alpha (\overline{\Omega})$ with $\|u_n\|_{C^\alpha (\overline{\Omega})}, \|v_n\|_{C^\alpha(\overline{\Omega})}\leq C$. Thus compact inclusion and a standard stability argument of viscosity solutions lead to solution pair $u,v$ of \eqref{LE}. Here $u,v>0$ in $\Omega$ by the uniformly bound from below via $u_0,v_0$.
	
	\medskip
	
	\textit{Step 2) Uniqueness:}	

	Let $(u_1,v_1)$ and $(u_2,v_2)$ be two positive pairs of viscosity solutions to \eqref{LE}. Since $f_i \le 0$ we can use Hopf lemma to conclude ${\partial_\nu u_i} < 0$ and ${\partial}_\nu v_i < 0$ on $\partial \Omega$, $i=1,2$, where $\nu$ is the exterior unit normal.
	Thus we may define (the nonempty set)
	\begin{center}
		$ S=\{ s>0 \, : \; u_1 > s^{\frac{p+1}{p}} u_2 \, , \; v_1 > s^{\frac{q+1}{q}} v_2 \textrm{\, in }\Omega\, \} $, \quad $s_*=\sup S$.
	\end{center}
	Here $s_*<+\infty$ since $u_2$ and $v_2$ are positive. 
	So, up to exchanging the roles of $(u_1, v_1)$ and $( u_2, v_2)$ if necessary, say that $s_*\le 1$.
	Let us look at the nonnegative functions $w,z$ given by
	\begin{center}
		$w:=u_1 -  s_*^{\frac{p+1}{p}} u_2$ \;\; and \;\; $z:=v_1 - s_*^{\frac{q+1}{q}}\, v_2$ \quad in $\Omega$,
	\end{center}
	which satisfy in $\Omega$, in the viscosity sense,
	\begin{align*}
	-(F_1)_*[w] &\ge - F_1[u_1] + s_*^{\frac{p+1}{p}} F_1[u_2] 
	=\tau_1(x) v_1^q -f_1(x)- \tau_1 (x)s_*^{\frac{p+1}{p}}  v_2^q   +s_*^{\frac{p+1}{p}} f_1(x)\\
	&\ge 
	\tau_1(x)\, v_1^q \,( 1- s_*^{\frac{1-pq}{p}}  ) -f_1(x)\,(1-s_*^{\frac{1-pq}{p}})
	\ge 0,
	\end{align*}
	and similarly $
	-(F_2)_*[z]\, \ge 0.
	$
	Whence SMP for scalar equations and the strongly coupling of the Lane-Emden type system imply either $w,z>0$ or $w,z\equiv 0$ in $\Omega$. 
	
	Suppose on the contrary that $w,z>0$ in $\Omega$.
	Now we may repeat the preceding argument with the pair $(u_2,v_2)$ replaced by $(w,z)$ in order to conclude the existence of some small $\varepsilon>0$ such that
	$u_1 > (s_*+\varepsilon) u_2$ in $\Omega$. But this contradicts the definition of $s_ *$ as the supremum of $S$. 
	
	Therefore $w,z\equiv 0$ in $\Omega$. 
	To finish we infer that $s^*=1$; otherwise the strict inequalities above and SMP would be in force to produce the positivity of $w$ and $z$. So one concludes $u_1\equiv u_2$ and $v_1\equiv v_2$ in $\Omega$, as desired. 
\end{proof}

Next we return to the regime $pq=1$.

\begin{proof}[Proof of Theorem \ref{ThDir solvability intro} (ii), (iii)]
We prove $(ii)$; the case $(iii)$ is similar.
Since we have $pq=1$, and $f_i\equiv 0$ is an eigenvalue problem which we have already studied, is enough to consider either $f_1\not\equiv 0$ or $f_2\not\equiv 0$. Therefore, we are going to obtain positive solutions, by SMP and strong coupling of the system.
In particular, the uniqueness can be carried out as in Step 2 of the proof of Theorem \ref{sublinear}.
Let us show the existence assertion in $(ii)$.
	
	\smallskip
	
	Step 1) Continuous and compactly supported $f_i$, with uniformly positive weights.
	
Say $\tau_i\ge a_i>0$ a.e.\ in $\Omega$, and $f_i\ge -b_i$ in $K_i=\mathrm{supp}(f_i)\subset\subset\Omega$, for some $b_i>0$, $i=1,2$. In this case we choose a large constant $A > 0$ such that
	\begin{center}
		$ b_1 \leq A a_1 (\lambda^+_1-\lambda) (\psi^+_1)^q$\, a.e.\ in $K_1$, \;\; $ b_2 \leq A^p a_2 (\lambda^+_1-\lambda) (\varphi^+_1)^p$\, a.e.\ in $K_2$.
	\end{center}
	Then the pairs $u^* = A\varphi^+_1$,  $v^* = A^{1/q}\,\psi^+_1$  and $u_* \equiv 0$, $v_*\equiv 0$\, satisfy
	\begin{align*}
	F_1[u_*] + \lambda\tau_1(x)|v_*|^{q-1}v_* \geq f_1(x) \geq F_1[u^*] + \lambda\tau_1(x)|v^*|^{q-1}v^* \;\;&\textrm{ a.e.\ in\, $\Omega$},
	\\
	F_2[v_*]+\lambda\tau_2(x) |u_*|^{p-1}u_* \geq f_2(x) \geq F_2[v^*] +\lambda\tau_2(x) |u^*|^{p-1}u^* \;\;&\textrm{ a.e.\ in\, $\Omega$},
	\end{align*}
	with $u^* = u_* = 0$ and $v^* = v_* = 0$ on $\partial \Omega$. 
Thus we apply the same monotone iterations technique as in Step 1 in the proof of Theorem \ref{sublinear}, from the supersolution case instead of the subsolution one. This time is even a bit simpler since the supersolution already gives us a uniform bound from above on the uniform norm of the iterated solutions produced by the method.
	
	\medskip
	
	Step 2) General nonnegative $f_i\in L^\varrho(\Omega)$, but still uniformly positive weights.
	
	We take sequences $(f_1^k) , (f_2^k) \in L^\varrho (\Omega)$ of continuous nonpositive functions with compact support in $\Omega$ such that $f_1^k \to f_1$, $f_2^k \to f_2$  in $L^\varrho(\Omega)$. By Step~1, let $u_k,v_k \geq 0$ solving
	\begin{align*}
	F_1 [u_k]+\lambda\tau_1(x) v_k^{q} =f_1^k(x) , \quad
	F_2 [v_k]+\lambda\tau_2(x)u_k^{p} =f_2^k(x) \quad\textrm{ in }  \,\Omega, 
	\end{align*}
	with $u_k=v_k=0$ on $\partial\Omega$. We infer that
	\begin{equation}\label{uk, vk le C}
	\| u_k \|_{L^\infty(\Omega)} , \; \| v_k \|_{L^\infty(\Omega)}\leq C \textrm{ \; for all $k$.}
	\end{equation}
	Otherwise, assume for instance that $\theta_k=\| u_k \|_{L^\infty(\Omega)} \to \infty$ as $k\to \infty$, and set $u_k = \theta_k U_k$ and $v_k = \theta_k^{p}\, V_k$. Notice that the pair $(U_k,V_k)$ satisfies the equation
	\begin{align*}
	\textstyle F_1 [U_k]+\lambda\tau_1(x) V_k^{q} =\frac{f_1^k(x)}{\theta_k} , \quad
	F_2 [V_k]+\lambda\tau_2(x)U_k^{p} =\frac{f_2^k(x)}{\theta_k^p} \quad\textrm{ in }  \,\Omega, 
	\end{align*}
	and $\|U_k\|_\infty= 1$ for all $k$. Hence $\|V_k\|_\infty\leq C$ by Theorem \ref{ABP-MP} for scalar equations.
	Thus, by $C^{1,\alpha}$ estimates one gets that $\|U_k\|_{C^{1,\alpha}(\overline{\Omega})},\|V_k\|_{C^{1,\alpha}(\overline{\Omega})}\le C$.
	Extracting a subsequence if necessary we may assume $U_k\to U$ and $V_k\to V$ in $C^1(\overline{\Omega})$. Passing to limits, through stability of viscosity solutions we end up with $U,V\ge 0$ in $\Omega$ satisfying
	\begin{align*}
	\textstyle F_1 [U]+\lambda\tau_1(x) V^{q} =0 , \quad
	F_2 [V]+\lambda\tau_2(x)U^{p} =0\quad\textrm{ in }  \,\Omega, 
	\end{align*}
	with $U,V = 0$ on $\partial \Omega$ and $\|U\|_\infty=1$. Since $F_1 [U]\leq 0$, by SMP for scalar equations we have $U>0$ in $\Omega$. Whence $F_2[V]=-\lambda \tau_2(x)U^p \lneqq  0$, from which also $V>0$ in $\Omega$. By Corollary \ref{cor uniqueness} one derives $U=t\phi_1^+$ and $V=t^q \psi^q$ in $\Omega$, which contradicts the fact that $U,V$ satisfy
	\begin{align*}
	\textstyle F_1 [U]+\lambda_1^+\tau_1(x) V^{q} \gneqq 0 , \qquad
	F_2 [V]+\lambda_1^+\tau_2(x)U^{p} \gneqq 0\;\;\textrm{ in }  \,\Omega,
	\end{align*}
	since $\lambda<\lambda_1^+$.
	If instead $\norm{v_k}_{\infty} \to \infty$ as $k\to \infty$ the argument is analogous, by taking  $\theta_k=\norm{v_k}_{\infty}^{q}$.
	Therefore one proves \eqref{uk, vk le C}.
	Thus, again by $C^{1,\alpha}$ estimates, compact inclusion, and stability of viscosity solutions one finds a nonnegative limit solution pair $u,v\in C^{1,\alpha}(\overline{\Omega})$ of \eqref{Dir lambda mu}. 
	
	\medskip
	
	Step 3) General weights $ \tau_1, \tau_2\in L^\varrho_+ (\Omega)$ and general nonpositive $f\in L^\varrho (\Omega)$.
	
	This case is very similar to the preceding step, by arguing via uniform bounds and stability arguments. It is enough to pick up sequences of uniformly positive weights $\tau^k_1,\tau^k_2$ such that $\tau_1^k\to \tau_1$ and $\tau_2^k \to \tau_2$ in $L^\varrho (\Omega)$. By Step 2, we then take $u_k,v_k \geq 0$ solving
	\begin{align*}
	F_1 [u_k]+\lambda\tau_1^k(x) v_k^{q} =f_1(x) , \quad
	F_2 [v_k]+\lambda\tau_2^k(x)u_k^{p} =f_2(x) \quad\textrm{ in }  \,\Omega, 
	\end{align*}
	with $u_k=v_k=0$ on $\partial\Omega$.
	Again one produces \eqref{uk, vk le C} by arguing exactly as in Step 2. Thus, $C^{1,\alpha}$ estimates, compact inclusion, and stability of viscosity solutions conclude the proof.
\end{proof}

\section{An application: Isaac's operators}\label{appendix}

In this section we prove that Isaac's operators in the form \eqref{Isaacs} are examples for which all our results are new even in the scalar case. In order to do so, it is enough to prove that they satisfy \eqref{H strong}.
We recall that the $W^{2,\varrho}$ regularity results in \cite{regularidade} were extended in \cite[Theorem 5.2]{Swiech2020} to the context of unbounded coefficients and superlinear gradient growth, but only for convex or concave operators. It is known that the same proof there works if the corresponding pure second order operator $F(x,0,0,X)$ enjoys $C^{1,1}$ regularity estimates, see \cite[Remark 4.4]{Winter} for instance. 

The pure Isaac's operator associated to \eqref{Isaacs} is again in the form \eqref{Isaacs} with $\gamma_{s,t},\vartheta_{s,t}\equiv 0$ for all $s,t$.
It is worth mentioning that $C^{1,1}$ regularity estimates do not hold in general for these operators, see \cite{NVadv2011}. On the other hand, in \cite{W2punderW2p} the authors weakened the $C^{1,1}$ hypothesis to a $W^{2,p}$ one. More recently in \cite{Eaihpan2019}, regularity is proved for Isaac's operators with bounded drifts. 
Here we extend the preceding results to unbounded weights.
This is of independent interest in view of applications \cite{BFQjfaLL,  BFQsiam2010} to more general models driven by unbounded data.

As in \cite{Eaihpan2019}, we assume that there exists $\bar{A}_t$ satisfying:
\begin{align}\label{HS}
| A_{s,t}(x) - \bar{A}_t(x) |\le \epsilon_1 \quad\textrm{ uniformly in }\, x,s,t
\end{align}
where $\epsilon_1$ is the number of condition $A_2$ in \cite{Eaihpan2019}, 
with $	\bar{A}_t\in C(\overline{\Omega})$ unifomly in $t\in \n$. In this case the corresponding homogeneous Belmann operators generated by $\bar{A}_t$ have $W^{2,q}$ regularity estimates for $q>\varrho>n$, see \cite{CafCab}.

\begin{prop}[$W^{2,\varrho}$ regularity estimates for Isaac's operators] \label{W2,p reg Isaacs}
Let $\Omega\subset\rN$ be a bounded $C^{1,1}$ domain, $f\in L^\varrho(\Omega)$, $\varrho>N$, and assume \eqref{HS}. Then any viscosity solution $u\in C(\Omega)$ of 
\begin{center}
	$I[u] =f(x)$ in $\Omega$, \;\; where	$I[u]=\inf_{s\in \real}\sup_{t\in \real } L_{s,t} $ or $I[u]=\sup_{s\in \real}\inf_{t\in \real } L_{s,t} $ 
\end{center} 
where $L_{s,t}$ satisfies \eqref{Isaacs Ls,t}, belongs to $W^{2,\varrho}_{\mathrm{loc}}(\Omega)$ and satisfies the estimate
\begin{align*}
	\|u\|_{W^{2,\varrho}(\Omega^\prime)} \leq C\, \{ \| u\|_{L^{\infty} (\Omega)} + \|f \|_{L^p (\Omega)}   \},\;\; \textrm{ for all }\Omega^\prime \subset\subset \Omega,
	\end{align*}
	where $C$ depends only on $N,\varrho,\alpha,\beta$, $\| \gamma \|_{L^\varrho (\Omega)}$, $\| \vartheta \|_{L^\varrho (\Omega)}$, $\Omega^\prime$, $\mathrm{dist}(\Omega^\prime,\partial\Omega)$, and $\mathrm{diam} (\Omega)$.
\smallskip
	
If in addition $u=\psi$ on $\partial\Omega$, for some $\psi\in W^{2,\varrho}(\Omega)$, then $u\in W^{2,\varrho}(\Omega)$ and
\begin{align*}
\|u\|_{W^{2,\varrho}(\Omega)} \leq C\, \{ \| u\|_{L^{\infty} (\Omega)} + \|f \|_{L^p (\Omega)}  +\|\psi\|_{W^{2,\varrho}(\Omega)} \},
\end{align*}
where $C$ depends only on $N,\varrho,\alpha, \beta$, $\| \gamma \|_{L^\varrho (\Omega)}$, $\| \vartheta \|_{L^\varrho (\Omega)}$, $\partial\Omega$, and $\mathrm{diam} (\Omega)$. 
\end{prop}

\begin{proof}
To fix the ideas we consider the first expression for $I$; the second case is identical. 
Let $\gamma_k,\vartheta_k\in L^\infty_+ (\Omega)$, $f_k\in C(\Omega)$, be such that $f_k\rightarrow f$, $\gamma_k\rightarrow \gamma$ and $\vartheta_k\rightarrow \vartheta$ in $L^\varrho(\Omega)$, with $\gamma_k\le \gamma$, $\vartheta_k\le \vartheta$, and $|f_k|\le |f|$. 
Up to a subsequence, one can choose $\ell_k =\ell -\varepsilon_k>0$ for some $\varepsilon_k\in (0,\ell)$, $\varepsilon_k\to 0$, such that \begin{center} $\ell_k<\lambda_1^+(\mathcal{L}^+_{k}(\vartheta_k),\Omega)$,\;\; where \, $\mathcal{L}_{k}^+[w]=\mathcal{M}^+(D^2w)+\gamma_k(x)|Dw|$ 
\end{center}
and $\ell$ as in \eqref{Isaacs Ls,t}. 
Note that $u\in C^{1,\alpha_0}_{\mathrm{loc}}(\Omega)$ by Proposition \ref{C1,alpha regularity estimates geral}.
Let $u_k \in C^{1,\alpha_1}(\overline{B}_{\varrho})$ be a viscosity solution of
	\begin{align}\label{app uk}
	I_k[u_k]= f_k(x) \;\textrm{ in } B_{\rho}, \quad
	u_k = u \;\textrm{ on }  \partial B_{\rho},
	\end{align}
given by Proposition \ref{solv scalar}, where $B_{\rho}$ is centered at $x_0\in \Omega$, and \begin{center}
	$I_k[w]:=\inf_{s\in \real}\sup_{t\in \real} L_{s,t}$, \; where $L_{s,t}[w]=\mathrm{tr}(A_{s,t}(x)D^2w)+\gamma_{s,t} (x)|Dw|+\ell_k \vartheta_{s,t} (x)w$,
\end{center}  with $A_{s,t}$ as in \eqref{Isaacs Ls,t}, but now $|\gamma_{s,t}|\le \gamma_k$ and $|\vartheta_{s,t}|\le \vartheta_k$ for all $s,t\in\real$. 

\smallskip

Note that \cite[Corollary 1.6]{Swiech} implies that $u_k$ is a viscosity solution of 
\begin{center}
	$\inf_{s\in \real}\sup_{t\in \real} \mathrm{tr}(A_{s,t}(x)D^2 u_k)=g_k(x)$, \; where $|g_k(x)|\le |f|+ \gamma(x)|Du_k|+ \vartheta(x)|u_k|\in L^\varrho(B_{\rho})$.
\end{center}
By \cite[Theorem 1.1]{Eaihpan2019} one has $u_k\in W^{2,\varrho}_{\mathrm{loc}}(B_{\rho})$, see also \cite{PMGisaacs}. Now by the second part of Proposition \ref{solv scalar}, we know that $u_k$ is the unique viscosity solution of \eqref{app uk}.

By the generalized Nagumo's lemma in \cite[Lemma 4.4]{regularidade} one gets 
	\begin{align} \label{ukW2,p}
	\|u_k\|_{W^{2,\varrho}(B_{r})} \leq C_k \, \{ \| u_k \|_{L^\infty(B_\rho)}  + \|f \|_{L^\varrho (B_\rho)}  \}, \;\textrm{ for all }r<\rho,
	\end{align}
where $C_k$ remains bounded, since $\gamma_k$ and $\vartheta_k$ are bounded in $L^\varrho(B_\rho)$.
Moreover, since \begin{center}
$\ell_k <\lambda_1^+(\mathcal{L}_{k}^+(\vartheta_k),\Omega)\le \lambda_1^+(\mathcal{L}_{k}^+(\vartheta_k),B_\rho)\le \lambda_1^-(\mathcal{L}_{k}^+(\vartheta_k),B_\rho)$, 
\end{center} 
then one may apply ABP-MP and ABP-mP in Theorem \ref{ABP-MP} to obtain
	$\|u_k\|_{L^\infty (B_\rho)}\leq \|u\|_{L^\infty (\partial B_\rho)} + C\, \|f\|_{L^\varrho(B_\rho)}$; again the constant does not depend on $k$.
This and \eqref{ukW2,p} yield
	$\|u_k\|_{W^{2,\varrho}(B_{r})} \leq C$. Hence there exists $v\in C^{1}(\overline{B}_r)$ such that $u_k\rightarrow v$ in $C^{1}(\overline{B}_r)$, for all $r<\rho$. Note that $\|u_k\|_{C^1(\overline{B}_\rho)}\le C$ by global $C^{1,\alpha_1}$ estimates in Proposition \ref{C1,alpha regularity estimates geral}.
Thus, Proposition \ref{stability} implies that $v$ is a viscosity solution of
\begin{align}\label{app v stability}
I[v]= f(x) \;\textrm{ in } B_{\rho}\, , \quad
v = u \;\textrm{ on }  \partial B_{\rho},
\end{align}
Now, since $W^{2,\varrho}(B_{r})$ is reflexive, there exists $\tilde{v}\in W^{2,\varrho}(B_r)$ such that $u_k$ converges weakly to $\tilde{v}$. By uniqueness of the limit, $\tilde{v}=v$ a.e.\ in $B_r$, and so $v\in W^{2,\varrho}(B_r)$, for all $r<\rho$.

Now, since $u$ is already a viscosity solution of \eqref{app v stability}, by using again the uniqueness assertion in Proposition \ref{solv scalar} one gets that $u\in W^{2,\varrho}_{\mathrm{loc}}(B_\rho)$.
Since the ball $B_\rho$ is arbitrary, and in each ball the solution is unique, a covering argument produces $u\in W^{2,\varrho}_{\mathrm{loc}}(\Omega)$.

In the case of global regularity the argument is simpler, by taking $\Omega$ instead of $B_\rho$ in \eqref{app uk}.
\end{proof}

\textbf{{Acknowledgments.}} We would like to thank Prof.\ A.\ {\'S}wi{\c{e}}ch for his valuable comments regarding our results on Isaac's operators, and Prof.\ I.\ Biridelli for kindly bringing reference \cite{BMS99} to our attention.

This project started during a visit of the first and second authors to the University of Lisbon. They appreciate the invitation, hospitality, and support provided by the Faculty of Sciences and FCT/Portugal.

\end{document}